\documentclass[11pt,reqno]{amsart}
\usepackage[T1]{fontenc}
\usepackage{lmodern}
\usepackage[letterpaper,textwidth=6.15in,textheight=8.6in,centering]{geometry}
\usepackage{amsmath,amssymb,mathtools,bm,mathrsfs}
\usepackage{microtype}
\usepackage{booktabs,array,needspace,enumitem,longtable}
\usepackage[dvipsnames]{xcolor}
\usepackage[colorlinks=true,linkcolor=MidnightBlue,citecolor=MidnightBlue,urlcolor=MidnightBlue]{hyperref}
\usepackage[nameinlink,noabbrev,capitalise]{cleveref}
\hypersetup{
 pdftitle={Universal Spacelikeness Estimates and Liouville Rigidity for Lorentzian sigma-k
Curvature Equations},
 pdfauthor={Shujun Shi and Yuzhou Zhang},
 pdfsubject={Entire spacelike graphs, universal spacelikeness estimates, and Liouville rigidity through the critical endpoint},
 pdfkeywords={Lorentz-Minkowski, sigma-k curvature, Liouville theorem, Newton tensor}}
\numberwithin{equation}{section}
\newtheorem{theorem}{Theorem}[section]
\newtheorem{proposition}[theorem]{Proposition}
\newtheorem{lemma}[theorem]{Lemma}
\newtheorem{corollary}[theorem]{Corollary}
\crefname{theorem}{Theorem}{Theorems}
\crefname{proposition}{Proposition}{Propositions}
\crefname{lemma}{Lemma}{Lemmas}
\crefname{corollary}{Corollary}{Corollaries}
\crefname{remark}{Remark}{Remarks}
\theoremstyle{definition}

\theoremstyle{remark}
\newtheorem{remark}[theorem]{Remark}
\newcommand{\R}{\mathbb R}
\newcommand{\Sym}{\operatorname{Sym}}
\newcommand{\tr}{\operatorname{tr}}
\newcommand{\Div}{\operatorname{div}}
\newcommand{\Id}{\operatorname{Id}}
\newcommand{\dd}{\,\mathrm d}
\newcommand{\dmu}{\,\mathrm d\mu}
\newcommand{\LF}{\mathcal L_F}
\newcommand{\Dk}{\mathcal D_k}
\newcommand{\J}{\mathcal J}
\newcommand{\pstar}{p_*}

\newcommand{\closureGC}{\overline{\Gamma_k}}

\newcommand{\ip}[2]{\langle#1,#2\rangle}
\newcommand{\bin}[2]{\binom{#1}{#2}}
\newcommand{\diag}{\operatorname{diag}}

\newcommand{\cK}{\mathcal K}
\newcommand{\cD}{\mathscr D}
\newcommand{\cB}{\mathcal B}

\newcommand{\eps}{\varepsilon}
\newcommand{\weakto}{\rightharpoonup}

\allowdisplaybreaks[1]
\title[Liouville rigidity for Lorentzian $\sigma_k$ equations]
{Universal Spacelikeness Estimates and Liouville Rigidity\\
for Lorentzian $\sigma_k$ Curvature Equations}
\author{Shujun Shi}
\address{School of Mathematical Sciences, Harbin Normal University,
Harbin 150025, Heilongjiang Province, China}
\email{shjshi@hrbnu.edu.cn}
\author{Yuzhou Zhang}
\address{School of Mathematical Sciences, Harbin Normal University,
Harbin 150025, Heilongjiang Province, China}
\email{yuzhouzhang@stu.hrbnu.edu.cn}
\date{}
\subjclass[2020]{35B53, 35J60, 53C50}
\keywords{Lorentz--Minkowski space, prescribed curvature, spacelike graph,
Liouville theorem, Newton tensor, Bernstein estimate}
\begin{document}
\begin{abstract}
We study nonnegative entire spacelike graphs in Lorentz--Minkowski space satisfying $\sigma_k(A[u])=u^p$. For $2\le k<n$ and $p\ge k$, pointwise strict spacelikeness and closed $k$-admissibility yield universal bounds for the height and Lorentz factor. Their proof combines block coercivity in the full G{\aa}rding cone, a Lorentzian cutoff, and hyperbolic-cap comparison. For $2\le k<n/2$ and $k\le p\le k(n+2)/(n-2k)$, every such solution vanishes, without symmetry, decay, finite-energy, curvature-pinching, or global Hessian assumptions. In the subcritical range we use a trace-free Newton tensor, a weighted divergence identity, and a finite descent. At the critical endpoint, the divergence identity retains an exact nonnegative defect and a positive quartic term. Direct estimates cover $2k<n\le4k+2$. For $n\ge4k+3$, we use compactness, concentration of the $k$-Hessian measure at the first crossing, a single-pole Pohozaev argument, core counting, a recursive defect estimate, and Souplet-type radius selection. A geometric corollary gives a rigidity result for complete spacelike immersions under the corresponding curvature equation.
\end{abstract}
\maketitle
\tableofcontents
\clearpage

\section{Introduction and main results}\label{sec:intro}

Throughout, $n$ and $k$ are integers with $1\leq k\leq n$, and $p>0$ is a real number;
additional restrictions are stated in each result. Let $\R^{n,1}$ denote
Lorentz--Minkowski space $\R^n\times\R$ with metric
\[
 \ip{(x,t)}{(y,s)}_L=x\cdot y-ts.
\]
For a $C^2$ function $u:\R^n\to\R$, consider the graph
$\Sigma_u=\{(x,u(x)):x\in\R^n\}$. It is spacelike precisely when $|Du|<1$. Here $D$ and $D^2$
denote the Euclidean gradient and Hessian, respectively. Write $g$ for the induced metric and $W=(1-|Du|^2)^{-1/2}$
for the Lorentz factor. Their coordinate formulas are recorded in
\cref{eq:metric-intro} below.
Let $\bar\nabla$ denote the flat connection of $\R^{n,1}$.
We fix the future-directed unit normal $\nu=W(Du,1)$ and define the second
fundamental form by $h(Y,Z)=-\ip{\bar\nabla_Y\nu}{Z}_L$. In graph coordinates,
\begin{equation}\label{eq:sign}
 h_{ij}=-W u_{ij},\qquad A=g^{-1}h.
\end{equation}
The entries of the eigenvalue vector $\lambda(A)=(\lambda_1,\cdots,\lambda_n)$
are the principal curvatures with this sign convention.
Write $\sigma_j$ for the $j$th elementary symmetric polynomial of
the principal curvatures. The G{\aa}rding cone $\Gamma_k$ is the component of $\{\sigma_k>0\}$ that contains the positive cone;
\cref{sec:prelim} gives the coordinate definition.
The equation under consideration is
\begin{equation}\label{eq:main}
 \sigma_k(A[u])=u^p \quad\hbox{in }\R^n,\qquad
 u\geq0,\quad |Du|<1,\quad \lambda(A[u])\in\closureGC.
\end{equation}
We distinguish the closed admissibility condition in \eqref{eq:main} from
strict admissibility, $\lambda(A)\in\Gamma_k$. For every nonzero solution,
\cref{lem:positive} establishes positivity and strict admissibility.

The sign convention \eqref{eq:sign} determines the admissible equation.
In particular,
\begin{equation}\label{eq:H-div}
 H:=\sigma_1(A)=-\Div_{\R^n}(WDu).
\end{equation}
For $k=1$, \eqref{eq:main} is therefore
$\Div_{\R^n}(WDu)+u^p=0$. Reversing the normal while retaining the condition
$\lambda(A)\in\overline{\Gamma_k}$ changes the admissible problem. All subsequent statements
use the convention \eqref{eq:sign}.

Classical rigidity results for entire spacelike hypersurfaces include
the maximal hypersurface theorem of Cheng and Yau~\cite{CY}. Prescribed mean curvature problems were developed by Bartnik
and Simon~\cite{BS} and Treibergs~\cite{Treibergs}; the behavior of the equation as the graph approaches the light cone is central to these problems.  Ma, Wu, Wu, and Yu~\cite{MWWY} studied the prescribed mean-curvature Lane--Emden equation.  For $k=1$, their theorem covers the critical Sobolev endpoint as well as the subcritical range, and they also obtain universal height and spacelikeness estimates.  For the same future-directed normal, our second fundamental form is the negative of the one used in \cite{MWWY}; under our convention, their equation takes the form \eqref{eq:main} with $k=1$.

The foundational elliptic theory of symmetric Hessian and curvature
operators is due to Caffarelli, Nirenberg, and Spruck~\cite{CNS3,CNS5}.
In Minkowski space, Urbas~\cite{UrbasScalar,UrbasInterior},
Bayard~\cite{Bayard}, Wang and Xiao~\cite{WX}, and Ren, Wang, and
Xiao~\cite{RWX} developed curvature estimates and existence theory in
several prescribed-curvature settings. These results concern related
geometric operators; their hypotheses and right-hand sides differ from
\eqref{eq:main}.

The corresponding Euclidean model is $\sigma_k(-D^2u)=u^p$.
For its semilinear counterpart, classical results include those of
Gidas and Spruck~\cite{GS} and Caffarelli, Gidas, and Spruck~\cite{CGS}. The relation of the $k$-Hessian model to Hessian Sobolev inequalities is
described by Chou and Wang~\cite{CW} and Wang~\cite{WangSurvey}.
Potential-theoretic and integral approaches to nonexistence were developed
by Phuc and Verbitsky~\cite{PV} and Ou~\cite{Ou}. Wang and Lei~\cite{WL}
studied radial critical exponents.

The invariant-tensor viewpoint used in Liouville problems was developed systematically by Ma and Wu~\cite{MaWuInvariant} and has also been applied to elliptic equations with source--gradient interactions by Ma and Wu~\cite{MaWuGradient}. In the present higher-curvature setting, this viewpoint motivates the trace-free tensor and weighted divergence identities below. The positive-weight integration-by-parts method has roots in conformal fully nonlinear equations studied by Chang, Gursky, and Yang~\cite{CGY} and Gonz\'alez~\cite{Gonzalez}, and was adapted to pure Hessian inequalities by Ou~\cite{Ou}. For the Euclidean $k$-Hessian algebra needed in our integral argument, we use quantitative Newton inequalities and the separate negative-weight cutoff estimate of Dai, Fu, Gui, and Qin~\cite{DFGQ}.

The gradient of the height function need not be a principal direction of the shape operator. Consequently, mixed curvature entries occur in the differential identity for the Lorentz factor $W$. A Lorentz-factor bound controls the graph metric but does not
control the ellipticity ratio of the Newton tensor. At the critical exponent, the coefficient of the weighted quadratic term vanishes. Lorentz admissibility does not imply $-D^2u\in\Gamma_k$, and an additive Green representation is unavailable for separating distant cores. 

We now state the main results.
\begin{theorem}[Universal a priori estimates]\label[theorem]{thm:apriori}
Let $2\leq k<n$ and $p\geq k$. There exist constants $U_*>0$ and $W_*>1$,
depending only on $n,k,p$, such that every nonnegative entire $C^2$ solution of
\eqref{eq:main} satisfies
\begin{equation}\label{eq:universal}
 u(x)\leq U_*,\qquad W(x)\leq W_*\quad(x\in\R^n).
\end{equation}
In particular, every such solution satisfies the uniform bound
$|Du|^2\leq1-W_*^{-2}$. We also have the pointwise estimate
\begin{equation}\label{eq:pointwise-G}
 W(x)\leq C_G(1+u(x))^{1+p/k},
\end{equation}
where $C_G$ depends only on $n,k,p$.
\end{theorem}

\begin{theorem}[Liouville theorem]\label[theorem]{thm:liouville}
Let $2\leq k<n/2$ and
\begin{equation}\label{eq:range}
 k\leq p\leq\pstar,\qquad \pstar=\frac{k(n+2)}{n-2k}.
\end{equation}
Every nonnegative entire $C^2$ solution of \eqref{eq:main} is identically zero.
\end{theorem}

No completeness, symmetry, decay, finite-energy condition, global Hessian
bound, or nonnegativity of each principal curvature is assumed in either
theorem.  The a priori theorem applies to every $p\ge k$; the restriction $p\le p_*$ belongs only to the Liouville theorem.

\begin{corollary}[Complete immersed hypersurfaces]\label[corollary]{cor:geometric}
Let $n>2k\geq4$ and let $p$ satisfy \eqref{eq:range}. Let
$X:\Sigma^n\to\R^{n,1}$ be a $C^2$ spacelike immersion of a connected
manifold without boundary, complete with respect to its induced metric.
Fix a spacelike affine hyperplane $\mathsf P$, its future-directed unit timelike
normal $N_0$, and a point $X_0\in \mathsf P$. Define the signed height by
\[
 u=-\ip{X-X_0}{N_0}_L.
\]
Suppose $u\geq0$ and, with the future-directed unit normal along $X$ and
the convention $h(Y,Z)=-\ip{\bar\nabla_Y\nu}{Z}_L$, the shape operator
satisfies $\lambda(A)\in\closureGC$ and $\sigma_k(A)=u^p$.
Then $X(\Sigma)=\mathsf P$.
\end{corollary}

The a priori estimates combine the block inequality of \cref{lem:block}, a Lorentzian cutoff, and hyperbolic-cap comparison; the block inequality controls mixed curvature terms even when the gradient is not principal. For Liouville rigidity we use $q=2(W-1)$ and a trace-free Newton tensor. When $p<p_*$, a weighted divergence identity and finite descent through $T_{k-1},\ldots,T_0$ remove the curvature weights. At $p=p_*$, the exact defect and quartic remainder yield direct rigidity for $2k<n\le4k+2$. For $n\ge4k+3$, the proof uses scale-invariant gradient control, small-defect bubble rigidity, concentration of the $k$-Hessian measure, a single-pole Pohozaev argument, scale drop, core counting, and a recursive estimate for the endpoint defect.

For $k=1$, the scaling and counting exponents agree with Ma, Wu, Wu, and Yu~\cite{MWWY}: the corresponding dimension ranges are $3\le n\le6$ and $n\ge7$. For $k>1$, we also need the compactness, Hessian-measure, and first-crossing arguments developed below.

The remaining sections are organized as follows. \Cref{sec:prelim,sec:block,sec:gradient,sec:cap} develop the preliminaries and a priori bounds. \Cref{sec:flux,sec:descent,sec:liouville} prove subcritical rigidity. The endpoint identity, direct low-dimensional proof, and local blow-down results appear in \cref{sec:critical-local}; \cref{sec:critical-global} gives neck exclusion, core counting, and the proofs of \cref{thm:liouville,cor:geometric}.

\section{Geometric and algebraic preliminaries}\label{sec:prelim}

\subsection{Graph notation and Newton tensors}
Here $u_i=\partial u/\partial x^i$ and
$u_{ij}=\partial^2u/\partial x^i\partial x^j$; the indices
$1\le i,j\le n$ are coordinate indices and $\delta_{ij}$ is
the Kronecker symbol. In these coordinates,
\begin{equation}\label{eq:metric-intro}
g_{ij}=\delta_{ij}-u_i u_j,\qquad W=(1-|Du|^2)^{-1/2}.
\end{equation}
For $\lambda=(\lambda_1,\ldots,\lambda_n)\in\R^n$, set
\[
\sigma_j(\lambda)=\sum_{1\le i_1<\cdots<i_j\le n}
\lambda_{i_1}\cdots\lambda_{i_j},\qquad\sigma_j(A)=\sigma_j(\lambda(A)).
\]
The full cone definition used throughout is
\[
\Gamma_k=\{\lambda:\sigma_j(\lambda)>0\text{ for }1\le j\le k\},\qquad\overline{\Gamma_k}=\operatorname{cl}(\Gamma_k).
\]
A domain is a connected open set. Write $B_R=B_R(0)$ and $B_R(x_0)$ for Euclidean balls, $\Id$ for the identity, and $\Sym(m)$ for real symmetric $m\times m$ matrices.
For symmetric matrices, $S\preceq T$ means that $T-S$ is positive semidefinite;
$S\prec T$ means that $T-S$ is positive definite. A uniform positive lower bound for $T$ on a specified set means that
$T\succeq c\Id$ throughout that set for a fixed $c>0$.
The dependence of $c$ on the data will be stated where the bound is used. A matrix can belong to $\Gamma_k$ even when some of its eigenvalues are negative.
The transpose of a matrix is denoted by a superscript $\mathsf T$.
We use the Hilbert--Schmidt inner product
$\langle S,T\rangle_{\mathrm{HS}}=\tr(S^{\mathsf T}T)$ and its associated
norm $|S|$ for coordinate matrices. The notation $\|S\|$ for a
matrix denotes its operator norm; in particular $\|\wedge^sS\|$
is the operator norm of the full compound matrix. Matrix products
and powers represent composition of the corresponding endomorphisms.

Tensor indices range from $1$ to $n$; a transverse frame index $\alpha$
ranges from $2$ to $n$. Repeated tensor indices in a term are summed unless
an explicit sum is written or an index is declared fixed. Indices specifying the order of a Newton tensor or the step of an
iteration are not summed under this convention.
For an $m\times m$ symmetric matrix, we set $\sigma_0=1$ and
$\sigma_j=0$ for $j>m$. The associated G{\aa}rding cone is taken in
$\R^m$. For a symmetric matrix
$S$, the abbreviation $S\in\Gamma_k$ means $\lambda(S)\in\Gamma_k$.
We denote the Levi--Civita connection of the induced metric $g$ by $\nabla$
and set $\Delta\phi=\tr_g\nabla^2\phi$. For a scalar function, $\nabla\phi$
also denotes its metric gradient. Unless otherwise indicated, $\Div$, inner products, and vector norms
are taken with respect to $g$; Euclidean divergence is written
$\Div_{\R^n}$.
Euclidean derivatives are denoted by $D$, and Euclidean inner products by
$\cdot$; the ambient indefinite inner product always has the subscript $L$.
Tensor components in graph coordinates are raised and lowered with $g$.
We use $C^m$, $C^{m,\alpha}$ ($0<\alpha<1$), and the space $C_c^\infty$ of smooth compactly supported functions. Domains are specified when they differ from the entire graph or $\R^n$.
Unless a domain is indicated, integrals with $\mathrm d\mu$ are over the
entire graph, and those with $\mathrm d x$ are over $\R^n$.
Integrals over subsets of $\R^n$ with $\mathrm d\mu$ are understood as integrals over the
corresponding graph. We also write $\mu(E)=\int_E d\mu$ for a
measurable set $E$. The volume form and inverse metric are
\begin{equation}\label{eq:measure}
 g^{ij}=\delta_{ij}+W^2u_i u_j,\qquad
 \dmu=W^{-1}\dd x.
\end{equation}
In an orthonormal tangent frame, we write $h$ for the symmetric matrix of
the shape operator. This convention allows expressions such as $h^2$ and
$T_j(h)$ without additional metric factors. For a self-adjoint tangent
endomorphism $T$, we use $T(X,Y)=\langle TX,Y\rangle$. For symmetric
$(0,2)$-tensors $S,T$, the contraction is $S:T=S^{ij}T_{ij}$. We identify
self-adjoint endomorphisms and their associated symmetric tensors using $g$.
In particular, $(\Div T)^j=\nabla_iT^{ij}$.

In graph coordinates, the matrix $g^{-1}h$ is similar to the symmetric
representative
\begin{equation}\label{eq:symmetric-shape}
 \widehat h=-Wg^{-1/2}D^2u\,g^{-1/2}.
\end{equation}
Here $g^{-1/2}$ is the positive definite square root of $g^{-1}$.
This representative will be used both for local regularity and for
comparison at a contact point with another graph.

The Newton tensors are defined recursively by
\begin{equation}\label{eq:newton-rec}
 T_0=\Id,\qquad T_j=\sigma_j(h)\Id-hT_{j-1}\quad(1\leq j\leq n).
\end{equation}
The Cayley--Hamilton identity gives $T_n=0$.
We use the notation
\begin{equation}\label{eq:basic-symbols}
 v=\nabla u,\quad z=|v|^2,\quad f=u^p,\quad
 F=T_{k-1}(h),\quad Q=\tr(Fh^2),\quad
 \LF\phi=F^{ij}\nabla_i\nabla_j\phi.
\end{equation}
For $h\in\Gamma_k$, the tensors $T_j$ are positive definite for
$0\leq j\leq k-1$. Positivity is pointwise and does not provide ellipticity constants uniform over all solutions. For $0\leq j\leq n-1$, the
standard trace identities, together with the order-$k$ recursion, are
\begin{equation}\label{eq:newton-traces}
 \tr T_j=(n-j)\sigma_j,\qquad
 \tr(T_jh)=(j+1)\sigma_{j+1},\qquad
Fh=f\Id-T_k.
\end{equation}
The trace identities follow in an eigenbasis. For the G{\aa}rding cones and Newton--Maclaurin inequalities, see \cite{Garding,CNS3,WangSurvey}.

Unless a constant is explicitly fixed, $C>0$ may change from one line to
the next; its allowed dependence is stated in the result being proved.
The notation $C_\varepsilon$ allows additional dependence on the chosen
positive parameter $\varepsilon$. Constants in cutoff estimates are
independent of the cutoff radius $R$. Any allowed dependence on the
particular solution is stated explicitly.

\subsection{Reduction to smooth positive solutions}
The next lemma justifies differentiating nonzero $C^2$ solutions using local ellipticity alone. No global bound for $W$ is required.

\begin{lemma}[Reduction to smooth positive solutions]\label[lemma]{lem:positive}
An entire nonnegative $C^2$ solution of \eqref{eq:main} is either identically zero or strictly
positive. In the latter case, $h\in\Gamma_k$ and $u\in C^\infty(\R^n)$.
\end{lemma}
\begin{proof}
The closed cone condition gives $H\geq0$. By \eqref{eq:H-div},
$-\Div_{\R^n}(WDu)\geq0$. On each relatively compact ball this is a
uniformly elliptic quasilinear inequality, since $|Du|<1$ and $Du$ is
continuous. In nondivergence form, the coefficient matrix is
$W\Id+W^3Du\otimes Du\succ0$. The strong minimum principle
\cite[Theorem 3.5]{GT} shows that an interior zero forces $u\equiv0$ on the
connected domain.

If $u>0$, the Newton--Maclaurin inequalities, obtained on the closed cone
by approximation, imply
\[
 \left(\frac{\sigma_j(h)}{\bin nj}\right)^{1/j}
 \geq\left(\frac{u^p}{\bin nk}\right)^{1/k}>0
 \quad(1\leq j\leq k).
\]
Thus $h\in\Gamma_k$. 

To justify differentiating the equation, write $w=-u$ and $\mathsf G=g^{-1/2}$.
Using \eqref{eq:symmetric-shape}, the equation becomes
\[
 \widehat h=W\mathsf G D^2w \mathsf G,\qquad
 \mathcal F(D^2w,Dw):=\sigma_k(\widehat h)^{1/k}=u^{p/k}.
\]
With the first derivatives fixed, the linearization in the Hessian variable is
\[
 \mathcal F^{ij}
 =\frac{W}{k}\sigma_k(\widehat h)^{1/k-1}
       \bigl(\mathsf G T_{k-1}(\widehat h)\mathsf G \bigr)^{ij}.
\]
The linearization matrix is positive definite on the admissible set.
For fixed first derivatives, $\mathcal F$ is concave in the Hessian variable.
Fix $x_0$ and regard the known function $Dw(x)$ as a coefficient by setting
\[
 \mathfrak F(x,N)=\mathcal F(N,Dw(x)).
\]
Continuity of $Dw$ and $D^2w$ allows us to choose a sufficiently small ball
$B_r(x_0)$ and a compact convex neighborhood $\mathfrak K\subset\Sym(n)$ of $D^2w(x_0)$
such that $D^2w(x)$ lies in the interior of $\mathfrak K$ and
$W(x)\mathsf G(x)N\mathsf G(x)\in\Gamma_k$ for every
$(x,N)\in\overline{B_r(x_0)}\times\mathfrak K$.
On this product set there are constants $0<\lambda_0\leq\Lambda_0<\infty$
such that
\[
 \lambda_0\Id\preceq\mathfrak F_N(x,N)\preceq\Lambda_0\Id.
\]
These constants are local and may depend on the chosen solution,
$x_0$, $r$, and $\mathfrak K$.
Since $Dw\in C^1$, the operator and its Hessian derivative are locally
Lipschitz in $x$, uniformly for $N\in\mathfrak K$.

For completeness, an extension to all symmetric matrices is given by
\[
 \widetilde{\mathfrak F}(x,N)
 =\inf_{S\in\mathfrak K}\left\{
   \mathfrak F(x,S)+\tr\bigl(\mathfrak F_N(x,S)(N-S)\bigr)
   \right\}.
\]
Concavity shows that this extension agrees with $\mathfrak F$ on
$\mathfrak K$. The coefficient matrices of these affine functions lie
between $\lambda_0\Id$ and $\Lambda_0\Id$. Hence
$\widetilde{\mathfrak F}$ is uniformly elliptic and concave on the space
of all symmetric matrices. Moreover,
\[
 |\widetilde{\mathfrak F}(x,N)-\widetilde{\mathfrak F}(y,N)|
 \leq C|x-y|(1+|N|).
\]
The local estimates for concave uniformly elliptic equations
\cite{Evans,Krylov,CC} therefore give $w\in C^{2,\alpha}$ on a smaller
ball, for some $\alpha\in(0,1)$. Here the right-hand side $u^{p/k}$ is
locally Lipschitz because $u$ is positive and $C^2$.
Returning to the original smooth operator, difference quotients and
interior Schauder estimates first give $C^{3,\alpha}$ regularity and
then, by iteration, smoothness; see \cite[Chapters 6 and 17]{GT}.
The nonlinearity is smooth on every compact subinterval of $(0,\infty)$,
which suffices because $u$ has a positive minimum on each compact subset.
The regularity argument uses neither a global bound for $W$ nor a positive lower bound for $u$ at infinity.
\end{proof}

By \cref{lem:positive}, every nonzero entire solution considered below is smooth, positive, and strictly admissible.
The dimension and exponent restrictions will be stated where they are used.

\subsection{Height and angle identities}
We use the following height and angle identities in both the maximum-principle argument and the integral estimates.

\begin{lemma}[Height and angle identities]\label[lemma]{lem:geometry}
For a smooth positive solution of \eqref{eq:main},
\begin{align}
 &z=W^2-1,\qquad \nabla^2u=-Wh,\qquad
 \nabla W=-hv,\label{eq:height-angle}\\
 &\Delta u=-WH,\qquad \LF u=-kWf,\label{eq:height-L}\\
 &\Div T_j=0\quad(0\leq j\leq n-1),\qquad
 \LF W=WQ-pu^{p-1}z.\label{eq:angle-L}
\end{align}
\end{lemma}
\begin{proof}
The first formula follows from \eqref{eq:measure}:
$g^{ij}u_i u_j=W^2|Du|^2=W^2-1$.
The Christoffel symbols of the graph metric are
$\Gamma^\ell_{ij}=-W^2u_\ell u_{ij}$. Hence
$u_{ij}-\Gamma^\ell_{ij}u_\ell=W^2u_{ij}=-Wh_{ij}$.
Differentiating $W^2=1+|\nabla u|^2$ and using this Hessian formula gives
$W\nabla_iW=-Wh_{ij}v^j$, so $\nabla W=-hv$.
The contractions in \eqref{eq:height-L} follow from \eqref{eq:newton-traces}.

The flat ambient Codazzi equation is $\nabla_i h_{j\ell}=\nabla_jh_{i\ell}$.
Together with the recursion \eqref{eq:newton-rec}, it proves
$\Div T_j=0$ by induction. Indeed,
in an orthonormal frame, with repeated indices summed,
\[
 \nabla_i(T_j)^{i\ell}
 =\nabla_\ell\sigma_j
  -(\nabla_iT_{j-1}^{ia})h_{a\ell}
  -T_{j-1}^{ia}\nabla_i h_{a\ell}=0.
\]
The middle term vanishes by the induction hypothesis. The Codazzi
equation rewrites the last term as
$-T_{j-1}^{ia}\nabla_\ell h_{ai}=-\nabla_\ell\sigma_j$. This is the usual Newton-tensor divergence
identity; see also Reilly~\cite{Reilly}.
Finally, differentiating $\nabla_iW=-h_{i\ell}v_\ell$ and contracting gives
\[
 \LF W=-F^{ij}(\nabla_jh_{i\ell})v_\ell
        -F^{ij}h_{i\ell}\nabla_jv_\ell
       =-\ip{\nabla f}{v}+WQ.
\]
Since $\nabla f=pu^{p-1}v$, this is \eqref{eq:angle-L}.
\end{proof}

\subsection{Two elementary consequences of admissibility}
The Newton--Maclaurin inequality at orders $1$ and $k$ gives
\begin{equation}\label{eq:H-lower}
 H\geq n\bin nk^{-1/k}u^{p/k}.
\end{equation}
If a symmetric $n\times n$ matrix $h$ belongs to $\Gamma_k$, its restriction
to any hyperplane belongs to $\Gamma_{k-1}$. Indeed, for a unit vector $e$
and $B=h|_{e^\perp}$, expansion of the principal minors in a frame with
first vector $e$ gives
\[
 \sigma_j(B)
 =\left.\frac{d}{dt}\right|_{t=0}\sigma_{j+1}(h+t e\otimes e)
 =T_j(h)(e,e)>0,\qquad 1\leq j\leq k-1.
\]
Thus $B\in\Gamma_{k-1}$, and its Newton tensors through order $k-2$ are
positive definite. The argument applies to every unit vector $e$; the hyperplane normal
need not be an eigenvector of $h$.
\subsection{Coordinate comparison at a contact point}\label{subsec:contact-comparison}
At a contact point, comparison relies on an ordering of the symmetric shape representatives. In graph coordinates, this representative is
\[
 \widehat h=-Wg^{-1/2}D^2u\,g^{-1/2}.
\]
If two spacelike graphs $u,V$ have the same Euclidean gradient at a contact point, then they have the same $g$ and $W$.  Hence
\begin{equation}\label{eq:contact-ordering}
 D^2u\succeq D^2V
 \quad\Longrightarrow\quad
 \widehat h[u]\preceq \widehat h[V].
\end{equation}
This is the ordering used in the hyperbolic-cap and zero-curvature barrier comparisons.  The Gauss formula with our sign convention is
\[
 \bar\nabla_iX_j=\Gamma^\ell_{ij}X_\ell-h_{ij}\nu,
\]
whose time component gives again $\nabla^2u=-Wh$.  These coordinate formulas will be used in the comparison arguments below.

\section{A block matrix estimate in the admissible cone}\label{sec:block}

In a frame aligned with the height gradient, $e_1$ need not be a principal direction. The equation for $\log W$ involves all transverse columns of $h$. The next estimate controls them when $h_{11}<0$ and the mixed entries are small.

A negative diagonal entry is compatible with admissibility. For any $2\leq k<n$ consider
$h=\operatorname{diag}(-a,1,\ldots,1)$, with
$0<a<(n-k)/k$. For $1\leq j\leq k$,
\[
 \sigma_j(h)=\bin{n-1}{j-1}
                      \left(\frac{n-j}{j}-a\right)>0.
\]

\begin{lemma}[Block coercivity]\label[lemma]{lem:block}
Let $2\leq k<n$, $d=n-k$, and
\[
 h=\begin{pmatrix}a&b^{\mathsf T}\\ b&B\end{pmatrix}\in\Gamma_k,\qquad
 a<0,\qquad |b|\leq\frac{|a|}{4\sqrt d}.
\]
Write
$F=T_{k-1}(h)=\begin{pmatrix}\varphi&c^{\mathsf T}\\c&F_\perp\end{pmatrix}$,
where $a,\varphi\in\R$, $b,c\in\R^{n-1}$, and
$B,F_\perp\in\Sym(n-1)$. Let $e_1,\ldots,e_n$ be the
standard orthonormal basis used in this block decomposition.
Then
\begin{align}
 \tr F&\leq(d+1)\varphi,\label{eq:block-tr}\\
 \varphi\geq c_{n,k}\,\sigma_k(h)^{(k-1)/k},
 &\quad
 c_{n,k}=\bin{n-1}{k-1}\bin{n-1}{k}^{-(k-1)/k},
 \label{eq:block-f}\\
 \sum_{\alpha=2}^n F(he_\alpha,he_\alpha)
 &\geq\frac{\varphi a^2}{16d}.\label{eq:block-K}
\end{align}
All matrix entries and norms in this statement are Euclidean in the chosen
orthonormal frame. No positivity assumption on the eigenvalues of $B$ is
made.
\end{lemma}
\begin{proof}
We first prove \eqref{eq:block-tr} and \eqref{eq:block-f} using the
restriction of $h$ to the transverse subspace. We then prove
\eqref{eq:block-K} by combining a Schur complement estimate with a
weighted trace inequality.

The restriction property gives $B\in\Gamma_{k-1}$. Expanding the principal
minors according to the first row and column yields
\begin{equation}\label{eq:block-expansion}
 \sigma_j(h)=\sigma_j(B)+a\sigma_{j-1}(B)
              -b^{\mathsf T}T_{j-2}(B)b \quad(j\geq2).
\end{equation}
For $j=1$ the identity is $\sigma_1(h)=a+\sigma_1(B)$.
For $2\leq j\leq k$, the tensor $T_{j-2}(B)$ in \eqref{eq:block-expansion} is positive
definite because $B\in\Gamma_{k-1}$. Consequently
\[
 \sigma_{k-1}(h)\leq\sigma_{k-1}(B)=\varphi.
\]
For $k=2$, the preceding inequality follows from the $j=1$ identity.
Combining it with $\tr F=(n-k+1)\sigma_{k-1}(h)$ proves
\eqref{eq:block-tr}.
At order $k$, \eqref{eq:block-expansion} also gives
\[
 \sigma_k(B)\geq \sigma_k(h)+|a|\sigma_{k-1}(B)>0.
\]
Thus $B\in\Gamma_k$ in dimension $n-1$. Applying the Maclaurin inequality
in that dimension proves \eqref{eq:block-f}.

To prove \eqref{eq:block-K}, we use the positivity of $F$.
The Schur complement inequality, together with \eqref{eq:block-tr}, gives
\begin{equation}\label{eq:schur}
 cc^{\mathsf T}\preceq\varphi F_\perp,\qquad
 \tr F_\perp\leq d\varphi,\qquad |c|\leq\sqrt d\,\varphi.
\end{equation}
The trace identity $\tr(Fh)=k\sigma_k(h)$ implies
\[
 J_B:=\tr(F_\perp B)
   =k\sigma_k(h)-\varphi a-2c\cdot b
   \geq\varphi|a|-2\sqrt d\,\varphi|b|
   \geq\frac{\varphi|a|}{2}.
\]
Let $Y_B=\tr(F_\perp B^2)\geq0$. Cauchy--Schwarz for the Hilbert--Schmidt
inner product gives
\[
 |J_B|=|\ip{F_\perp^{1/2}}{F_\perp^{1/2}B}_{\mathrm{HS}}|
       \leq(\tr F_\perp)^{1/2}(\tr(F_\perp B^2))^{1/2}.
\]
Since $F_\perp\succ0$, it follows that
\begin{equation}\label{eq:DB}
 Y_B\geq \frac{J_B^2}{\tr F_\perp}.
\end{equation}
The left-hand side of \eqref{eq:block-K} is
\[
 \varphi|b|^2+2b^{\mathsf T}Bc+\tr(F_\perp B^2).
\]
Since $B^2$ is positive semidefinite, \eqref{eq:schur} implies
\[
 |Bc|^2=\tr(B^2cc^{\mathsf T})\leq\varphi\tr(B^2F_\perp)=\varphi Y_B.
\]
Young's inequality gives $2|b^{\mathsf T}Bc|\leq2\varphi|b|^2+Y_B/2$.
Hence the preceding expression is at least
$Y_B/2-\varphi|b|^2$. Combining this with \eqref{eq:DB} gives
\[
 \frac{Y_B}{2}-\varphi|b|^2
 \geq\frac{\varphi a^2}{8d}-\frac{\varphi a^2}{16d}
 =\frac{\varphi a^2}{16d},
\]
as claimed.
\end{proof}

\begin{remark}\label[remark]{rem:block-use}
The smallness condition on the mixed block will follow from the first
derivative equation at the maximum point in \cref{prop:gradient}.
This condition is needed only at the maximum point and imposes no global curvature-pinching assumption on the solution.
\end{remark}

\section{The universal pointwise gradient estimate}\label{sec:gradient}

\subsection{A Lorentzian cutoff}
For the cutoff construction, let $u$ be a smooth positive entire solution
of \eqref{eq:main}. Fix $x_0\in\R^n$ and set
\begin{equation}\label{eq:cutoff-def}
 R=1+u(x_0),\quad \xi=x-x_0,\quad t=u(x)-u(x_0),\quad
 \eta=1-\frac{|\xi|^2-t^2}{R^2}-\frac{4t}{R}.
\end{equation}
Set $\Omega=\{x:\eta(x)>0\}$. The next lemma proves that $\Omega$ is bounded under positivity and strict spacelikeness, without a prior global estimate for $W$.

\begin{lemma}[Cutoff geometry]\label[lemma]{lem:cutoff}
On $\Omega$,
\begin{equation}\label{eq:cutoff-bounds}
 -R<t<R/4,\qquad |\xi|<\sqrt6R,\qquad 0<\eta<5.
\end{equation}
If
\begin{equation}\label{eq:a0d0}
 a_0=\frac4R-\frac{2t}{R^2},\qquad
 d_0=a_0+\frac{2\xi\cdot Du}{R^2},
\end{equation}
then $0<a_0\leq6/R$, $d_0>1/R$, and
\begin{align}
 \ip{\nabla\eta}{v}&=-W^2d_0+a_0,\label{eq:eta-v}\\
 \LF\eta&=-\frac{2}{R^2}\tr F+kWf\,d_0.\label{eq:Leta}
\end{align}
\end{lemma}
\begin{proof}
Positivity of $u$ gives $t>-u(x_0)>-R$.
For $\xi\ne0$, integrating $Du$ along the compact line segment from
$x_0$ to $x$ gives $|t|<|\xi|$. Therefore $\eta>0$ implies $t<R/4$.
The defining inequality for $\Omega$ gives
$|\xi|^2<R^2+t^2-4Rt<6R^2$, and $0<\eta\leq1-4t/R<5$.
In particular $\overline\Omega$ is compact, and $\eta=0$ on its boundary.

Put $s=t/R$ and $\rho=|\xi|/R$. Since
\[
 (4-2s)^2-4\rho^2=12+4\eta>12
\]
and $4-2s+2\rho<6+2\sqrt6<12$, division gives
$4-2s-2\rho>1$. Using $|Du|<1$ in \eqref{eq:a0d0} proves $Rd_0>1$.
The estimates for $a_0$ are immediate.

The graph identities
$\ip{\nabla x^j}{v}=W^2u_j$ and $z=W^2-1$ show that
\[
 \ip{\nabla\eta}{v}
   =-\frac{2W^2\xi\cdot Du}{R^2}-a_0z=-W^2d_0+a_0.
\]
The ambient squared separation $s_L^2=|\xi|^2-t^2$ satisfies
\[
 \LF s_L^2=2\tr F-2kWf(\xi\cdot Du-t).
\]
This follows either from the Gauss formula or by differentiating the graph
coordinates, for which $\LF x^j=-kWf\,u_j$.
Using $\LF t=-kWf$ in \eqref{eq:cutoff-def} now gives \eqref{eq:Leta}.
\end{proof}

The coefficient $4$ of the term linear in $t$ in \eqref{eq:cutoff-def}
ensures $d_0>1/R$ without a global bound for $W$. As shown below, this
lower bound forces the gradient-direction entry of $h$ to be negative
at a maximum point where $W$ is large. Boundedness of the open set $\Omega$ suffices;
its connectedness is not needed.

\subsection{The maximum-point calculation}
\begin{proposition}\label[proposition]{prop:gradient}
Under the assumptions of \cref{thm:apriori}, \eqref{eq:pointwise-G} holds.
One admissible choice is
\begin{equation}\label{eq:CG}
 C_G=5^{256n}\max\left\{100\sqrt n,\
 \frac{32(n-k)p}{(1+p/k)^2c_{n,k}}\right\}.
\end{equation}
\end{proposition}
\begin{proof}
We apply the maximum principle to the Lorentz factor multiplied by a
cutoff and divided by a power of $1+u$. When the Lorentz factor is large at the maximum point, the first
derivative equation forces the gradient-direction entry of $h$ to be negative
and controls the mixed entries. The block estimate then supplies the
positive curvature term needed to absorb the two cutoff errors.

\smallskip\noindent
\emph{Step 1: the maximum point and the bounded-angle case.}
If $u\equiv0$, then $W=1$ and the estimate is immediate. Otherwise, \cref{lem:positive} implies that $u$ is smooth and positive,
so we may differentiate the equation. Write
\[
 A_0=1+\frac pk,\qquad \beta=256n,\qquad W_0=100\sqrt n,
 \qquad G=\frac{\eta^\beta W}{(1+u)^{A_0}}.
\]
The function $G$ is positive at $x_0$, continuous on $\overline\Omega$,
and zero on $\partial\Omega$. Since $W$ is continuous on the compact set $\overline\Omega$, it is
bounded there. Hence $G$ attains its maximum at an interior point $x_1$;
no global bound for $W$ is needed. If $W(x_1)<W_0$, then
$G(x_0)\leq G(x_1)\leq5^\beta W_0$, which already gives the conclusion.
In the remaining case, $W(x_1)\geq W_0$; all calculations below are at $x_1$.

\smallskip\noindent
\emph{Step 2: the curvature block in the gradient direction.}
Choose an orthonormal frame with $e_1=v/\sqrt z$. Define
\begin{equation}\label{eq:L-def}
 b_0=\frac{A_0}{1+u},\qquad
 \ell_0=-\frac{\beta}{\eta}\frac{\ip{\nabla\eta}{v}}{z},
 \qquad L=b_0+\ell_0.
\end{equation}
By \eqref{eq:eta-v}, $d_0>R^{-1}$, and $a_0\leq6R^{-1}$,
\begin{equation}\label{eq:ell-bound}
 \ell_0\geq\frac{\beta}{2R\eta},
\end{equation}
whenever $W^2\geq12$. The stationarity condition $\nabla\log G=0$ gives
\begin{equation}\label{eq:first-max}
 \frac{\beta}{\eta}\nabla\eta=b_0v+\frac{hv}{W}.
\end{equation}
Its $e_1$ component implies $h_{11}=-WL<0$.

Let $b=(h_{1\alpha})_{\alpha>1}$. Since $e_\alpha u=0$, the spatial
component of $e_\alpha$ has Euclidean length one and the collection of those
components is orthonormal. Thus
$|\nabla_{\{e_1\}^\perp}\eta|\leq2|\xi|/R^2\leq2\sqrt6/R$,
where $\nabla_{\{e_1\}^\perp}\eta$ denotes the orthogonal projection
of $\nabla\eta$ onto $e_1^\perp$.
The remaining components of \eqref{eq:first-max} give
\begin{equation}\label{eq:b-bound}
 |b|\leq\frac{4\sqrt6\,\beta}{R\eta}
       \leq8\sqrt6\,\ell_0\leq8\sqrt6\,L.
\end{equation}
In deriving \eqref{eq:b-bound}, we used $W/\sqrt z\leq2$. Because $W^2\geq10^4n\geq6144(n-k)$,
\eqref{eq:b-bound} implies
$|b|\leq|h_{11}|/(4\sqrt{n-k})$. The block lemma applies with $a=-WL$.

\smallskip\noindent
\emph{Step 3: the logarithmic equation and the cutoff errors.}
Set $d=n-k$ and $\varphi=F^{11}$. The nonnegative curvature expression
in the logarithmic angle equation is
\begin{align}
 K&:=Q-\frac{F(hv,hv)}{W^2}\notag\\
  &=\sum_{\alpha>1}F(he_\alpha,he_\alpha)
                  +\frac1{W^2}F(he_1,he_1)
   \geq\frac{\varphi W^2L^2}{16d}.\label{eq:K-max}
\end{align}
Twice differentiating $\log G$ and using \cref{lem:geometry} yields
\begin{align}
 0\geq\LF\log G
 ={}&K-\frac{pu^{p-1}z}{W}
     +\frac{\beta}{\eta}\LF\eta
     -\frac{\beta}{\eta^2}F(\nabla\eta,\nabla\eta)\notag\\
 &+kb_0Wf+\frac{A_0}{(1+u)^2}F(v,v).\label{eq:second-max}
\end{align}
The term $kWf\,d_0$ in \eqref{eq:Leta} is positive. It remains to bound the two negative cutoff terms.
Equation \eqref{eq:first-max} can be rewritten as
\[
 \frac{\beta}{\eta}\nabla\eta
       =\sqrt z\left(-\ell_0e_1+\frac bW\right).
\]
Write $F=\left(\begin{smallmatrix}\varphi&c^{\mathsf T}\\c&F_\perp\end{smallmatrix}\right)$
in this frame. By \eqref{eq:block-tr} and \eqref{eq:schur},
$\tr F_\perp\leq d\varphi$ and $|c|\leq\sqrt d\,\varphi$. In particular,
\[
 b^{\mathsf T}F_\perp b\leq d\varphi|b|^2,\qquad
 |c\cdot b|\leq\sqrt d\,\varphi|b|.
\]
Expanding the quadratic form therefore gives
\[
 F\left(-\ell_0e_1+\frac bW,-\ell_0e_1+\frac bW\right)
 \leq\varphi\left(\ell_0+\frac{\sqrt d\,|b|}{W}\right)^2
 \leq3\varphi L^2.
\]
Indeed, \eqref{eq:b-bound} and $W\geq100\sqrt n$ give
$\sqrt d\,|b|/W\leq(8\sqrt{6d}/(100\sqrt n))L<L/5$, while
$\ell_0\leq L$. Hence the quadratic form is at most $(6/5)^2\varphi L^2\leq3\varphi L^2$.
It follows that
\begin{equation}\label{eq:cutoff-square}
 \frac{\beta}{\eta^2}F(\nabla\eta,\nabla\eta)
 \leq\frac3\beta\varphi W^2L^2
 \leq\frac{\varphi W^2L^2}{64d}.
\end{equation}
Using \eqref{eq:ell-bound} and $\eta<5$ also gives
\begin{equation}\label{eq:cutoff-trace}
 \frac{2\beta\tr F}{R^2\eta}
 \leq\frac{40(d+1)}{\beta W^2}\varphi W^2L^2
 \leq\frac{\varphi W^2L^2}{64d}.
\end{equation}
The last inequalities in \eqref{eq:cutoff-square} and
\eqref{eq:cutoff-trace} require only
$\beta\geq192d$ and $\beta W^2\geq2560d(d+1)$.
Both conditions follow from the choices of $\beta$ and $W_0$ and the
assumption $W\geq W_0$.

Combining \eqref{eq:second-max} with \eqref{eq:K-max},
\eqref{eq:cutoff-square}, and \eqref{eq:cutoff-trace}, dropping the other
positive terms, and using $z/W\leq W$, we obtain
\begin{equation}\label{eq:max-final}
 0\geq\frac{\varphi W^2b_0^2}{32d}-pu^{p-1}W.
\end{equation}
\smallskip\noindent
\emph{Step 4: returning to the center of the cutoff.}
The bound \eqref{eq:block-f} now gives
\begin{equation}\label{eq:W-at-max}
 W(x_1)\leq\frac{32dp}{A_0^2c_{n,k}}\,
             u(x_1)^{p/k-1}(1+u(x_1))^2
       \leq\frac{32dp}{A_0^2c_{n,k}}(1+u(x_1))^{A_0}.
\end{equation}
The second inequality in \eqref{eq:W-at-max} uses $p\geq k$.
Consequently $G(x_0)\leq G(x_1)\leq5^\beta C$, with $C$ the maximum in
\eqref{eq:CG}. Since $\eta(x_0)=1$, the asserted estimate follows.
\end{proof}

\begin{remark}
The identity for $\mathcal L_F W$ in \eqref{eq:angle-L} is obtained by differentiating $\sigma_k(h)=u^p$ once. No differential
inequality for $H$ is needed.
The negative gradient-direction entry of $h$ is compatible with
$h\in\Gamma_k$; it is a consequence of the maximum-point equation and is
used only to apply the block lemma at that point.
\end{remark}

\section{Hyperbolic caps and uniform height bounds}\label{sec:cap}

To remove the height dependence in \eqref{eq:pointwise-G}, we compare the solution with a hyperbolic cap. The boundary minimum controls the cap curvature, and the gradient estimate at that point bounds the cap radius and the height at its center.

We first record the radial curvature formula used here and in the
exterior comparison in \cref{lem:lower-barrier}. For a radial spacelike graph $V(r)$,
put $y=-V'/\sqrt{1-(V')^2}$, where a prime denotes differentiation
with respect to $r$. With the convention \eqref{eq:sign}, its
principal curvatures for $r>0$ are
\begin{equation}\label{eq:radial-curvatures}
 \lambda_{\rm rad}=y',\qquad \lambda_{\rm tan}=\frac yr,
\end{equation}
where the tangential curvature has multiplicity $n-1$. Indeed,
the radial entry of the metric is $1-(V')^2$, so
$\lambda_{\rm rad}=-V''/(1-(V')^2)^{3/2}=y'$; on a unit tangential
vector the metric is Euclidean and the curvature is
$-V'/(r\sqrt{1-(V')^2})=y/r$.

Consequently, for a radial spacelike graph,
\begin{equation}\label{eq:radial-sigma-k}
 \sigma_k(A[V])
 =\binom{n-1}{k-1}y'\left(\frac yr\right)^{k-1}
  +\binom{n-1}{k}\left(\frac yr\right)^k
 =\frac{\binom{n-1}{k-1}}{k r^{n-1}}
   \left(r^{n-k}y^k\right)'.
\end{equation}
The choice $y=r/r_0$ gives the constant-curvature hyperbolic cap, while $y=a r^{-(n-k)/k}$ gives a zero-$k$-curvature profile on the boundary of $\Gamma_k$.  Both forms will be used below.

\subsection{Hyperbolic cap comparison and the universal bounds}

\begin{proof}[Proof of \cref{thm:apriori}]
The estimates hold for the zero solution. By \cref{lem:positive}, it remains
to consider a smooth positive solution. Let
\[
 b_{n,k}=\bin nk^{1/k},\qquad A_0=1+\frac pk,\qquad
 K_{\mathrm{cap}}=C_G(2+b_{n,k})^{A_0}.
\]
Fix $x_c$ such that $u(x_c)>1$ and put $R_c=u(x_c)-1$.
The strict spacelike condition implies
\[
 m:=\min_{\partial B_{R_c}(x_c)}u\geq u(x_c)-R_c=1.
\]
There are no interior local minima of $u$: at such a minimum,
$D^2u\succeq0$ would imply $h\preceq0$ and $H\leq0$, contradicting
\eqref{eq:H-lower}. Hence $u>m$ in $B_{R_c}(x_c)$.

Set $r_0=b_{n,k}m^{-p/k}$ and define the hyperbolic cap
\begin{equation}\label{eq:cap}
 V(x)=m+\sqrt{r_0^2+R_c^2}-\sqrt{r_0^2+|x-x_c|^2}.
\end{equation}
For this graph $y(r)=-V'/\sqrt{1-(V')^2}=r/r_0$, where $r=|x-x_c|$.
Formula \eqref{eq:radial-curvatures}, extended continuously at $r=0$,
therefore gives $r_0^{-1}$ for every principal curvature. Thus
\begin{equation}\label{eq:cap-curvature}
 \sigma_k(A[V])=m^p,\qquad V=m\quad\hbox{on }\partial B_{R_c}(x_c).
\end{equation}

We claim that $u\geq V$ in the ball. Otherwise $u-V$ has a negative
interior minimum. At that point, $Du=DV$ and $D^2u\succeq D^2V$.
The two graphs therefore have the same metric and Lorentz factor.
Their symmetric shape representatives in \eqref{eq:symmetric-shape}
satisfy $\widehat h[u]\preceq\widehat h[V]$.
Monotonicity of $\sigma_k$ under positive
semidefinite addition in $\Gamma_k$ gives $u^p\leq m^p$.
The segment from $\widehat h[u]$ to $\widehat h[V]$ remains in $\Gamma_k$, since this cone
is preserved under positive semidefinite addition. Along that segment the
Newton tensor is positive definite, proving the stated monotonicity without
simultaneous diagonalization. This contradicts $u>m$ and proves the comparison.

At the center, comparison implies
\[
 R_c+1=u(x_c)\geq m+\sqrt{r_0^2+R_c^2}-r_0\geq m+R_c-r_0.
\]
It follows that
\begin{equation}\label{eq:m-bound}
 m\leq1+r_0\leq1+b_{n,k}.
\end{equation}
Choose $\widehat x\in\partial B_{R_c}(x_c)$ with $u(\widehat x)=m$.
Let $\mathbf n$ denote the outward Euclidean unit normal to the ball.
The tangential Euclidean derivatives of $u$ vanish at $\widehat x$.
Since $u-V\geq0$ in the ball and $(u-V)(\widehat x)=0$, the outward
normal derivative of $u-V$ at $\widehat x$ is nonpositive.
Thus
\[
 u_{\mathbf n}(\widehat x)\leq V_{\mathbf n}(\widehat x)
        =-\frac{R_c}{\sqrt{r_0^2+R_c^2}}.
\]
Consequently
\begin{equation}\label{eq:boundary-W}
 \sqrt{1+(R_c/r_0)^2}\leq W(\widehat x)
       \leq C_G(1+m)^{A_0}\leq K_{\mathrm{cap}}.
\end{equation}
Here the upper bound is \cref{prop:gradient}, evaluated at the boundary
minimum, whose height is controlled by \eqref{eq:m-bound}.
Thus $R_c\leq r_0K_{\mathrm{cap}}\leq b_{n,k}K_{\mathrm{cap}}$.
Points with $u(x_c)\leq1$ satisfy the same upper bound.
Thus $u(x_c)\leq1+b_{n,k}K_{\mathrm{cap}}$ for every $x_c$, and we may take
\begin{equation}\label{eq:explicit-universal}
 U_*=1+b_{n,k}K_{\mathrm{cap}},\qquad W_*=C_G(1+U_*)^{A_0}.
\end{equation}
The bound for $W$ in \eqref{eq:universal} follows from
\cref{prop:gradient} and the height bound. Equation \eqref{eq:pointwise-G}
was proved in \cref{prop:gradient}, completing \cref{thm:apriori}.
\end{proof}

\begin{remark}
Here $R_c=u(x_c)-1$ ensures a boundary height of at least one, whereas the cutoff radius in \cref{eq:cutoff-def} is $1+u(x_0)$. All constants used later depend only on $n,k,p$.
\end{remark}

\section{The trace-free Newton tensor and the weighted divergence identity}\label{sec:flux}

The trace-free tensor and weighted divergence identity follow the invariant-tensor viewpoint of \cite{MaWuInvariant,MaWuGradient}. The pointwise estimates hold throughout the admissible cone without a bound on the ellipticity ratio. The integral estimates later require a global bound for $W$.

\subsection{A quantitative algebraic inequality}
Let
\begin{equation}\label{eq:E-D}
 c_k=\frac{n-k}{n},\qquad
 E=Fh-\frac{k}{n}f\Id=c_k f\Id-T_k,\qquad
 \Dk=\tr(Eh).
\end{equation}
The notation $\Dk$ is reserved for this scalar deficit; it is unrelated to
Euclidean differentiation. The trace identities \eqref{eq:newton-traces}
and the Newton-tensor divergence identity \eqref{eq:angle-L} give
\begin{equation}\label{eq:ED-properties}
 \tr E=0,\qquad
 \Div E=c_k\nabla f,\qquad
 \Dk=Q-\frac{k}{n}fH=c_kHf-(k+1)\sigma_{k+1}\geq0.
\end{equation}
To verify $\Dk\geq0$, first suppose $\sigma_{k+1}\leq0$.
Both terms in $c_kHf-(k+1)\sigma_{k+1}$ are then nonnegative.
If $\sigma_{k+1}>0$, then $h\in\Gamma_{k+1}$, and the normalized
Maclaurin inequalities at orders $1,k,k+1$ give $\Dk\geq0$. We use the following quantitative Newton inequality for the $k$-Hessian operator; see also \cite[Lemma 2.4]{DFGQ}.

\begin{lemma}[Quantitative Newton inequality]
\label[lemma]{lem:quantitative}

Let $h\in\Gamma_k$ and $2\leq k<n$. With $f=\sigma_k(h)$ in
\eqref{eq:E-D} and $F=T_{k-1}(h)$, one has
\begin{equation}\label{eq:quantitative}
 E^2\preceq c_k\Dk F.
\end{equation}
In particular, for tangent vectors $X,Y$,
\begin{equation}\label{eq:quantitative-pair}
 |E(X,Y)|^2\leq c_k\Dk\,|X|^2F(Y,Y).
\end{equation}
\end{lemma}
\begin{proof}
Diagonalize $h$ in an orthonormal tangent frame, and denote its eigenvalues
by $\lambda_1,\ldots,\lambda_n$. Since $F$ and $E$ are polynomials in $h$,
they are diagonal in the same frame; write their diagonal entries as $F_i$
and $E_i$. Fix $i$ and let $\lambda|i$ denote the tuple obtained by deleting
$\lambda_i$ from $\lambda$. Set $s_j=\sigma_j(\lambda|i)$, with $s_j=0$ for
$j>n-1$.
Since $F_i=s_{k-1}$ and $E_i=c_k f-s_k$, direct expansion gives
\begin{align*}
 c_k\Dk F_i-E_i^2
 ={}&\frac{(n-k)f}{n^2}
       \big((n-k)s_1s_{k-1}-k(n-1)s_k\big)\\
 &+\frac1n\big(k(n-k-1)s_k^2
                   -(n-k)(k+1)s_{k-1}s_{k+1}\big).
\end{align*}
If $s_k>0$, the tuple $\lambda|i$ belongs to $\Gamma_k$, and the
Maclaurin inequalities imply that the first bracket is nonnegative.
If $s_k\leq0$, the same conclusion follows from $s_1,s_{k-1}>0$.
For $k\leq n-2$, the second bracket is nonnegative by Newton's coefficient
inequality for an arbitrary real $(n-1)$-tuple; no sign assumption on $s_k$ or $s_{k+1}$ is needed.
When $k=n-1$,
both $n-k-1$ and $s_{k+1}=s_n$ vanish, so that bracket is identically
zero. Thus every diagonal entry of $c_k\Dk F-E^2$ is nonnegative. Finally,
$|E(X,Y)|^2\leq|X|^2\ip{E^2Y}{Y}$ proves
\eqref{eq:quantitative-pair}.
\end{proof}

\subsection{The angle variable and the exact divergence identity}
We introduce the angle variable
\begin{equation}\label{eq:q}
 q=2(W-1).
\end{equation}
The function $q$ is smooth and nonnegative, including at points where
$v=0$. The exact
identities
\begin{equation}\label{eq:q-identities}
 \nabla q=-2hv,\qquad
 z=q+\frac{q^2}{4},\qquad Wq=z+\frac{q^2}{4}
\end{equation}
follow from \cref{lem:geometry}. In particular,
\begin{equation}\label{eq:q-ratio}
 \frac qz=\frac2{W+1},\qquad
 \frac{Wq}{z}=\frac{2W}{W+1},
\end{equation}
with the ratios interpreted by continuous extension at $W=1$.
Whenever $W\leq W_*$,
\begin{equation}\label{eq:q-comparable}
 \frac2{W_*+1}z\leq q\leq z,\qquad z\leq W_*^2-1.
\end{equation}
The divergence calculation uses the exact identities in \eqref{eq:q-identities}.

The following vector field is the Lorentzian counterpart of the Euclidean
vector field in \cite[Section 3, equation (3.1) and Proposition 3.3]{DFGQ}.
The quantity $q$ takes the place of the squared Euclidean gradient norm.
The identity for $\nabla q$ in \eqref{eq:q-identities} gives the
cancellation in the divergence calculation.

\begin{proposition}[Weighted divergence identity]\label[proposition]{prop:flux}
Let $u$ be a smooth positive solution of \eqref{eq:main}, with $F,v,f$ as
in \eqref{eq:basic-symbols}, $E,\Dk$ as in \eqref{eq:E-D}, and $q$ as in
\eqref{eq:q}. For any real $M$ and any $\tau>0$, define
\begin{equation}\label{eq:J-def}
 \J_0=-u^MEv,\qquad \J_1=-u^{M-1}qFv,\qquad
 \J=\J_0+\tau\J_1.
\end{equation}
Then
\begin{align}
 \Div\J={}&u^MW\Dk+(2\tau-M)u^{M-1}E(v,v)\notag\\
 &+\tau(1-M)u^{M-2}qF(v,v)
 +a_{M,\tau}u^{M+p-1}z
 +\frac{k\tau}{4}u^{M+p-1}q^2,\label{eq:flux}\\
 a_{M,\tau}:={}&\frac{k(n+2)\tau-(n-k)p}{n}.\notag
\end{align}
\end{proposition}
\begin{proof}
The tensor divergence and Hessian identities give, term by term,
\begin{align*}
 \Div(-u^MEv)
 &=-Mu^{M-1}E(v,v)-u^M\ip{\Div E}{v}
                         -u^M E:\nabla^2u\\
 &=u^MW\Dk-Mu^{M-1}E(v,v)-c_kp\,u^{M+p-1}z.
\end{align*}
For $\J_1$, we use $\Div F=0$, $\nabla q=-2hv$, and
$Fh=E+(k/n)f\Id$ to obtain
\begin{align*}
 \Div(-u^{M-1}qFv)
 ={}&(1-M)u^{M-2}qF(v,v)
       +2u^{M-1}F(hv,v)+kWu^{M+p-1}q\\
 ={}&(1-M)u^{M-2}qF(v,v)+2u^{M-1}E(v,v)\\
 &+\frac{2k}{n}u^{M+p-1}z
       +ku^{M+p-1}\left(z+\frac{q^2}{4}\right).
\end{align*}
Combining the two formulas proves \eqref{eq:flux}.
\end{proof}

\subsection{Coercivity and the exponent interval}
\begin{proposition}\label[proposition]{prop:coercivity}
Let $u$ be a smooth positive admissible solution of \eqref{eq:main},
let $2\leq k<n/2$, and assume $p>0$. Choose
\begin{equation}\label{eq:tau-interval}
 \frac{(n-k)p}{k(n+2)}<\tau<\frac{n-k}{n-2k},
 \qquad M=-\frac{2k\tau}{n-k}.
\end{equation}
Then there exists $c_0=c_0(n,k,p,\tau)>0$, independent of $u$ and
of the evaluation point, such that
\begin{equation}\label{eq:coercive}
 \Div\J\geq c_0\left(
       u^MW\Dk+u^{M-2}qF(v,v)+u^{M+p-1}z\right).
\end{equation}
The interval \eqref{eq:tau-interval} is nonempty exactly when $p<\pstar$.
\end{proposition}
\begin{proof}
Write $b_\tau=\tau(1-M)>0$ and
\[
 X_{\tau}=\sqrt{u^MW\Dk},\qquad
 Y_{\tau}=\sqrt{b_\tau u^{M-2}qF(v,v)}.
\]
By \cref{lem:quantitative} and $z\leq Wq$, the absolute value of the
mixed term in \eqref{eq:flux} is at most
$2\sqrt{\rho_\tau}\,X_{\tau}Y_{\tau}$, where
\begin{equation}\label{eq:rho}
 \rho_\tau=\frac{c_k(2\tau-M)^2}{4b_\tau}.
\end{equation}
At a point where $q=0$, we also have $v=0$.
Hence $E(v,v)=0$ and $Y_{\tau}=0$, so the mixed-term estimate remains
valid at that point.
Direct substitution of $M$ gives
\begin{equation}\label{eq:margin}
 b_\tau-\frac{c_k}{4}(2\tau-M)^2
 =\tau\left(1-\frac{n-2k}{n-k}\tau\right)>0.
\end{equation}
Thus $\rho_\tau<1$, and
$X_{\tau}^2+Y_{\tau}^2-2\sqrt{\rho_\tau}X_{\tau}Y_{\tau}
\geq(1-\sqrt{\rho_\tau})(X_{\tau}^2+Y_{\tau}^2)$.
The lower bound on $\tau$ makes $a_{M,\tau}>0$.
The quadratic-form estimate and the positive $u^{M+p-1}z$ term
then give \eqref{eq:coercive}; the $q^2$ term is nonnegative and may be discarded. Comparing the two endpoints in
\eqref{eq:tau-interval} gives exactly $p<k(n+2)/(n-2k)$.
\end{proof}

\begin{remark}
The upper endpoint of \eqref{eq:tau-interval} comes from the tensor
quadratic form, whereas its lower endpoint comes from the coefficient of
$u^{M+p-1}z$. Both conditions are needed for \eqref{eq:coercive}: positivity of $a_{M,\tau}$ does not control the tensor quadratic form.
For fixed $\tau$, the choice of $M$ in \eqref{eq:tau-interval} maximizes
$\tau(1-M)-c_k(2\tau-M)^2/4$, as the identity
\[
 \tau(1-M)-\frac{c_k}{4}(2\tau-M)^2
 =\tau\left(1-\frac{n-2k}{n-k}\tau\right)
       -\frac{c_k}{4}\left(M+\frac{2k\tau}{n-k}\right)^2
\]
shows. Thus a different constant height weight does not enlarge the interval
obtained by this quadratic absorption. This restriction applies to the
present argument; it neither rules out other identities nor determines
the sharp existence threshold for \eqref{eq:main}.
\end{remark}

\section{A finite weighted Newton-tensor descent}\label{sec:descent}

Throughout this section, $2\leq k<n$, $p>0$, and $u$ is a smooth
positive entire admissible solution of \eqref{eq:main} satisfying
$W\leq W_*$ for a fixed constant $1\leq W_*<\infty$. In the application to
\cref{thm:liouville}, this bound is supplied by \cref{thm:apriori}.

Under this bound, \eqref{eq:metric-intro} and \eqref{eq:measure} yield
\begin{equation}\label{eq:metric-comparable}
 W_*^{-2}\Id\preceq g\preceq\Id,\qquad
 W_*^{-1}\dd x\leq\dmu\leq\dd x.
\end{equation}
Fix a smooth Euclidean cutoff $0\leq\chi_0\leq1$, equal to one on $B_1$
and supported in $B_2$, chosen so that every positive power of $\chi_0$
extends smoothly by zero. Such a cutoff can be chosen to vanish
exponentially at the boundary of its support. We use the rescaled cutoffs $\chi_R(x)=\chi_0(x/R)$ throughout the
integral estimates. Then
\begin{equation}\label{eq:cutoff-grad}
 |\nabla\chi_R|\leq C(W_*,\chi_0)R^{-1}.
\end{equation}
Only the first-derivative bound for the cutoff is needed below.

Write $\chi=\chi_R$, $R\ge1$. The finite Newton-tensor descent removes $\sigma_{k-1}(h)$ from the subcritical cutoff term $R^{-2}\int\sigma_{k-1}(h)u^Mz\chi^{m-2}\dmu$. We state the estimate for general $\delta>-1$, $\delta\ne0$, and $\theta>2k$; \cref{sec:liouville} uses $\delta=-M-1$ and $\theta=m-2$.

The positive-weight integration-by-parts method originates in the
conformal setting of \cite{CGY,Gonzalez} and was adapted to pure
$k$-Hessian inequalities by Ou~\cite{Ou}. Our Lorentzian descent
uses this approach for $\delta>0$ and adapts the separate
negative-weight estimate of \cite[Lemma~3.1]{DFGQ} for
$-1<\delta<0$. The coefficients have different signs in the two ranges; $\theta>2k$ keeps all differentiated cutoff powers positive.

\begin{lemma}[Weighted cutoff estimate]\label[lemma]{lem:descent}
Let $\delta>-1$, $\delta\ne0$, and $\theta>2k$. Define
\begin{equation}\label{eq:Lambda}
 \Lambda_s=\int \sigma_{k-s}(h)\,zq^{s-1}
                    u^{-\delta-s}\chi^\theta\dmu
 \quad(1\leq s\leq k).
\end{equation}
If $\delta>0$, then
\begin{equation}\label{eq:descent-positive}
 \int u^{p-\delta}\chi^\theta\dmu+\sum_{s=1}^k\Lambda_s
 \leq C R^{-2k}\int u^{k-\delta}\chi^{\theta-2k}\dmu.
\end{equation}
If $-1<\delta<0$, then
\begin{equation}\label{eq:descent-negative}
 \sum_{s=1}^k\Lambda_s
 \leq C\int u^{p-\delta}\chi^\theta\dmu
       +CR^{-2k}\int u^{k-\delta}\chi^{\theta-2k}\dmu.
\end{equation}
For fixed $n,k,\delta,\theta,W_*$ and $\chi_0$, the estimates hold with
a constant $C$ independent of $R\geq1$ and of the solution, provided
$W\leq W_*$. The constant does not depend on an upper bound for $H$
or on the ellipticity ratio of $F$.
\end{lemma}
\begin{proof}
We derive the recurrence, estimate its cutoff errors, and absorb them after finitely many steps. For fixed $R$, the compact cutoff support and $u>0$ make all negative-power integrals finite.

\smallskip\noindent
\emph{Step 1: the exact recurrence.}
Introduce
\begin{align}
 \mathcal M_s&=\int T_{k-s}(v,v)\,q^{s-1}
                         u^{-\delta-s}\chi^\theta\dmu,\label{eq:Ms}\\
 \mathcal E_s&=\int T_{k-s}(v,\nabla\chi)\,q^{s-1}
                         u^{1-\delta-s}\chi^{\theta-1}\dmu.\label{eq:Es}
\end{align}
In particular, $\mathcal M_k=\Lambda_k$. For $1\leq s<k$, integrate
the divergence of
$q^su^{-\delta-s}\chi^\theta T_{k-s-1}v$.
With $j=k-s$, the divergence expands into the following four terms:
\begin{align*}
 &-jW\sigma_j q^su^{-\delta-s}\chi^\theta,\\
 &-(\delta+s)q^su^{-\delta-s-1}\chi^\theta T_{j-1}(v,v),\\
 &-2s q^{s-1}u^{-\delta-s}\chi^\theta T_{j-1}(hv,v),\\
 &\theta q^su^{-\delta-s}\chi^{\theta-1}T_{j-1}(v,\nabla\chi).
\end{align*}
Since $T_{j-1}h=\sigma_j\Id-T_j$, rearrangement gives
\begin{equation}\label{eq:recurrence}
 \mathcal M_s=\mathcal A_s+\frac{\delta+s}{2s}\mathcal M_{s+1}
                              -\frac{\theta}{2s}\mathcal E_{s+1},
\end{equation}
where
\begin{equation}\label{eq:As}
 \mathcal A_s=\int a_s(W)\sigma_{k-s}zq^{s-1}
                      u^{-\delta-s}\chi^\theta\dmu,\qquad
 a_s(W)=1+\frac{k-s}{s}\frac W{W+1}.
\end{equation}
Thus
\begin{equation}\label{eq:As-bounds}
 \frac{k+s}{2s}\Lambda_s\leq\mathcal A_s\leq\frac ks\Lambda_s.
\end{equation}
The starting identity, obtained from
$\Div(u^{-\delta}\chi^\theta Fv)$, is
\begin{equation}\label{eq:starting}
 k\int Wu^{p-\delta}\chi^\theta\dmu
                  =-\delta\mathcal M_1+\theta\mathcal E_1.
\end{equation}

Let $m_1=\delta$ and $m_{s+1}=m_s(\delta+s)/(2s)$; explicitly,
\[
 m_s=\frac{\delta(\delta+1)\cdots(\delta+s-1)}{2^{s-1}(s-1)!}.
\]
Because $\delta>-1$, all these coefficients have the same sign as
$\delta$ and none vanishes. Iterating \eqref{eq:recurrence} gives the finite
telescoping identity
\begin{align}
 k\int Wu^{p-\delta}\chi^\theta\dmu
 +\sum_{s=1}^{k-1}m_s\mathcal A_s+m_k\Lambda_k
 =\theta\mathcal E_1+
       \theta\sum_{s=1}^{k-1}\frac{m_s}{2s}\mathcal E_{s+1}.
 \label{eq:telescope}
\end{align}
The signs of the coefficients $m_s$ distinguish the two cases in the
lemma.

\smallskip\noindent
\emph{Step 2: the first cutoff errors.}
For $1\leq s\leq k$, put
\begin{equation}\label{eq:Bs}
 \mathcal B_s=R^{-2s}\int\sigma_{k-s}\,
                 u^{s-\delta}\chi^{\theta-2s}\dmu.
\end{equation}
Positivity and the trace formula for $T_{k-s}$ imply
$|T_{k-s}(v,\nabla\chi)|
\leq C\sigma_{k-s}\sqrt z\,R^{-1}$.
By \eqref{eq:q-comparable}, the absolute value of the integrand of $\mathcal E_s$ is bounded
by a constant times
\[
 R^{-1}\sigma_{k-s}z^{s-1/2}
             u^{1-\delta-s}\chi^{\theta-1}.
\]
The displayed expression is the weighted geometric mean of the
densities of $\Lambda_s$ and $\mathcal B_s$, with weights
$(2s-1)/(2s)$ and $1/(2s)$, after replacing $zq^{s-1}$ by $z^s$
in the density of $\Lambda_s$. Weighted Young's inequality therefore gives
\begin{equation}\label{eq:E-bound}
 |\mathcal E_s|\leq\varepsilon\Lambda_s+C_\varepsilon\mathcal B_s.
\end{equation}

\smallskip\noindent
\emph{Step 3: elimination of the residual curvature integrals.}
For $1\leq s<k$ again write $j=k-s$.
Integrating the divergence of
$u^{s-\delta}\chi^{\theta-2s}T_{j-1}v$ gives
\begin{align}
 j\int W\sigma_j u^{s-\delta}\chi^{\theta-2s}\dmu
 ={}&(s-\delta)\int T_{j-1}(v,v)
                u^{s-\delta-1}\chi^{\theta-2s}\dmu\notag\\
 &+(\theta-2s)\int T_{j-1}(v,\nabla\chi)
                u^{s-\delta}\chi^{\theta-2s-1}\dmu.
 \label{eq:B-start}
\end{align}
When $s-\delta<0$, the first integral has nonpositive coefficient and
may be dropped in an upper bound. In either case, multiplying by
$R^{-2s}$, using $W\geq1$ and $T_{j-1}\preceq C\sigma_{j-1}\Id$,
bounds the right-hand side from above by
\begin{align*}
 C R^{-2s}\int \sigma_{j-1}z\,
                    u^{s-\delta-1}\chi^{\theta-2s}\dmu
 +C R^{-2s-1}\int \sigma_{j-1}\sqrt z\,
                    u^{s-\delta}\chi^{\theta-2s-1}\dmu.
\end{align*}
The first density is the geometric mean of the densities of
$\Lambda_{s+1}$ and $\mathcal B_{s+1}$ with weights
$1/(s+1)$ and $s/(s+1)$. For the second density the weights are
$1/(2s+2)$ and $(2s+1)/(2s+2)$. Replacing powers of $q$ by powers of $z$ changes only the constants,
by \eqref{eq:q-comparable}. Thus, for every $\varepsilon>0$,
\begin{equation}\label{eq:B-step}
 \mathcal B_s\leq\varepsilon\Lambda_{s+1}
                         +C_\varepsilon\mathcal B_{s+1}.
\end{equation}
A finite induction based on \eqref{eq:B-step}, choosing each small
constant before the next, gives, for every $\varepsilon>0$,
\begin{equation}\label{eq:B-all}
 \sum_{s=1}^k\mathcal B_s
 \leq\varepsilon\sum_{s=1}^k\Lambda_s+C_\varepsilon\mathcal B_k.
\end{equation}
The same induction is applied from each remaining starting index.

\smallskip\noindent
\emph{Step 4: absorption.}
Write $\mathcal I_\delta=\int u^{p-\delta}\chi^\theta\dmu$ and
$S=\sum_{s=1}^k\Lambda_s$. Choose
\[
 c_\delta=\min\left\{k,\ |m_k|,\
              \min_{1\leq s<k}\frac{k+s}{2s}|m_s|\right\}>0.
\]
If $\delta>0$, \eqref{eq:As-bounds}, $W\geq1$, and
\eqref{eq:telescope} give
$c_\delta(\mathcal I_\delta+S)\leq C_\delta\sum_s|\mathcal E_s|$.
For $-1<\delta<0$, all $m_s$ are negative, and the exact rearrangement is
\[
 \sum_{s=1}^{k-1}|m_s|\mathcal A_s+|m_k|\Lambda_k
 =k\int Wu^{p-\delta}\chi^\theta\dmu
   -\theta\mathcal E_1
   -\theta\sum_{s=1}^{k-1}\frac{m_s}{2s}\mathcal E_{s+1}.
\]
Taking absolute values only in the error terms yields
$c_\delta S\leq kW_*\mathcal I_\delta+C_\delta\sum_s|\mathcal E_s|$.
Here $C_\delta$ also depends on the fixed parameters $k,\theta$.

In either case, first choose $\varepsilon_1>0$ in \eqref{eq:E-bound} so that
$C_\delta\varepsilon_1\leq c_\delta/4$. Summing that estimate gives
\[
 C_\delta\sum_s|\mathcal E_s|
 \leq\frac{c_\delta}{4}S+C_\delta C_{\varepsilon_1}\sum_s\mathcal B_s.
\]
With this coefficient fixed, choose $\varepsilon_2>0$ in
\eqref{eq:B-all} so that
$C_\delta C_{\varepsilon_1}\varepsilon_2\leq c_\delta/4$.
The error terms are therefore bounded by $c_\delta S/2+C\mathcal B_k$.
Moving $c_\delta S/2$ to the left leaves at least $c_\delta S/2$ there,
together with $c_\delta\mathcal I_\delta$ when $\delta>0$.
The term $C\mathcal B_k$ remains on the right.
Finally,
$\mathcal B_k=R^{-2k}\int u^{k-\delta}\chi^{\theta-2k}\dmu$,
which proves both estimates with constants independent of $R$.
\end{proof}

\begin{remark}\label[remark]{rem:descent}
The telescoping argument requires $\delta\ne0$ because $m_1=\delta$
would otherwise vanish. In the Liouville proof, the weight can be chosen
in an open interval while avoiding this value. The condition
$\delta>-1$ ensures that every factor $\delta+s$ in the recurrence is
positive. Both restrictions are used in the absorption argument.
\end{remark}

\section{The subcritical Liouville theorem}\label{sec:liouville}

For the subcritical part of \cref{thm:liouville}, a solution that is not
identically zero is smooth and positive by \cref{lem:positive}, and
\cref{thm:apriori} gives a uniform bound for $W$.
For $p=k$, we test the mean curvature inequality directly.
For $k<p<\pstar$, the weighted estimates from the preceding sections
will reduce the proof to decay of an integral involving only a power
of the height.

\subsection{The case $p=k$}
\begin{proof}[Proof of the subcritical part of \cref{thm:liouville} for $p=k$]
Assume $p=k$ and $u>0$. The a priori estimate gives $W\leq W_*$, while
\eqref{eq:H-div} and \eqref{eq:H-lower} give
$-\Div_{\R^n}(WDu)=H\geq c u$.
For any $\phi\in C_c^\infty(\R^n)$, use the nonnegative test function
$\phi^2/u$ in this inequality:
\begin{align*}
 c\int\phi^2\dd x
 &\leq\int H\frac{\phi^2}{u}\dd x\\
 &=\int W\left(\frac{2\phi}{u}Du\cdot D\phi
                      -\frac{\phi^2}{u^2}|Du|^2\right)\dd x
 \leq\int W|D\phi|^2\dd x
 \leq W_*\int|D\phi|^2\dd x.
\end{align*}
The penultimate inequality follows from
$|D\phi-(\phi/u)Du|^2\geq0$. Take $\phi(x)=\phi_0(x/R)$ with a fixed
nonzero compactly supported $\phi_0$. The left-hand side is a positive constant times $R^n$, whereas the
right-hand side is a constant times $R^{n-2}$. Letting $R\to\infty$
gives a contradiction.
\end{proof}

\begin{remark}\label[remark]{rem:linear-reaction}
The preceding argument uses \cref{thm:apriori} and \eqref{eq:H-lower} for every $2\le k<n$. Thus $u\equiv0$ at $p=k$ without the dimension restriction $n>2k$; in particular, $k=p=2$ also covers $n=3,4$. The bound $n>2k$ in \cref{thm:liouville} specifies where the finite endpoint $p_*$ is defined.
\end{remark}

\subsection{Localized gradient estimates and reduction to height powers}
The remaining argument for \cref{thm:liouville} concerns $k<p<\pstar$.
We first combine the coercive divergence inequality with the finite
descent. We state the hypotheses of the next two lemmas separately so that the
lemmas also apply whenever a uniform bound for $W$ is available
independently of \cref{thm:apriori}.
\begin{lemma}\label[lemma]{lem:flux-cutoff}
Let $2\leq k<n/2$ and $0<p<\pstar$. Let $u$ be a smooth positive entire solution of \eqref{eq:main} with
$W\leq W_*$. Let $\chi=\chi_R$ be the cutoff in
\eqref{eq:cutoff-grad}, with $R\geq1$, and choose $M,\tau$ as in
\eqref{eq:tau-interval}, with $M\ne-1$. For $m>2k+2$,
\begin{align}
 \int u^{M+p-1}z\chi^m\dmu
 \leq{}&CR^{-2}\int u^{M+p+1}\chi^{m-2}\dmu\notag\\
 &+CR^{-2k-2}\int u^{M+k+1}\chi^{m-2k-2}\dmu.
 \label{eq:gradient-integral}
\end{align}
For fixed $n,k,p,M,\tau,m,W_*$ and $\chi_0$, the estimate holds
with a constant independent of $R\geq1$ and of the solution satisfying
these hypotheses.
\end{lemma}
\begin{proof}
Multiply \eqref{eq:coercive} by $\chi^m$ and integrate.
The divergence term is $-m\int\chi^{m-1}\ip{\J}{\nabla\chi}\dmu$.
For the contribution from $\J_0$, \eqref{eq:quantitative-pair} gives
\begin{align*}
 m u^M|E(v,\nabla\chi)|\chi^{m-1}
 &\leq C u^M\sqrt{\Dk zF(\nabla\chi,\nabla\chi)}\chi^{m-1}\\
 &\leq\varepsilon u^MW\Dk\chi^m
       +C_\varepsilon u^M\frac zW
                  F(\nabla\chi,\nabla\chi)\chi^{m-2}.
\end{align*}
For the contribution from $\J_1$, the Cauchy--Schwarz inequality for
the positive definite form $F$ gives
\begin{align*}
 m\tau u^{M-1}q|F(v,\nabla\chi)|\chi^{m-1}
 \leq\varepsilon u^{M-2}qF(v,v)\chi^m
       +C_\varepsilon u^MqF(\nabla\chi,\nabla\chi)\chi^{m-2}.
\end{align*}
Since $q\leq z$, $W\geq1$, and
$F(\nabla\chi,\nabla\chi)\leq\tr F\,|\nabla\chi|^2
\leq CR^{-2}\sigma_{k-1}$, absorption yields
\begin{equation}\label{eq:pre-descent}
 \int u^{M+p-1}z\chi^m\dmu
      \leq CR^{-2}\int\sigma_{k-1}u^Mz\chi^{m-2}\dmu.
\end{equation}
Apply \cref{lem:descent} with
$\delta=-M-1$ and $\theta=m-2$.
The choice $M<0$ gives $\delta>-1$, and $M\ne-1$ gives $\delta\ne0$.
The integral on the right of \eqref{eq:pre-descent}, before the
factor $CR^{-2}$, is $\Lambda_1$ in \eqref{eq:Lambda}.
For $\delta>0$, substitution of $\delta=-M-1$ and $\theta=m-2$
into \eqref{eq:descent-positive} gives \eqref{eq:gradient-integral}
without the first term on its right-hand side. Adding that nonnegative
term gives the stated estimate. For $\delta<0$, the same substitution
in \eqref{eq:descent-negative} gives \eqref{eq:gradient-integral}
directly.
\end{proof}

\begin{lemma}[Reduction to height powers]\label[lemma]{lem:scalar-integral}
Under the assumptions and with the notation of \cref{lem:flux-cutoff}, define
\begin{equation}\label{eq:Pab}
 P=M+p+\frac pk,\qquad a=M+p+1,\qquad b=M+k+1.
\end{equation}
Then
\begin{equation}\label{eq:scalar-integral}
 I_P(R):=\int u^P\chi^m\dmu
 \leq CR^{-2}\int u^a\chi^{m-2}\dmu
       +CR^{-2k-2}\int u^b\chi^{m-2k-2}\dmu.
\end{equation}
The constant has the same allowed dependence as in \cref{lem:flux-cutoff}.
No additional restriction on the sign of $M+p$ is needed.
\end{lemma}
\begin{proof}
By \eqref{eq:height-L} and \eqref{eq:H-lower},
$-\Delta u=WH\geq c u^{p/k}$. Multiply this inequality by
$u^{M+p}\chi^m$, a compactly supported nonnegative test function.
Integration by parts gives
\[
 cI_P(R)\leq
 (M+p)\int u^{M+p-1}z\chi^m\dmu
 +m\int u^{M+p}\ip{v}{\nabla\chi}\chi^{m-1}\dmu.
\]
For $M+p\leq0$, the first term on the right is nonpositive and may be
dropped. For $M+p>0$, the first term is a fixed positive multiple of the
gradient integral in \eqref{eq:gradient-integral}. In either case Young's inequality and
\eqref{eq:cutoff-grad} bound the absolute value of the last integral by
\[
 C\int u^{M+p-1}z\chi^m\dmu
 +CR^{-2}\int u^{M+p+1}\chi^{m-2}\dmu.
\]
The test function is admissible for every real exponent of $u$, since
$u$ has a positive minimum on the compact support of $\chi$. Applying
\eqref{eq:gradient-integral} proves the assertion.
\end{proof}

\subsection{Choice of the weights and the cutoff power}
We choose $M$ in an open interval that ensures coercivity, the descent hypotheses, and decay after Young's inequality. The interval also allows us to avoid $M=-1$.

For $k<p<\pstar$, set
\begin{equation}\label{eq:M-choice}
 M_0=-\frac{2p}{n+2},\qquad
 M_{\rm low}=\max\left\{
    -\frac{2k}{n-2k},\ -p,\
    \frac{n(p-k)}{2k}-p-\frac pk\right\}.
\end{equation}
Choose any
\begin{equation}\label{eq:M-open}
 M_{\rm low}<M<M_0,\qquad M\ne-1,\qquad
 \tau=-\frac{n-k}{2k}M.
\end{equation}
We now show that $M_{\rm low}<M_0$.

First, $-2k/(n-2k)<M_0$ is equivalent to $p<\pstar$.
Also, $-p<M_0$ because $p>0$ and $n>0$. For the third lower bound, define
$P_0=M_0+p+p/k$. A direct simplification gives
\begin{equation}\label{eq:zeta0}
 n-\frac{2kP_0}{p-k}
 =-\frac{n\big(k(n+2)-(n-2k)p\big)+4p}{(n+2)(p-k)}<0.
\end{equation}
The third lower bound in $M_{\rm low}$ is therefore strictly below $M_0$. Hence the interval in \eqref{eq:M-open} is nonempty, even after removing the single value $-1$. Combining $M>-2k/(n-2k)$ with $M<M_0$ gives
\eqref{eq:tau-interval}. The bound $M>-p$ gives $M+p>0$. This positivity is not needed in
\cref{lem:scalar-integral}, but it ensures $a>1$ when we apply Young's
inequality in the final step. The third restriction gives
\begin{equation}\label{eq:zeta}
 \zeta:=n-\frac{2kP}{p-k}<0.
\end{equation}
Furthermore,
\begin{equation}\label{eq:exponent-diffs}
 a>1,\qquad P-a=\frac{p-k}{k}>0,\qquad
 P-b=\frac{(k+1)(p-k)}{k}>0.
\end{equation}
Choose the cutoff power so that
\begin{equation}\label{eq:cutoff-power-choice}
 m>\max\left\{2k+2,\frac{2kP}{p-k}\right\}.
\end{equation}
This choice makes the cutoff exponents nonnegative in both applications
of Young's inequality below.

The following table records where each parameter restriction is used:
\begin{center}
\begin{tabular}{@{}p{2.58in}p{3.27in}@{}}
\toprule
Condition & Role in the proof\\
\midrule
$p\geq k$ & Ensures $p/k-1\geq0$ in \eqref{eq:W-at-max}.\\[3pt]
$M>-2k/(n-2k)$ & Ensures positivity of the quantity in \eqref{eq:margin}.\\[3pt]
$M<-2p/(n+2)$ & Ensures $a_{M,\tau}>0$ in \eqref{eq:flux}.\\[3pt]
$M>-p$ & Ensures $a=M+p+1>1$ in the final Young estimate; testing with $u^{M+p}\chi^m$ does not require this restriction.\\[3pt]
$M>\frac{n(p-k)}{2k}-p-p/k$ & Ensures $\zeta<0$ in \eqref{eq:zeta}.\\[3pt]
$M<0$, $M\ne-1$ & Gives $\delta=-M-1>-1$ and $\delta\ne0$ in the finite descent.\\[3pt]
$m>\max\{2k+2,2kP/(p-k)\}$ & Ensures nonnegative cutoff powers in both Young estimates.\\
\bottomrule
\end{tabular}
\end{center}
The interval \eqref{eq:M-open} and the choice \eqref{eq:cutoff-power-choice}
satisfy all these requirements when $k<p<\pstar$.
The case $p=k$ was proved separately and involves no factor $(p-k)^{-1}$.
Identity \eqref{eq:zeta0} will also be used in \cref{sec:critical-local}.

\subsection{A zero-$k$-curvature exterior barrier}
The radial comparison below controls negative powers of the height in \cref{sec:liouville,sec:direct-low-dim}. Its constant may depend on the solution but not on the outer radius.

\begin{lemma}[Exterior lower bound]\label[lemma]{lem:lower-barrier}
Let $2\leq k<n/2$, $p>0$, and let $u$ be a positive entire $C^2$ solution of \eqref{eq:main}.
There are constants $c>0$ and $R_0\geq1$, which may depend on $u$, such that
\begin{equation}\label{eq:lower-barrier}
 u(x)\geq c|x|^{-\gamma}\quad(|x|\geq R_0),\qquad
 \gamma=\frac{n-2k}{k}.
\end{equation}
No bound for $W$ is required for this comparison.
\end{lemma}
\begin{proof}
Fix $R_0\geq1$ and let $m_0=\min_{\partial B_{R_0}}u>0$.
Set $\beta_{\mathrm{rad}}=(n-k)/k=1+\gamma$, and choose $\varepsilon_0>0$ so small that
$\int_{R_0}^\infty\varepsilon_0t^{-\beta_{\mathrm{rad}}}\dd t\leq m_0$.
For $S>R_0$, define on $R_0\leq r\leq S$
\begin{equation}\label{eq:radial-barrier}
 y(r)=\varepsilon_0r^{-\beta_{\mathrm{rad}}},\qquad
 V_S(r)=\int_r^S\frac{y(t)}{\sqrt{1+y(t)^2}}\dd t.
\end{equation}
This graph is strictly spacelike. By \eqref{eq:radial-curvatures},
the principal curvatures of \eqref{eq:radial-barrier} are
\[
(-\beta_{\mathrm{rad}}\lambda,\lambda,\ldots,\lambda),\qquad\lambda=y/r>0.
\]
For $1\leq j\leq k$ their symmetric functions are
\begin{align*}
 \sigma_j
 &=\bin{n-1}{j-1}\lambda^j
          \left(\frac{n-j}{j}-\frac{n-k}{k}\right)\\
 &=\bin{n-1}{j-1}\lambda^j\frac{n(k-j)}{jk}.
\end{align*}
These symmetric functions are positive for $j<k$ and vanish for $j=k$.
Thus the principal curvature vector of $V_S$ lies on $\partial\Gamma_k$, and $\sigma_k(A[V_S])=0$.

We have $V_S(R_0)\leq m_0$ and $V_S(S)=0<u$.
If $u-V_S$ had a negative interior minimum in the annulus, equality of
gradients and the Hessian ordering at contact would imply
$\widehat h[u]\preceq\widehat h[V_S]$. Monotonicity of $\sigma_k$ on the closed admissible
cone would give $u^p\leq0$, a contradiction. To justify the closed-cone
comparison explicitly, add $\varepsilon\Id$ to both symmetric matrices.
The perturbed matrices lie in $\Gamma_k$ and have the same ordering.
The monotonicity inequality in the open cone therefore applies. Letting $\varepsilon\downarrow0$
gives the asserted inequality by continuity of $\sigma_k$.
Hence $u\geq V_S$.
Letting $S\to\infty$ gives
\[
 u(x)\geq\int_{|x|}^\infty
       \frac{\varepsilon_0t^{-\beta_{\mathrm{rad}}}}{\sqrt{1+\varepsilon_0^2t^{-2\beta_{\mathrm{rad}}}}}
       \dd t
 \geq\frac{\varepsilon_0}{\gamma\sqrt{1+\varepsilon_0^2R_0^{-2\beta_{\mathrm{rad}}}}}
          |x|^{-\gamma}.
\]
This is the claimed lower bound.
\end{proof}

\subsection{The case $k<p<\pstar$}
\begin{proof}[Proof of the subcritical part of \cref{thm:liouville} for $k<p<\pstar$]
Suppose that the solution is not identically zero. By
\cref{lem:positive}, it is smooth and positive, and \cref{thm:apriori}
provides the uniform bound for $W$ needed for the metric and cutoff
estimates above.
Choose $M,\tau$ by \eqref{eq:M-choice}--\eqref{eq:M-open} and $m$
by \eqref{eq:cutoff-power-choice}. These parameters are fixed before $R$ varies.
We apply \eqref{eq:scalar-integral} in the three cases $b>0$, $b=0$, and $b<0$.

Suppose first that $b>0$. For a term of the form
$R^{-\zeta_0}u^r\chi^{m-\zeta_0}$ with $0<r<P$, write it as
\[
 (u^P\chi^m)^{r/P}\,
       R^{-\zeta_0}\chi^{m(1-r/P)-\zeta_0}.
\]
The exponent of the cutoff is nonnegative if
$m\geq \zeta_0P/(P-r)$. Young's inequality therefore bounds its integral by
$\varepsilon I_P(R)+C_\varepsilon R^{n-\zeta_0P/(P-r)}$, using
$\mu(B_{2R})\leq C R^n$. In the two applications here,
\begin{equation}\label{eq:matching-radii}
 \frac{2P}{P-a}=\frac{(2k+2)P}{P-b}=\frac{2kP}{p-k}.
\end{equation}
By \eqref{eq:cutoff-power-choice}, both resulting cutoff exponents are
nonnegative. Absorbing
the two small multiples of $I_P(R)$ gives
\begin{equation}\label{eq:IP-decay}
 \int_{B_R}u^P\dmu\leq I_P(R)\leq C R^\zeta.
\end{equation}
For any fixed $r>0$ and all $R\geq\max\{1,r\}$, \eqref{eq:IP-decay} gives
\[
0\leq\int_{B_r}u^P\dmu\leq CR^\zeta.
\]
Letting $R\to\infty$ yields $\int_{B_r}u^P\dmu=0$, contrary to $u>0$.

If $b=0$, the second term in \eqref{eq:scalar-integral} is at most
$CR^{n-2k-2}$. Since $b=0$, the second identity in
\eqref{eq:matching-radii} says $\zeta=n-2k-2<0$. Applying Young's inequality to the first term shows that $I_P(R)$
tends to zero. Integration over any fixed ball then contradicts $u>0$,
as in the case $b>0$.

Finally, suppose $b<0$. This requires $n=2k+1$: if $n-2k\geq2$,
then $M>-2k/(n-2k)\geq-k$ and hence $b=M+k+1>1$.
The exterior lower bound in \cref{lem:lower-barrier} gives
$u(x)\geq c|x|^{-1/k}$ outside a fixed ball. Since $b<0$, raising this inequality to the power $b$ reverses its
direction, and hence
\begin{equation}\label{eq:negative-b}
 R^{-2k-2}\int_{B_{2R}}u^b\dmu
 \leq C R^{-2k-2}+C R^{-1-b/k}.
\end{equation}
The integral on the fixed interior ball contributes the first term.
The constants may depend on the fixed solution $u$, but they remain
independent of $R$. Because $M>-2k$, one has $b>1-k$ and
\[
 -1-\frac bk<-\frac1k<0.
\]
Apply Young's inequality to the first term of \eqref{eq:scalar-integral},
and use \eqref{eq:negative-b} for the second. The resulting upper bound
for $I_P(R)$ tends to zero by \eqref{eq:zeta}, again a contradiction.
Together with the argument for $p=k$, these three cases prove \cref{thm:liouville} for $k\le p<\pstar$.
\end{proof}

\section{The critical endpoint: defect rigidity, bubble compactness, and blow-down analysis}\label{sec:critical-local}

In this section $p=p_*$ and $2\le k<n/2$, so the bounds of \cref{thm:apriori} apply to every nonzero solution. At the endpoint, the interval \eqref{eq:tau-interval} closes: its limiting parameters make the quadratic margin and $a_{M,\tau}$ vanish, while the quartic coefficient remains positive. We first derive the exact defect and prove rigidity for $2k<n\le4k+2$. The local lemmas that follow have explicit hypotheses. In the range $n\ge4k+3$, \cref{sec:critical-global} verifies those hypotheses.

\subsection{The endpoint vector field and the exact defect identity}

We retain $v,z,f,F$ from \eqref{eq:basic-symbols}, $c_k,E,\Dk$ from
\eqref{eq:E-D}, and $q$ from \eqref{eq:q}, with $f=u^{p_*}$. Set
\[
d=n-2k,\qquad
\alpha_*=\frac{d}{k+1},\qquad
M_*=-\frac{2k}{d},\qquad\tau_*=\frac{n-k}{d}.
\]
The exponent $\alpha_*$ is reserved for the critical natural scaling;
$\alpha$ continues to denote a transverse frame index when it appears as
a summation index.  Strict admissibility gives $F\succ0$.  For a tangent
vector $X$, we use the notation
\(
|X|_{F^{-1}}^2:=F^{-1}(X,X),
\)
where $F^{-1}$ denotes the inverse of the positive definite Newton tensor
$F=T_{k-1}(h)$.  We may therefore define
\begin{equation}\label{eq:endpoint-K}
	\mathcal K=c_k\mathcal D_k\Id-EF^{-1}E.
\end{equation}
Since $E$ and $F$ are polynomials in $h$, they commute.  The quantitative
Newton inequality  \eqref{eq:quantitative} therefore gives
\(
EF^{-1}E=E^2F^{-1}\preceq c_k\mathcal D_k\Id,
\)
and hence $\mathcal K\succeq0$.  We use the exact angle identities
\eqref{eq:q-identities} below.

\begin{lemma}[Endpoint divergence identity]\label[lemma]{lem:C1}
Define
\begin{equation}\label{eq:Jstar}
J_*=-u^{M_*}\left(Ev+\frac{\tau_*q}{u}Fv\right).
\end{equation}
Then, wherever $q>0$,
\begin{align}
\Div J_*
={}&\frac{|J_*|_{F^{-1}}^2}{c_ku^{M_*}q}
+\frac14u^{M_*}q\Dk{}\label{eq:flux-square}\\
&+\frac{u^{M_*}}{c_kq}\cK(v,v)
+\frac{k\tau_*}{4}u^{M_*+p_*-1}q^2.
\nonumber
\end{align}
Formula \eqref{eq:flux-square} is used on $\{q>0\}$.  At a point where $q=0$ one has $v=0$ and the original divergence identity gives
\[
\cD:=\Div J_*=u^{M_*}\Dk{}\ge0.
\]
The sum of the two $q^{-1}$ terms in \eqref{eq:flux-square} has this continuous limit, although the two terms need not possess separate direction-independent limits.
\end{lemma}

\begin{proof}
Proposition~\ref{prop:flux} gives
\begin{align*}
\Div J
={}&u^MW\Dk{}+(2\tau-M)u^{M-1}E(v,v)\\
&+\tau(1-M)u^{M-2}qF(v,v)
+a_{M,\tau}u^{M+p-1}z
+\frac{k\tau}{4}u^{M+p-1}q^2.
\end{align*}
At the endpoint,
\[
a_{M_*,\tau_*}=0,
\qquad
1-M_*=\frac{\tau_*}{c_k},
\qquad
2\tau_*-M_*=\frac{2\tau_*}{c_k}.
\]
Expanding the first term on the right-hand side of \eqref{eq:flux-square}, using \eqref{eq:endpoint-K}, yields
\begin{align*}
&\frac{u^{M_*}}{c_kq}
\left|Ev+\frac{\tau_*q}{u}Fv\right|_{F^{-1}}^2
+\frac{u^{M_*}}{c_kq}\cK(v,v)\\
&\qquad=\frac{u^{M_*}\Dk{}z}{q}
+\frac{2\tau_*}{c_k}u^{M_*-1}E(v,v)
+\frac{\tau_*^2}{c_k}u^{M_*-2}qF(v,v).
\end{align*}
By \eqref{eq:q-identities},
\[
W-\frac zq=\frac q4.
\]
Hence
\[
u^{M_*}W\Dk{}
=\frac{u^{M_*}\Dk{}z}{q}+\frac14u^{M_*}q\Dk{},
\]
which proves \eqref{eq:flux-square}.  Nonnegativity follows from \eqref{eq:quantitative}.
\end{proof}

The function $D(R)$ records the endpoint defect in $B_R$. It is nondecreasing and is distinct from the Euclidean derivative operator $D$.
A cutoff consequence used later is
\begin{equation}\label{eq:defect-cutoff}
D(R):=\int_{B_R}\cD\,d\mu
\le CR^{-2}\mathcal N(2R),
\end{equation}
where
\begin{equation}\label{eq:Ndef}
\mathcal N(R)=\int_{B_R}\sigma_{k-1}(h)u^{M_*}q\,d\mu.
\end{equation}
Indeed, \eqref{eq:flux-square} gives
$|J_*|_{F^{-1}}^2\le C u^{M_*}q\cD$, and one integrates $\eta^2\cD=\eta^2\Div J_*$ by parts.

\subsection{Direct endpoint rigidity through $4k+2$}\label{sec:direct-low-dim}

We first determine the full dimension range in which direct estimates prove endpoint rigidity.  The matching descent weight changes sign at $n=4k$, but the direct estimates remain effective through $n=4k+2$.  Beginning at $n=4k+3$, the proof uses the concentration argument developed below.

Let $\chi_R$ be the cutoff from \cref{lem:descent}. We display the dependence of $\Lambda_s$ in \eqref{eq:Lambda} on
$\delta$ and $R$. For $\delta>-1$, $\delta\ne0$, and $\theta>2k$, define
\[
\Lambda_s^{(\delta)}(R)=\int \sigma_{k-s}(h)\,zq^{s-1}u^{-\delta-s}\chi_R^\theta\,d\mu,
\qquad 1\le s\le k.
\]
For $\delta>0$ the positive-weight form of that lemma gives
\begin{equation}\label{eq:positive-delta-descent}
\int u^{p_*-\delta}\chi_R^\theta\,d\mu
+\sum_{s=1}^k\Lambda_s^{(\delta)}(R)
\le
CR^{-2k}\int u^{k-\delta}\chi_R^{\theta-2k}\,d\mu.
\end{equation}
Whenever $\delta>0$ and $0<k-\delta<p_*-\delta$, the Young absorption in this estimate must retain the cutoff power.  Put
\[
s_\delta=p_*-\delta,\qquad b_\delta=k-\delta.
\]
Choose $\theta$ so that
\[
\theta\left(1-\frac{b_\delta}{s_\delta}\right)\ge2k.
\]
Then, with $X=u^{s_\delta}\chi_R^\theta$,
\[
CR^{-2k}u^{b_\delta}\chi_R^{\theta-2k}
=CR^{-2k}X^{b_\delta/s_\delta}
\chi_R^{\theta(1-b_\delta/s_\delta)-2k}
\le \frac12X+C R^{-\frac{2ks_\delta}{p_*-k}}.
\]
After integration and absorption of the first term in
\eqref{eq:positive-delta-descent}, this gives
\begin{equation}\label{eq:positive-delta-growth}
\int u^{s_\delta}\chi_R^\theta\,d\mu
+\sum_{s=1}^k\Lambda_s^{(\delta)}(R)
\le C R^{\,n-\frac{2ks_\delta}{p_*-k}}.
\end{equation}
In later applications of this estimate, $\theta$ may be increased as needed.
When
\[
\delta_*=-M_*-1=\frac{4k-n}{n-2k}
\]
is nonzero, $q\le z$ gives
\begin{equation}\label{eq:NbyLambda1}
\mathcal N(R)\le \Lambda_1^{(\delta_*)}(R).
\end{equation}
For $n>4k$, where $\delta_*\in(-1,0)$, the negative-weight form of the same finite descent gives
\begin{equation}\label{eq:negative-delta-direct}
\mathcal N(R)
\le C I_a(2R)+CR^{-2k}\int_{B_{2R}}u^b\,d\mu,
\qquad
 a=M_*+p_*+1,\quad b=M_*+k+1.
\end{equation}
Since $b>0$ and $a-b=p_*-k$, Young's inequality yields
\begin{equation}\label{eq:N-direct-postresonance}
\mathcal N(R)\le C I_a(2R)+CR^\nu,
\qquad \nu=\frac{2k}{k+1}<2.
\end{equation}
Here, for $\sigma>0$, we use the notation
\(
I_\sigma(R):=\int_{B_R}u^\sigma\,d\mu.
\)
We also record the direct use of the quartic endpoint remainder.  Set
\[
\mathcal Q(R)=\int_{B_R}u^{a-2}q^2\,d\mu,
\qquad t=\frac{p_*}{k}>1,
\qquad T=a+2t.
\]
By the last term in \eqref{eq:flux-square} and \eqref{eq:defect-cutoff},
\begin{equation}\label{eq:Q-direct}
\mathcal Q(R)\le CD(R)
\le CR^{-2}\mathcal N(2R).
\end{equation}
The Newton--Maclaurin inequality gives $H\ge c u^t$, hence
$-\Delta_g u=WH\ge c u^t$.  Let $\chi_R$ be supported in $B_{2R}$, equal to one on $B_R$, and satisfy $|D\chi_R|\le C/R$.  Choose an integer
\[
m>\frac{2T}{t-1}.
\]
Testing by $u^{a+t}\chi_R^m$ and integrating by parts gives, with
\[
A_R:=\int u^T\chi_R^m\,d\mu,
\qquad
G_R:=\int u^{a+t-1}z\chi_R^m\,d\mu,
\]
\begin{equation}\label{eq:quartic-test}
A_R\le C G_R+CR^{-2}\int u^{a+t+1}\chi_R^{m-2}\,d\mu.
\end{equation}
Since $z\le Cq$ under the universal spacelikeness bound,
\[
G_R\le C\left(\int u^{a-2}q^2\chi_R^m\,d\mu\right)^{1/2}
        A_R^{1/2}
\le \varepsilon A_R+C_\varepsilon\mathcal Q(2R).
\]
For the last term in \eqref{eq:quartic-test}, put
\[
\beta=\frac{a+t+1}{T}\in(0,1),
\qquad 1-\beta=\frac{t-1}{T}.
\]
Our choice of $m$ gives $m(1-\beta)-2\ge0$, and hence
\[
R^{-2}u^{a+t+1}\chi_R^{m-2}
=R^{-2}(u^T\chi_R^m)^\beta
 \chi_R^{m(1-\beta)-2}
\le \varepsilon u^T\chi_R^m
   +C_\varepsilon R^{-\frac{2T}{t-1}}.
\]
After integration and absorption of the two $\varepsilon A_R$ terms,
\begin{equation}\label{eq:quartic-higher-moment}
I_T(R)
\le C\mathcal Q(2R)
   +CR^{\,n-\frac{2T}{t-1}}.
\end{equation}
We use this estimate in the borderline dimension $n=4k+2$.

\begin{proposition}[Direct critical rigidity through $4k+2$]\label[proposition]{prop:direct-low-dim}
Let $2\le k<n/2$ and $p=p_*$.  If
\[
2k<n\le4k+2,
\]
then every nonnegative entire admissible solution is identically zero.
\end{proposition}

\begin{proof}
We treat five dimension ranges using the exterior barrier, positive-weight descent, a perturbation of the zero weight, and, for $n=4k+2$, the quartic remainder.

\smallskip
\noindent\emph{Case 1: $n=2k+1$.}
Here
\[
\delta_*=2k-1,
\qquad
k-\delta_*=1-k<0.
\]
By \cref{lem:lower-barrier}, for the fixed nonzero solution under contradiction,
\[
u(x)\ge c|x|^{-1/k}\qquad(|x|\ge R_0).
\]
Consequently
\[
u^{1-k}(x)\le C(1+|x|^{(k-1)/k}).
\]
Applying \eqref{eq:positive-delta-descent} with $\delta=2k-1$ and using \eqref{eq:NbyLambda1},
\[
\mathcal N(R)
\le CR^{-2k}\int_{B_{2R}}u^{1-k}\,d\mu
\le CR^{2-1/k}.
\]
Thus $D(R)\le CR^{-1/k}\to0$, and monotonicity gives $D\equiv0$.

\smallskip
\noindent\emph{Case 2: $2k+2\le n<4k$.}
Then $\delta_*>0$ and $k-\delta_*\ge1$.  Using \eqref{eq:positive-delta-growth} with $\delta=\delta_*$ and
\eqref{eq:NbyLambda1} gives
\[
\mathcal N(R)
\le C R^{n-\frac{2k(p_*-\delta_*)}{p_*-k}}.
\]
Since
\begin{equation}\label{eq:nu-direct}
n-\frac{2k(p_*-\delta_*)}{p_*-k}
=\frac{2k}{k+1}=:\nu<2,
\end{equation}
we obtain $\mathcal N(R)\le CR^\nu$ and hence
$D(R)\le CR^{-2/(k+1)}\to0$.

\smallskip
\noindent\emph{Case 3: $n=4k$.}
Now $M_*=-1$ and the matching parameter is $\delta_*=0$, the excluded value in the finite descent.  Fix
\[
0<\varepsilon<\frac1k.
\]
The height bound gives $u^{-1}\le U_*^\varepsilon u^{-1-\varepsilon}$, so
\[
\mathcal N(R)
\le U_*^\varepsilon\Lambda_1^{(\varepsilon)}(R).
\]
Apply \eqref{eq:positive-delta-growth} with $\delta=\varepsilon$.  Since $p_*=2k+1$,
\[
\mathcal N(R)
\le C_\varepsilon R^{\nu_\varepsilon},
\qquad
\nu_\varepsilon
=\frac{2k(1+\varepsilon)}{k+1}<2.
\]
Hence $D(R)\le C_\varepsilon R^{-2+\nu_\varepsilon}\to0$.

\smallskip
\noindent\emph{Case 4: $n=4k+1$.}
Here
\[
\alpha_*=\frac{n-2k}{k+1}=\frac{2k+1}{k+1}<2.
\]
Choose
\[
0<\varepsilon<\frac1{2k+1},
\qquad s=p_*-\varepsilon.
\]
The cutoff-aware estimate \eqref{eq:positive-delta-growth}, used with the small positive weight $\delta=\varepsilon$, gives
\begin{equation}\label{eq:low-moment-4k1}
I_s(R):=\int_{B_R}u^s\,d\mu
\le CR^{\,n-\frac{2ks}{p_*-k}}
=CR^{\alpha_*(1+\varepsilon)}.
\end{equation}
Because $\alpha_*(1+\varepsilon)<2$ and
$a=M_*+p_*+1>p_*>s$, the universal height bound yields
\[
I_a(R)\le U_*^{a-s}I_s(R)\le CR^{2-\eta}
\]
for some $\eta>0$.  Combining this with
\eqref{eq:N-direct-postresonance} gives
$\mathcal N(R)\le CR^{2-\eta'}$ for some $\eta'>0$, and therefore
$D(R)\le CR^{-\eta'}\to0$.

\smallskip
\noindent\emph{Case 5: $n=4k+2$.}
This is the limiting direct dimension.  Here
\[
p_*=2k,\qquad
\alpha_*=2,\qquad
M_*=-\frac{k}{k+1},\qquad
a=2k+\eta_0,
\qquad \eta_0=\frac1{k+1},
\]
and $t=p_*/k=2$.  Choose
\[
\varepsilon=\frac{\eta_0}{8},
\qquad s=2k-\varepsilon.
\]
As in \eqref{eq:low-moment-4k1}, the cutoff-aware estimate
\eqref{eq:positive-delta-growth} gives
\begin{equation}\label{eq:low-moment-borderline}
I_s(R)\le CR^{2+2\varepsilon}.
\end{equation}
Since $a>s$ and $u\le U_*$,
\begin{equation}\label{eq:Ia-rough-borderline}
I_a(R)\le CR^{2+2\varepsilon}.
\end{equation}
By \eqref{eq:N-direct-postresonance} and \eqref{eq:Q-direct},
\[
\mathcal Q(R)
\le CR^{-2}I_a(CR)+CR^{-2/(k+1)}
\le CR^{2\varepsilon}.
\]
Formula \eqref{eq:quartic-higher-moment}, with $T=a+4$, now gives
\begin{equation}\label{eq:high-moment-borderline}
I_{a+4}(R)\le CR^{2\varepsilon},
\end{equation}
because $n-2(a+4)<0$.  Interpolate between
\eqref{eq:low-moment-borderline} and
\eqref{eq:high-moment-borderline}.  Writing
\[
a=(1-\vartheta)s+\vartheta(a+4),
\qquad
\vartheta=\frac{a-s}{a+4-s}
=\frac{\eta_0+\varepsilon}{4+\eta_0+\varepsilon},
\]
we obtain
\[
I_a(R)
\le C R^{2+2\varepsilon-2\vartheta}.
\]
Since $\varepsilon=\eta_0/8$ and $k\ge2$, one has
$\vartheta\ge\eta_0/4$, hence
\begin{equation}\label{eq:Ia-subquadratic-borderline}
I_a(R)
\le CR^{\,2-\frac1{4(k+1)}}.
\end{equation}
Equations \eqref{eq:N-direct-postresonance} and
\eqref{eq:defect-cutoff} therefore imply $D(R)\to0$.

In all five cases, $D\equiv0$.  Returning to \eqref{eq:flux-square}, the quartic term has strictly positive coefficient, so $q\equiv0$.  Hence $Du\equiv0$ and the equation gives $u\equiv0$.
\end{proof}

For the remainder of the endpoint analysis, assume $n\ge4k+3$. The next lemmas identify a bubble and a first-crossing limit under explicit hypotheses. The entire sequences satisfying those hypotheses are constructed in \cref{sec:critical-global}.

\subsection{A scale-invariant logarithmic gradient estimate}

Rescaling introduces a Lorentz parameter. We state a gradient estimate with a constant uniform in this parameter. For $\eps\ge0$ and a function
$U$ with $\eps|DU|^2<1$, set
\[
g_\eps=\Id-\eps DU\otimes DU,\quad
W_\eps=(1-\eps|DU|^2)^{-1/2},\quad
h_\eps=-W_\eps D^2U,\quad A_\eps[U]=g_\eps^{-1}h_\eps.
\]
In matrix products and Newton tensors $h_\eps$ denotes the
$g_\eps$-orthonormal representative of $A_\eps[U]$; its displayed
coordinate formula is the associated bilinear form.
Write $\nabla_\eps$ for the Levi--Civita connection of $g_\eps$,
$z_\eps=|\nabla_\eps U|_{g_\eps}^2$, and
$d\mu_\eps=W_\eps^{-1}dx$. At $\eps=0$ these reduce to the
Euclidean metric, $W_0=1$, and $A_0[U]=-D^2U$. Define
\[
q_\eps=\begin{cases}
2(W_\eps-1)/\eps,&\eps>0,\\
|DU|^2,&\eps=0.
\end{cases}
\]
Then
\begin{equation}\label{eq:scaled-angle}
\nabla_\eps q_\eps=-2h_\eps\nabla_\eps U,
\quad
z_\eps=q_\eps+\frac{\eps q_\eps^2}{4},
\quad
W_\eps=1+\frac{\eps q_\eps}{2}.
\end{equation}
In particular,
\[
\eps z_\eps=W_\eps^2-1,
\qquad
\frac{q_\eps}{z_\eps}=\frac{2}{W_\eps+1},
\]
with the first identity interpreted as $0=0$ when $\eps=0$.
For nonnegative quantities, we write $X\asymp Y$ when
$cY\le X\le CY$ with positive constants uniform in the parameters
under consideration. Thus $q_\eps\asymp z_\eps$ whenever $W_\eps\le W_*$.

\begin{lemma}[Scale-invariant logarithmic gradient bound]\label[lemma]{lem:C2}
Let $U>0$, $\eps\ge0$, and $\kappa>0$.  Assume that in $B_{2r}$,
\[
\sigma_k(A_\eps[U])=\kappa U^p,
\qquad
A_\eps[U]\in\Gamma_k,
\qquad
1\le W_\eps\le W_*,
\]
with $p\ge k$.  Then
\begin{equation}\label{eq:loggrad}
\sup_{B_r}|D\log U|
\le C\left[r^{-1}+\kappa^{1/(2k)}
\left(\sup_{B_{2r}}U\right)^{(p-k)/(2k)}\right],
\end{equation}
where $C=C(n,k,p,W_*)$ is independent of $\eps,\kappa,r$.
\end{lemma}

\begin{proof}
In this proof abbreviate $\nabla=\nabla_\eps$ and $z=z_\eps$, and set
\[
v=\nabla U,\qquad z=|v|_{g_\eps}^2,
\qquad F=T_{k-1}(h_\eps),\qquad L_F\varphi=F^{ij}\nabla_i\nabla_j\varphi.
\]
The scaled graph identities are
\begin{equation}\label{eq:C2-basic-identities}
\nabla^2U=-W_\eps h_\eps,
\qquad
L_FU=-kW_\eps\kappa U^p,
\qquad
L_FW_\eps=\eps\bigl(W_\eps Q-p\kappa U^{p-1}z\bigr),
\end{equation}
where $Q=\tr(Fh_\eps^2)$.  Hence, with $q=q_\eps$,
\begin{equation}\label{eq:C2-Lq}
L_Fq=2W_\eps Q-2p\kappa U^{p-1}z,
\qquad
\nabla q=-2h_\eps v.
\end{equation}
All constants below depend only on $n,k,p,W_*$.  The bound $W_\eps\le W_*$ makes the Euclidean and induced norms uniformly equivalent; in particular all cutoff estimates below are uniform in $\eps$.

Let
\[
\eta=1-\frac{|x|^2}{4r^2},
\qquad
G=\eta^2\frac q{U^2}.
\]
We work in $B_{2r}$, where $0<\eta\le1$. If $G\equiv0$ in $B_{2r}$, then $q\equiv0$ and
$DU\equiv0$ there; the estimate is immediate.  Otherwise let $x_0$ be an interior maximum of $G$.  All quantities in the next computation are evaluated at $x_0$.  Put
\[
X=r\eta\frac{\sqrt z}{U}.
\]
If $X\le X_0$ for a fixed, sufficiently large universal constant $X_0$, then at the maximum point
\[
G(x_0)=\eta(x_0)^2\frac{q(x_0)}{U(x_0)^2}\le C r^{-2},
\]
because $q\asymp z$.  Since $G(x)\le G(x_0)$ and $\eta\ge3/4$ on $B_r$, the estimate follows in this case. We henceforth assume $X>X_0$.

The first derivative equation $\nabla\log G=0$ gives
\begin{equation}\label{eq:C2-first-order}
\frac{h_\eps v}{q}=A-\frac vU,
\qquad A=\frac{\nabla\eta}{\eta}.
\end{equation}
Choose an induced orthonormal frame with $e_1=v/\sqrt z$.  Then
\begin{equation}\label{eq:C2-h-column}
h_\eps e_1=\frac q{\sqrt z}A-\frac qUe_1.
\end{equation}
Writing $h_\eps$ in block form relative to $e_1\oplus e_1^\perp$,
\[
h_\eps=\begin{pmatrix}a&b^T\\ b&B\end{pmatrix},
\]
we have
\[
a=-\frac qU+\frac q{\sqrt z}A_1,
\qquad
b=\frac q{\sqrt z}A_\perp.
\]
Since $|A|\le C/(r\eta)$ and $q/z=2/(W_\eps+1)$, increasing $X_0$ if necessary gives
\begin{equation}\label{eq:C2-block-hyp}
a\le-\frac q{2U}<0,
\qquad
|b|\le\frac{|a|}{4\sqrt{n-k}}.
\end{equation}
Lemma~\ref{lem:block} therefore applies.  If
\[
\phi=F(e_1,e_1),
\]
then
\begin{equation}\label{eq:C2-block-consequences}
\tr F\le C\phi,
\qquad
\phi\ge c(\kappa U^p)^{(k-1)/k},
\end{equation}
and
\begin{equation}\label{eq:C2-transverse}
\frac{2W_\eps}{q}
\sum_{\alpha=2}^nF(h_\eps e_\alpha,h_\eps e_\alpha)
\ge c\phi\frac z{U^2}.
\end{equation}
Indeed $|a|\ge q/(2U)$ and $q\asymp z$ under the uniform bound for $W_\eps$.

We next compute the second derivative inequality.  From \eqref{eq:C2-Lq},
\begin{align}
0\ge L_F\log G
={}&\frac{2L_F\eta}{\eta}-2F(A,A)
+\frac{2W_\eps}{q}Q
-\frac{4}{q^2}F(h_\eps v,h_\eps v)\nonumber\\
&+\frac{2kW_\eps\kappa U^p}{U}
+\frac{2}{U^2}F(v,v)
-\frac{2p\kappa U^{p-1}z}{q}.
\label{eq:C2-logG}
\end{align}
The two terms containing $\kappa$ combine into
\begin{equation}\label{eq:C2-reaction}
2\kappa U^{p-1}
\left(kW_\eps-p\frac zq\right),
\end{equation}
which is bounded below by $-C\kappa U^{p-1}$ because $1\le W_\eps\le W_*$ and $z/q=(W_\eps+1)/2$.

Separate in \eqref{eq:C2-logG} the $e_1$-column of $Q$.  Since $v=\sqrt z\,e_1$, the longitudinal part is
\[
-\frac2qF(h_\eps e_1,h_\eps e_1)
+\frac{2z}{U^2}\phi-2F(A,A).
\]
Substituting \eqref{eq:C2-h-column} and using $z-q=\eps q^2/4$ yields the exact identity
\begin{align}
&-\frac2qF(h_\eps e_1,h_\eps e_1)
+\frac{2z}{U^2}\phi-2F(A,A)\nonumber\\
&\qquad=
\frac{\eps q^2}{2U^2}\phi
+\frac{4q}{\sqrt z\,U}F(A,e_1)
-\left(2+\frac{2q}{z}\right)F(A,A).
\label{eq:C2-longitudinal}
\end{align}
The first term is nonnegative.  By Cauchy--Schwarz for the positive tensor $F$, \eqref{eq:C2-block-consequences}, and $|A|\le C/(r\eta)$,
\begin{equation}\label{eq:C2-Aerrors}
\frac{4q}{\sqrt z\,U}F(A,e_1)
-\left(2+\frac{2q}{z}\right)F(A,A)
\ge
-\frac c4\phi\frac z{U^2}
-\frac{C\phi}{r^2\eta^2}.
\end{equation}

It remains to estimate the cutoff Hessian.  Since
\[
D^2\eta=-\frac1{2r^2}\Id,
\]
the Christoffel symbols of $g_\eps$ give
\begin{equation}\label{eq:C2-cutoff-Hessian}
\nabla_{g_\eps}^2\eta
=D^2\eta-\eps W_\eps(D\eta\cdot DU)h_\eps.
\end{equation}
Using $F:h_\eps=k\kappa U^p$, \eqref{eq:C2-block-consequences}, and $|D\eta|\le C/r$,
\begin{equation}\label{eq:C2-cutoff-LF}
\frac{2L_F\eta}{\eta}
\ge
-\frac{C\phi}{r^2\eta^2}
-\frac{C\eps\kappa U^p|DU|}{r\eta}.
\end{equation}
The last term is uniform in $\eps$.  Indeed, after division by $\phi$ and use of \eqref{eq:C2-block-consequences}, it is bounded by
\[
C\kappa^{1/k}U^{p/k-1}
\frac{\eps z}{X},
\qquad
X=r\eta\frac{\sqrt z}{U},
\]
and $\eps z=W_\eps^2-1\le W_*^2-1$.  Since $X>X_0$, this contributes at most
\begin{equation}\label{eq:C2-scaled-error}
C\phi\,\kappa^{1/k}U^{p/k-1}.
\end{equation}

Combining \eqref{eq:C2-transverse}--\eqref{eq:C2-scaled-error} with the reaction estimate gives
\[
0\ge
c\phi\frac z{U^2}
-\frac{C\phi}{r^2\eta^2}
-C\kappa U^{p-1}
-C\phi\kappa^{1/k}U^{p/k-1}.
\]
Finally
\[
\frac{\kappa U^{p-1}}{\phi}
\le C\kappa^{1/k}U^{p/k-1}
\]
by \eqref{eq:C2-block-consequences}.  Thus
\begin{equation}\label{eq:C2-finalG}
\eta^2\frac z{U^2}
\le C\left(r^{-2}+\kappa^{1/k}\sup_{B_{2r}}U^{p/k-1}\right).
\end{equation}
Since $G\le G(x_0)$, $\eta\ge3/4$ on $B_r$, $q\asymp z$ under
$W_\eps\le W_*$, and $|DU|^2\le z$, \eqref{eq:C2-finalG} implies
\[
\sup_{B_r}|D\log U|^2
\le C\left[r^{-2}+\kappa^{1/k}
\left(\sup_{B_{2r}}U\right)^{(p-k)/k}\right].
\]
Taking square roots proves \eqref{eq:loggrad}.
\end{proof}

\subsection{Small-defect compactness and bubble rigidity}

To obtain compactness, we need a quantitative estimate that rules out Hessian concentration when the endpoint defect is small. The nonnegativity of $\Dk{}$ alone is insufficient.

\begin{lemma}[Quantitative trace and transverse deficits]\label[lemma]{lem:quant-deficits}
For $h\in\Gamma_k$ and $f=\sigma_k(h)>0$,
\begin{equation}\label{eq:Dtracefree}
\Dk{}\ge \frac{n-k}{n-1}\frac fH
\left|h-\frac Hn\Id\right|^2.
\end{equation}
Moreover, for any unit vector $e$, if
\[
h=\begin{pmatrix}a&b^T\\b&B\end{pmatrix}
\]
relative to $e\oplus e^\perp$, then
\begin{equation}\label{eq:Ktrans}
\cK(e,e)\ge C_{n,k}\frac fH
\left[
\left|B-\frac{\tr B}{n-1}\Id\right|^2
+\frac{n-2}{n-1}|b|^2
\right].
\end{equation}
\end{lemma}

\begin{proof}
We first prove \eqref{eq:Dtracefree}.  If $\sigma_{k+1}(h)\le0$, then
\[
\Dk{}=c_kHf-(k+1)\sigma_{k+1}(h)\ge c_kHf.
\]
Because $k\ge2$ and $h\in\Gamma_k\subset\Gamma_2$,
\[
|h|^2<H^2,
\qquad
\left|h-\frac Hn\Id\right|^2
<\frac{n-1}{n}H^2,
\]
which gives \eqref{eq:Dtracefree}.  If $\sigma_{k+1}(h)>0$, then $h\in\Gamma_{k+1}$.  Put
\[
S_j=\frac{\sigma_j(h)}{\binom nj}.
\]
The normalized Newton quotients are decreasing, so
\[
\frac{S_{k+1}}{S_k}\le\frac{S_2}{S_1}
=\frac{H^2-|h|^2}{(n-1)H}.
\]
Using
$(k+1)\binom n{k+1}=(n-k)\binom nk$ gives
\begin{align*}
\Dk{}
&\ge \frac{n-k}{n}Hf
-(n-k)f\frac{H^2-|h|^2}{(n-1)H}\\
&=\frac{n-k}{n-1}\frac fH
\left(|h|^2-\frac{H^2}{n}\right),
\end{align*}
which is \eqref{eq:Dtracefree}.

We next prove \eqref{eq:Ktrans}.  Diagonalize $h$ in an orthonormal eigenframe with eigenvalues $\lambda_1,\ldots,\lambda_n$.  For fixed $i$ write
\[
s_j=\sigma_j(\lambda|i),
\qquad F_i=s_{k-1},
\qquad E_i=c_kf-s_k.
\]
The diagonal expansion underlying the quantitative Newton inequality gives
\begin{align}
(c_k\Dk{}F_i-E_i^2)
={}&\frac{(n-k)f}{n^2}
\left[(n-k)s_1s_{k-1}-k(n-1)s_k\right]\nonumber\\
&+\frac1n\left[k(n-k-1)s_k^2-(n-k)(k+1)s_{k-1}s_{k+1}\right].
\label{eq:Kdiag-exp}
\end{align}
The second bracket is nonnegative.  Set
\[
V_i=\sum_{j\ne i}\lambda_j^2-\frac{s_1^2}{n-1}.
\]
If $s_k>0$, the deleted $(n-1)$-tuple belongs to $\Gamma_k$, and monotonicity of the normalized Newton quotients in dimension $n-1$ gives
\begin{equation}\label{eq:deleted-variance}
(n-k)s_1-k(n-1)\frac{s_k}{s_{k-1}}
\ge
\frac{(n-k)(n-1)}{n-2}\frac{V_i}{s_1}.
\end{equation}
If $s_k\le0$ and $k\ge3$, the same inequality follows from the right-hand side bound
$V_i\le\frac{n-2}{n-1}s_1^2$, since the deleted tuple lies in $\Gamma_{k-1}\subset\Gamma_2$.  For $k=2$, \eqref{eq:deleted-variance} is the exact identity
\[
(n-2)s_1-2(n-1)\frac{s_2}{s_1}
=(n-1)\frac{V_i}{s_1}.
\]
Since $s_1=H-\lambda_i>0$ and $|\lambda_i|<H$, one has $s_1<2H$.  Dividing \eqref{eq:Kdiag-exp} by $F_i=s_{k-1}$ therefore yields
\begin{equation}\label{eq:Kii-var}
\cK_{ii}
\ge
\frac{(n-k)^2(n-1)}{2n^2(n-2)}\frac fH V_i.
\end{equation}

Let $e=\sum_i\theta_i e_i$ be an arbitrary unit vector.  Since $\cK$ is diagonal in the eigenframe,
\[
\cK(e,e)\ge C_{n,k}\frac fH\sum_i\theta_i^2V_i.
\]
Writing $h$ relative to $e\oplus e^\perp$ as in the statement, a direct calculation gives
\[
\sum_i\theta_i^2V_i
=
\left|B-\frac{\tr B}{n-1}\Id\right|^2
+\frac{n-2}{n-1}|b|^2.
\]
This proves \eqref{eq:Ktrans}.  In particular, no assumption that $e$ is a principal direction is used.
\end{proof}

The next lemma uses the algebraic deficits to control points where the trace is large.

\begin{lemma}[Large trace costs endpoint defect]\label[lemma]{lem:highH-defect}
Fix $0<f_0\le f\le f_1$ and $0\le s\le S_0$.  For $h\in\Gamma_k$, $|e|=1$, define
\[
\mathfrak G(h,e,s)
=c_k\Dk{}+2sE(e,e)+s^2F(e,e).
\]
Then there exist $H_0,c_0>0$, depending only on $n,k,f_0,f_1,S_0$, such that
\[
H\ge H_0\quad\Longrightarrow\quad
\mathfrak G(h,e,s)\ge c_0H.
\]
\end{lemma}

\begin{proof}
The algebraic identity
\begin{equation}\label{eq:Gsquare}
\mathfrak G
=|(E+sF)e|_{F^{-1}}^2+\cK(e,e)
\end{equation}
shows $\mathfrak G\ge0$.  Suppose the conclusion fails.  Then there are $H_j\to\infty$ with
\[
\frac{\mathfrak G_j}{H_j}\to0.
\]
By \eqref{eq:Ktrans}, after passing to a subsequence and using $|h_j|<H_j$, the normalized matrices relative to $e_j\oplus e_j^\perp$ converge to
\[
\diag(a,\lambda,\ldots,\lambda),
\qquad
a+(n-1)\lambda=1.
\]
The transverse restriction $B_j$ belongs to $\Gamma_{k-1}$, so
$\tr B_j>0$.  Dividing by $H_j$ and passing to the limit gives
$(n-1)\lambda\ge0$, and hence $\lambda\ge0$.
Since $f_j/H_j^k\to0$, the leading $k$-th symmetric function vanishes:
\begin{equation}\label{eq:rankone-limit}
\lambda^{k-1}\left(a+\frac{n-k}{k}\lambda\right)=0.
\end{equation}
If $\lambda>0$, then $a=-(n-k)\lambda/k$ and
$H_j^{-1}B_j\to\lambda\Id$. Consequently
$F_j(e_j,e_j)=\sigma_{k-1}(B_j)\asymp H_j^{k-1}$, whereas
$-E_j(e_j,e_j)=\sigma_k(B_j)-c_kf_j\asymp H_j^k$.
Because $0\le s_j\le S_0$, Cauchy--Schwarz for $F_j$ gives
\[
\mathfrak G_j\ge
\left|(E_j+s_jF_j)e_j\right|_{F_j^{-1}}^2
\ge\frac{|E_j(e_j,e_j)+s_jF_j(e_j,e_j)|^2}{F_j(e_j,e_j)}
\ge cH_j^{k+1},
\]
contradicting $\mathfrak G_j/H_j\to0$.

It remains to consider $\lambda=0$, hence $a=1$ and $\tr B_j=o(H_j)$.  If $\sigma_k(B_j)\le0$, then
\[
E(e_j,e_j)=c_kf_j-\sigma_k(B_j)\ge c_kf_0.
\]
Moreover $s_j\ge0$ and $F(e_j,e_j)>0$, while \eqref{eq:Dtracefree} gives $\Dk{}\ge cH_j$.  Hence directly from the definition,
\[
\mathfrak G_j=c_k\mathcal D_{k,j}+2s_jE_j(e_j,e_j)+s_j^2F_j(e_j,e_j)
\ge c_k\mathcal D_{k,j}\ge cH_j,
\]
contradicting $\mathfrak G_j/H_j\to0$.

Assume therefore $\sigma_k(B_j)>0$, so $B_j\in\Gamma_k$, and set
\[
\phi_j=\sigma_{k-1}(B_j),
\qquad
r_j=\frac{\sigma_k(B_j)}{\phi_j}.
\]
The block identities and ellipticity of the quotient $\sigma_k/\sigma_{k-1}$ give
\[
T_{k-1}(B_j)-r_jT_{k-2}(B_j)\succeq0.
\]
Moreover
\begin{align*}
r_jf_j-\sigma_{k+1}(h_j)
={}&r_j^2\phi_j-\sigma_{k+1}(B_j)\\
&+b_j^T\bigl(T_{k-1}(B_j)-r_jT_{k-2}(B_j)\bigr)b_j.
\end{align*}
The normalized Newton quotient inequality in dimension $n-1$ yields
\[
\sigma_{k+1}(B_j)
\le\frac{k(n-k-1)}{(k+1)(n-k)}r_j^2\phi_j,
\]
and hence
\[
r_jf_j-\sigma_{k+1}(h_j)
\ge\frac{n}{(k+1)(n-k)}r_j^2\phi_j
=\frac{1}{c_k(k+1)}r_j^2\phi_j.
\]
Substitution in the exact expression for $\mathfrak G_j$ gives
\begin{equation}\label{eq:G-r-bound}
\mathfrak G_j
\ge
\phi_j(r_j-s_j)^2
+c_k^2H_jf_j+2s_jc_kf_j-c_k(k+1)r_jf_j.
\end{equation}  Finally
\[
r_j\le\frac{n-k}{k(n-1)}\tr B_j=o(H_j).
\]
The last two terms in \eqref{eq:G-r-bound} are therefore $o(H_j)$, whereas $c_k^2H_jf_j\ge cH_j$.  This contradicts $\mathfrak G_j/H_j\to0$.
\end{proof}

\begin{lemma}[Small-defect compactness and global identification of the bubble]\label[lemma]{lem:C3}
Fix $\eps_0>0$, $K_0\ge1$, and $W_*>1$. Define
\begin{equation}\label{eq:L0turn}
R_{\rm turn}^2=(n-2k)\binom nk^{1/k},\qquad
L_0>4(R_{\rm turn}+1),
\end{equation}
where $L_0$ is any fixed constant satisfying the displayed bound.
Let $U_j$ be positive critical solutions on balls $B_{R_j}$,
$R_j\to\infty$, with Lorentz parameters
\[
0\le\eps_j\le\eps_0,
\qquad
\sigma_k(A_{\eps_j}[U_j])=U_j^{p_*},
\qquad
A_{\eps_j}[U_j]\in\Gamma_k,
\]
and
\[
U_j(0)=1,\qquad 0<U_j\le K_0,\qquad
W_j=(1-\eps_j|DU_j|^2)^{-1/2}\le W_*.
\]
Assume that there are numbers $\Lambda_j\to\infty$ such that, on every fixed ball and for all sufficiently large $j$,
\begin{equation}\label{eq:C3-contact-centering}
U_j(x)^{-1/\alpha_*}\ge 1-\frac{|x|}{100\Lambda_j},
\end{equation}
and suppose that, with $d\mu_j=W_j^{-1}dx$ and the scaled
endpoint defect $\cD_j$ specified in
\eqref{eq:C3-scaled-defect} below,
\[
\int_{B_{L_0}}\cD_j\,d\mu_j\to0.
\]
Then, after passing to a subsequence, $\eps_j\to0$ and
\begin{equation}\label{eq:C3global}
U_j\longrightarrow\cB
\quad\text{locally uniformly in }\mathbb R^n,
\end{equation}
where
\begin{equation}\label{eq:bubble}
\cB(x)=
\left(1+\frac{k|x|^2}{(n-2k)\binom nk^{1/k}}
\right)^{-(n-2k)/(2k)}.
\end{equation}
Moreover, on $B_{L_0}$ the high-Hessian region carries vanishing $L^1$ Hessian mass and, after extraction,
\[
U_j\to\cB\quad\text{strongly in }W^{2,1}_{\rm loc}(B_{L_0}).
\]
\end{lemma}

The contact-envelope construction in \cref{sec:critical-global} will
produce the centering inequality \eqref{eq:C3-contact-centering}. The lemma
itself uses only the explicit hypotheses above.

\begin{proof}
We first control the region of large Hessian, identify the Euclidean bubble on a fixed ball, and then extend that identification using a diagonal limit and the scaling Jacobi field.

We write the scaled angle variable as $q_j=q_{\eps_j}$ and put
\[
\mathfrak s_j=\frac{\tau_*q_j}{U_j},
\qquad
e_j=\frac{\nabla U_j}{|\nabla U_j|}
\quad\text{when }|\nabla U_j|>0.
\]
At points where $\nabla U_j=0$, choose $e_j$ arbitrarily;
since $\mathfrak s_j=0$, neither $\mathfrak G_j$ below nor
$\mathfrak s_j e_j\otimes e_j$ depends on this choice. The exact endpoint algebra in the scaled geometry gives
\begin{equation}\label{eq:C3-scaled-defect}
\cD_j
=
\frac{U_j^{M_*}z_j}{c_kq_j}\,
\mathfrak G_j
+\frac{\eps_j}{4}U_j^{M_*}q_j\mathcal D_{k,j}
+\frac{\eps_jk\tau_*}{4}U_j^{M_*+p_*-1}q_j^2,
\end{equation}
where
\[
\mathfrak G_j
=c_k\mathcal D_{k,j}
+2\mathfrak s_jE_j(e_j,e_j)
+\mathfrak s_j^2F_j(e_j,e_j).
\]
At $q_j=0$ the first term is understood through $z_j/q_j=(W_j+1)/2$ and $\mathfrak s_j=0$.  Formula \eqref{eq:C3-scaled-defect} follows from the same expansion as Lemma~\ref{lem:C1}, with
\[
W_jq_j=z_j+\frac{\eps_jq_j^2}{4}.
\]

Fix $K\Subset B_{L_0}$.  Lemma~\ref{lem:C2}, together with $U_j(0)=1$ and $U_j\le K_0$, gives a Harnack lower bound
\begin{equation}\label{eq:C3-Ulower}
0<c_K\le U_j\le K_0
\quad\text{on }K,
\end{equation}
and a uniform bound for $|DU_j|$.  Consequently $f_j=U_j^{p_*}$ is bounded above and below by positive constants and $\mathfrak s_j$ is uniformly bounded on $K$.  Since $z_j/q_j=(W_j+1)/2$ and $U_j^{M_*}$ is bounded above and below, Lemma~\ref{lem:highH-defect} and \eqref{eq:C3-scaled-defect} imply
\begin{equation}\label{eq:C3-highH-cost}
H_j\ge H_0\quad\Longrightarrow\quad
\cD_j\ge c_KH_j.
\end{equation}
Thus
\begin{equation}\label{eq:C3-highH-vanish}
\int_{K\cap\{H_j>H_0\}}H_j\,d\mu_j\longrightarrow0.
\end{equation}

We also have a uniform local $L^1$ bound for $H_j$.  Indeed the scaled trace identity is
\[
H_j=-\Div(W_jDU_j),
\]
so for a cutoff $\zeta\equiv1$ on $K$,
\[
\int_KH_j\,dx
\le\int W_j|DU_j||D\zeta|\,dx\le C_K.
\]
Since $h_j\in\Gamma_k\subset\Gamma_2$, $|h_j|<H_j$; the uniformly controlled graph metric therefore gives
\begin{equation}\label{eq:C3-Hessian-L1}
\int_K|D^2U_j|\,dx\le C_K.
\end{equation}
Hence, after extraction, $U_j\to U_\infty$ locally uniformly.  Since \eqref{eq:C3-Hessian-L1} is a local $BV$ bound for the gradients, the compactness theorem for $BV$ functions also gives
$DU_j\to DU_\infty$ in $L^1_{\rm loc}$.

On $K\cap\{H_j\le H_0\}$ the triples $(h_j,e_j,\mathfrak s_j)$ range in a fixed compact set.  Let $\mathcal S_0$ denote the zero set of $\mathfrak G$ in this compact set.  Since $\mathfrak G\ge0$ is continuous, for every $\eta>0$ there exists $c_\eta>0$ such that
\[
\operatorname{dist}\bigl((h,e,\mathfrak s),\mathcal S_0\bigr)\ge\eta
\quad\Longrightarrow\quad
\mathfrak G(h,e,\mathfrak s)\ge c_\eta.
\]
The coefficient $U_j^{M_*}z_j/(c_kq_j)$ in the first term of
\eqref{eq:C3-scaled-defect} is bounded above and below by positive constants on $K$, with its value at $q_j=0$ understood through
$z_j/q_j=(W_j+1)/2$.  Since
$\int_{B_{L_0}}\cD_j\,d\mu_j\to0$, the first term in
\eqref{eq:C3-scaled-defect} therefore implies
\[
\operatorname{dist}\bigl((h_j,e_j,\mathfrak s_j),\mathcal S_0\bigr)
\longrightarrow0
\quad\text{in measure on }K\cap\{H_j\le H_0\}.
\]
Because the triples $(h_j,e_j,\mathfrak s_j)$  are uniformly bounded there, the same convergence holds in $L^1$ after composition with bounded continuous functions vanishing on $\mathcal S_0$.  We now record $\mathcal S_0$ explicitly.  If $\mathfrak G(h,e,\mathfrak s)=0$, then \eqref{eq:Gsquare} gives both
\[
(E+\mathfrak s F)e=0,
\qquad
\cK(e,e)=0.
\]
By Lemma~\ref{lem:quant-deficits}, the second equality forces, relative to $e\oplus e^\perp$,
\[
B=\lambda\Id,\qquad b=0.
\]
For such a block matrix, the $e$-component of the first equality is
\[
c_k(a-\lambda)+\mathfrak s=0,
\]
and therefore
\begin{equation}\label{eq:C3-zero-set}
h=\lambda\Id-\frac{\mathfrak s}{c_k}e\otimes e.
\end{equation}
Substitution into $\sigma_k(h)=f$ gives
\begin{equation}\label{eq:C3-lambda-equation}
\binom nk\lambda^{k-1}
\left(\lambda-\frac{k\mathfrak s}{n-k}\right)=f.
\end{equation}
On the admissible branch $\lambda>k\mathfrak s/(n-k)$ the left-hand side is strictly increasing, so $\lambda$ is the unique root in that range.  Thus $(f,\mathfrak s,e)$ uniquely determine the zero-set matrix on the admissible branch. The compact-range argument then identifies its limit in measure.

The tensor $\mathfrak s e\otimes e$ has the continuous representation
\[
\mathfrak s e\otimes e
=\frac{\tau_*}{U}\frac qz\,v\otimes v,
\]
with value zero at $v=0$.  Thus the corresponding zero-set matrix in \eqref{eq:C3-zero-set} depends continuously on $(U,DU)$ on the compact range under consideration.  Let $\mathcal M_{\eps}(U,DU)$ denote the Euclidean Hessian associated with this zero-set matrix at Lorentz parameter $\eps$.  On $K\cap\{H_j\le H_0\}$, the preceding compact zero-set argument gives
\[
D^2U_j-\mathcal M_{\eps_j}(U_j,DU_j)\longrightarrow0
\quad\text{in }L^1.
\]
On $K\cap\{H_j>H_0\}$, \eqref{eq:C3-highH-vanish} and the graph-metric bounds give
$\int |D^2U_j|\,dx\to0$.  The map $\mathcal M_{\eps}(U,DU)$ is uniformly bounded on the compact range of $(\eps,U,DU)$, while
$|\{H_j>H_0\}\cap K|\le H_0^{-1}\int_{K\cap\{H_j>H_0\}}H_j\,dx\to0$.
Consequently its $L^1$ contribution on the high-Hessian region also vanishes.  We conclude that
\begin{equation}\label{eq:C3-model-L1}
D^2U_j-\mathcal M_{\eps_j}(U_j,DU_j)\longrightarrow0
\quad\text{in }L^1_{\rm loc}(B_{L_0}).
\end{equation}
On the compact range considered here, $\mathcal M_{\eps}$ depends continuously on $(\eps,U,DU)$.

We next show that $\eps_j\to0$.  If a subsequence satisfied $\eps_j\to\eps_\infty>0$, the last term in \eqref{eq:C3-scaled-defect}, together with \eqref{eq:C3-Ulower}, would give $q_j\to0$ in $L^2_{\rm loc}$.  Since $q_j\asymp|DU_j|^2$ under $W_j\le W_*$, the local uniform limit would be constant, equal to $1$ by the normalization.  Stability of admissible viscosity solutions under the stated locally uniform convergence would then give $\sigma_k(0)=1$ at the constant limit, a contradiction.  Hence
\begin{equation}\label{eq:C3-eps-zero}
\eps_j\to0.
\end{equation}

Using \eqref{eq:C3-eps-zero} in \eqref{eq:C3-model-L1}, and the endpoint identity
\[
1-M_*=\frac{\tau_*}{c_k},
\]
we obtain in the limit
\begin{equation}\label{eq:C3-tracefree-power}
\left(D^2(U_\infty^{M_*})\right)=\frac{\Delta\!\left(U_\infty^{M_*}\right)}{n}\,I 
\quad\text{in }\mathcal D'(B_{L_0}).
\end{equation}
Indeed, at $\eps=0$ the zero-set relation gives
\[
D^2U_\infty
=
-\lambda\Id
+\frac{\tau_*}{c_kU_\infty}
DU_\infty\otimes DU_\infty.
\]
Hence
\[
D^2(U_\infty^{M_*})
=
-M_*\lambda U_\infty^{M_*-1}\Id
+
M_*U_\infty^{M_*-2}
\left(\frac{\tau_*}{c_k}+M_*-1\right)
DU_\infty\otimes DU_\infty.
\]
Since $1-M_*=\tau_*/c_k$, the second term vanishes.
Hence
\[
U_\infty^{M_*}=A+b\cdot x+B|x|^2.
\]
The centering inequality \eqref{eq:C3-contact-centering} passes to the limit and gives $U_\infty\le1$ near the origin while $U_\infty(0)=1$; therefore $DU_\infty(0)=0$ and $b=0$.  Since $U_\infty(0)=1$, we have $A=1$.  Substitution into the Euclidean critical equation then determines
\[
B=\frac{k}{(n-2k)\binom nk^{1/k}},
\]
so $U_\infty=\cB$ on $B_{L_0/2}$.  Equation \eqref{eq:C3-model-L1} also gives strong $W^{2,1}_{\rm loc}$ convergence on $B_{L_0/2}$.

We now extend this local identification to every fixed ball.  Since $R_j\to\infty$ and $U_j\le K_0$, Lemma~\ref{lem:C2} gives equicontinuity on every fixed ball.  A diagonal subsequence converges locally uniformly in $\mathbb R^n$ to a positive Euclidean critical admissible solution, again denoted $U_\infty$, satisfying
\[
\sigma_k(-D^2U_\infty)=U_\infty^{p_*}
\]
in the viscosity sense.  We already know $U_\infty=\cB$ on $B_{L_0/2}$.

Set $w=U_\infty-\cB$ and $F_{\cB}=T_{k-1}(-D^2\cB)$.
We now derive the linear inequality satisfied by $w$.  Put $a=p_*/k>1$.
If a smooth test function $\varphi$ touches $w$ from below at $x$, then
$\cB+\varphi$ touches the admissible viscosity solution $U_\infty$ from below.
The admissible branch of the equation therefore gives
\[
\sigma_k\bigl(-D^2(\cB+\varphi)\bigr)^{1/k}
\ge U_\infty(x)^a.
\]
The admissible viscosity condition places the Hessian matrix associated with this lower test on the closed G{\aa}rding branch.
At $B_0=-D^2\cB(x)\in\Gamma_k$, the supporting-plane
inequality for the concave function $\sigma_k^{1/k}$ gives
\[
U_\infty(x)^a-\cB(x)^a
\le-\frac1k\cB(x)^{p_*/k-p_*}
 F_{\cB}(x):D^2\varphi(x).
\]
Define the continuous nonnegative remainder
\begin{equation}\label{eq:C3-convexity-remainder}
\mathcal R(x)=k\cB(x)^{p_*(1-1/k)}
\left[U_\infty(x)^a-\cB(x)^a
-a\cB(x)^{a-1}w(x)\right].
\end{equation}
Strict convexity of $t^a$ on $(0,\infty)$ shows that
$\mathcal R(x)=0$ exactly where $w(x)=0$. Multiplying the tested
inequality by $k\cB^{p_*(1-1/k)}$ yields, for every such lower
test $\varphi$, the viscosity supersolution inequality
\begin{equation}\label{eq:C3linearineq}
F_{\cB}:D^2w+p_*\cB^{p_*-1}w\le-\mathcal R,
\qquad \mathcal R\ge0,
\end{equation}
with the displayed remainder. On each fixed ball $F_{\cB}$ is
smooth and positive definite; the usual equivalence for linear
uniformly elliptic viscosity and distributional supersolutions
applies to the continuous right-hand side
$-p_*\cB^{p_*-1}w-\mathcal R$. Thus \eqref{eq:C3linearineq}
holds in distributions on every fixed ball as well. 

Write $\cB(x)=B_*(r)$, $r=|x|$, and $e_r=x/r$. By radial symmetry,
\[
F_{\cB}=f_r(r)e_r\otimes e_r
+f_t(r)(\Id-e_r\otimes e_r),
\qquad
f_r(r)=\binom{n-1}{k-1}
\left(-\frac{B_*'(r)}r\right)^{k-1}.
\]
The scaling Jacobi field $\mathcal Z=x\cdot D\cB+\alpha_*\cB$
has radial profile $Z_*(r)=rB_*'(r)+\alpha_*B_*(r)$.
It solves the homogeneous linearized equation and changes sign only at
$R_{\rm turn}$; thus $\psi(r):=-Z_*(r)>0$ for $r>R_{\rm turn}$. Set
\[
A(r)=r^{n-1}f_r(r),
\qquad
\overline w(r)
=
\frac1{|\partial B_r|}\int_{\partial B_r}w\,dS,
\qquad
\overline{\mathcal R}(r)
=
\frac1{|\partial B_r|}\int_{\partial B_r}\mathcal R\,dS.
\]
The divergence-free identity for $F_{\cB}$ gives
$f_r'+(n-1)(f_r-f_t)/r=0$, hence
$A'(r)=(n-1)r^{n-2}f_t(r)$.
Averaging \eqref{eq:C3linearineq} over \(\partial B_r\) gives, in the distributional sense on $(0,\infty)$,
\[
\frac1{r^{n-1}}
\left(A(r)\overline w'(r)\right)'
+
p_*\,
B_*(r)^{p_*-1}\overline w(r)
\le
-\overline{\mathcal R}(r).
\]
The radial Jacobi profile satisfies
\[
\left(A(r)\psi'(r)\right)'
+
r^{n-1}p_*B_*(r)^{p_*-1}\psi(r)
=0.
\]
Hence
\begin{equation}\label{eq:C3wronskian}
\left[
A(r)\psi(r)^2
\left(\frac{\overline w(r)}{\psi(r)}\right)'
\right]'
\le
-r^{n-1}\psi(r)\overline{\mathcal R}(r)
\le0.
\end{equation}
Choose $R_{\rm turn}<R_1<L_0/2$.  Since $w=0$ on $B_{L_0/2}$, the quantity in brackets in \eqref{eq:C3wronskian} vanishes at $R_1$.  We use the locally $BV$ representative of the bracket determined by
\eqref{eq:C3wronskian}; this representative is nonincreasing.  If it were negative at some $r_0>R_1$, then there would exist $c_0>0$ such that, for every $r\ge r_0$,
\[
A(r)\psi(r)^2
\left(\frac{\overline w(r)}{\psi(r)}\right)'
\le -c_0.
\]
Hence
\[
\frac{\overline w(r)}{\psi(r)}
\le
\frac{\overline w(r_0)}{\psi(r_0)}
-c_0\int_{r_0}^r\frac{ds}{A(s)\psi(s)^2}.
\]
Since
\[
A(r)\psi(r)^2\asymp r^{1-\gamma},
\qquad
\int^\infty\frac{dr}{A(r)\psi(r)^2}=\infty,
\qquad
\gamma=\frac{n-2k}{k}>0,
\]
this would imply
\[
\frac{\overline w(r)}{\psi(r)}\longrightarrow-\infty.
\]
On the other hand, $U_\infty\ge0$ gives
\[
\frac{\overline w(r)}{\psi(r)}\ge-\frac{B_*(r)}{\psi(r)},
\]
and the right-hand side has a finite limit as $r\to\infty$, a contradiction.  Therefore the bracket in \eqref{eq:C3wronskian} vanishes identically.  It follows from \eqref{eq:C3wronskian} that
\[
r^{n-1}\psi(r)\overline{\mathcal R}(r)=0
\qquad (r>R_1).
\]
Since $\psi>0$ and $\mathcal R\ge0$ is continuous,
$\overline{\mathcal R}(r)=0$ implies $\mathcal R=0$ on every sphere
$\partial B_r$, $r>R_1$.  Together with $w=0$ on $B_{L_0/2}$, this gives
$\mathcal R\equiv0$.  By the strict convexity in
\eqref{eq:C3-convexity-remainder}, $w\equiv0$, and hence
$U_\infty\equiv\cB$.  Returning to \eqref{eq:C3-model-L1}, the continuity of $\mathcal M_\eps$ on the compact range, together with $\eps_j\to0$ and $DU_j\to D\cB$ in $L^1_{\rm loc}$, now gives
\[
D^2U_j\to D^2\cB
\quad\text{in }L^1_{\rm loc}(B_{L_0}).
\]
This gives the stated strong $W^{2,1}_{\rm loc}(B_{L_0})$ convergence and completes the proof of \eqref{eq:C3global}.
\end{proof}

\subsection{Blow-down limits and concentration of the $k$-Hessian measure}

For a $k$-convex function $w$, write $\mu_k[w]$ for its
$k$-Hessian measure; when $w$ is smooth,
$\mu_k[w]=\sigma_k(D^2w)\,dx$. This differs from the graph
volume measure $\mu$. Let
\[
\gamma=\frac{n-2k}{k},
\qquad
\Phi_m(x)=A_m|x|^{-\gamma},
\]
where $A_m>0$ is normalized by
\[
\mu_k[-\Phi_m]=m\delta_0.
\]
Write $\cB(x)=B_*(|x|)$ as above.  The source mass of the standard bubble and its tail coefficient satisfy
\begin{equation}\label{eq:C4-bubble-mass}
m_*:=\int_{\R^n}\cB^{p_*}\,dx
=\frac{|S^{n-1}|}{k}\binom{n-1}{k-1}
  (\gamma\kappa_\infty)^k,
\qquad
\kappa_\infty=\lim_{r\to\infty}r^\gamma B_*(r).
\end{equation}
Indeed, integrate $\sigma_k(-D^2\cB)=\cB^{p_*}$ on $B_R$
using $k\sigma_k(-D^2\cB)
=-\Div_{\R^n}\bigl(T_{k-1}(-D^2\cB)D\cB\bigr)$.
On $\partial B_R$ the radial entry of the Newton tensor is
$\binom{n-1}{k-1}(-B_*'(R)/R)^{k-1}$, and the resulting boundary integral converges to
the right-hand side of \eqref{eq:C4-bubble-mass}. Thus the fundamental
solution with mass $m_*$ is
\begin{equation}\label{eq:C4-Phi-tail}
\Phi_{m_*}(x)=\kappa_\infty|x|^{-\gamma}.
\end{equation}
\begin{lemma}[Exact inner-pole lower barrier]\label[lemma]{lem:inner-pole-lower}
Let $U_j:\R^n\to(0,\infty)$ be smooth entire strictly admissible solutions of
$\sigma_k(A_{\eps_j}[U_j])=U_j^{p_*}$ with $0\le\eps_j\to0$ and
$W_{\eps_j}[U_j]\le W_*$. Assume $U_j\to\cB$ locally uniformly in $\R^n$.
For every $t_j\to\infty$ and every compact $K\Subset\R^n\setminus\{0\}$,
\begin{equation}\label{eq:inner-pole-liminf}
\liminf_{j\to\infty}\inf_{x\in K}
\bigl(t_j^\gamma U_j(t_jx)-\Phi_{m_*}(x)\bigr)\ge0.
\end{equation}
\end{lemma}

\begin{proof}
Recall $B_*(r)=\cB(x), r=|x|$, choose $b_j\to\infty$ sufficiently slowly that $b_j=o(t_j)$ and
\[
\frac{\min_{\partial B_{b_j}}U_j}{B_*(b_j)}\longrightarrow1.
\]
Such a choice follows by diagonal selection from local uniform convergence:
for each fixed radius the convergence is uniform, and the chosen radii may
increase arbitrarily slowly relative to $t_j$. Set
\[
a_j=\gamma b_j^\gamma\min_{\partial B_{b_j}}U_j.
\]
Since $b_j^\gamma B_*(b_j)\to\kappa_\infty$, we have
$a_j\to\gamma\kappa_\infty$. On $r\ge b_j$, put
\[
y_j(r)=a_jr^{-\gamma-1},
\qquad
G_j(r)=\int_r^\infty
\frac{y_j(s)}{\sqrt{1+\varepsilon_jy_j(s)^2}}\,ds,
\]
and observe that
\[
G_j(b_j)\le\int_{b_j}^\infty a_js^{-\gamma-1}\,ds
=\frac{a_j}{\gamma}b_j^{-\gamma}
=\min_{\partial B_{b_j}}U_j.
\]
We use the same symbol $G_j$ for its radial extension to $\mathbb R^n$, that is, $G_j(x):=G_j(|x|)$.  The radial curvature formula with Lorentz parameter $\eps_j$ shows that
the principal curvatures of $G_j$ are
$(-\beta_{\mathrm{rad}}y_j/r,y_j/r,\ldots,y_j/r)$, where
$\beta_{\mathrm{rad}}=\gamma+1=(n-k)/k$.
Thus they lie on $\partial\Gamma_k$ and
$\sigma_k(A_{\eps_j}[G_j])=0$.
To justify the exterior comparison without imposing a boundary value at
infinity, fix $R>b_j$ and use the truncated barrier
\[
G_{j,R}(r)=\int_r^R 
\frac{y_j(s)}{\sqrt{1+\varepsilon_jy_j(s)^2}}\,ds.
\]
Then $G_{j,R}(R)=0<U_j$ and
$G_{j,R}(b_j)\le G_j(b_j)\le U_j$ on the inner boundary.
At a negative interior minimum of $U_j-G_{j,R}$ the gradients agree and
the symmetric shape representatives satisfy
$\widehat h[U_j]\preceq\widehat h[G_{j,R}]$.
Monotonicity of $\sigma_k$ along the intervening segment in
$\overline\Gamma_k$ would give
$U_j^{p_*}\le\sigma_k(A_{\eps_j}[G_{j,R}])=0$, a contradiction.
Hence $U_j\ge G_{j,R}$ on the annulus. Since $U_j$ is defined on
\emph{every} ball $B_R$, we may now let $R\to\infty$ for each fixed $j$,
obtaining $U_j(r\omega)\ge G_j(r)$ for every $r\ge b_j$.

For a compact $K\Subset\R^n\setminus\{0\}$, the condition
$b_j/t_j\to0$ ensures $t_j|x|\ge b_j$ for all $x\in K$ and large $j$.
Changing variables in the radial integral gives, uniformly for $x\in K$,
\[
t_j^\gamma G_j(t_j|x|)
=a_j\int_{|x|}^\infty
\frac{s^{-\gamma-1}}
{\sqrt{1+\eps_j a_j^2t_j^{-2\gamma-2}s^{-2\gamma-2}}}\,ds
\longrightarrow\frac{\gamma\kappa_\infty}{\gamma}|x|^{-\gamma}
=\Phi_{m_*}(x).
\]
The comparison proves \eqref{eq:inner-pole-liminf} in the stated
locally uniform lower-limit sense.
\end{proof}

The barrier applies to entire normalized solutions because the annular comparison uses an outer-radius limit. The concentration lemma also allows solutions defined only up to the first-crossing scale, provided the pole lower bound is assumed. The entire solutions constructed in \cref{sec:critical-global} meet that hypothesis by the preceding barrier.

Consider positive normalized critical solutions $U_j$ satisfying
\[
\sigma_k(A_{\eps_j}[U_j])=U_j^{p_*},\qquad
A_{\eps_j}[U_j]\in\Gamma_k,\qquad
U_j(0)=1,\qquad 0<U_j\le K_0,\qquad W_{\eps_j}\le W_*.
\]
Assume that $\eps_j\to0$, that $U_j\to\cB$ locally uniformly on
$\mathbb R^n$, and that $B_{\rho_j}$ is contained in the domain of $U_j$
for radii $\rho_j\to\infty$.
The bounds used below are stated in \cref{lem:C4}; in
\cref{sec:critical-global} they will follow from the first spherical crossing.
Set
\[
V_j(x)=\rho_j^\gamma U_j(\rho_j x).
\]
The Lorentz parameter after this blow-down is
\begin{equation}\label{eq:C4-lambda-scaling}
\lambda_j=\eps_j\rho_j^{-2(\gamma+1)},
\qquad 0\le\eps_j\le\eps_0,
\end{equation}
and
\[
W_j=(1-\lambda_j|DV_j|^2)^{-1/2}.
\]
Since $-D(W_jDV_j)$ is the transpose of the coordinate matrix
of the scaled shape operator $A_{\lambda_j}[V_j]$, the two matrices
have the same elementary symmetric functions.  Hence the rescaled
equation can be written as
\begin{equation}\label{eq:first-crossing-eq}
\sigma_k(-D(W_jDV_j))=\rho_j^{-2}V_j^{p_*}.
\end{equation}

\begin{lemma}[Compound minors on the G{\aa}rding cone]
\label[lemma]{lem:compound-cone}
Let $2\le k<n$ and $0\le s\le k-1$. For every real symmetric
$n\times n$ matrix $h$ with eigenvalues in $\Gamma_k$,
\begin{equation}\label{eq:compound-cone}
\|\wedge^s h\|\le C_{n,k,s}\sigma_s(h),
\qquad \wedge^0h=1,\quad\sigma_0(h)=1.
\end{equation}
Consequently, if $g=\Id-\lambda p\otimes p\succ0$,
$W=(1-\lambda|p|^2)^{-1/2}\ge1$ and
$h=Wg^{-1/2}Bg^{-1/2}$, then
\begin{equation}\label{eq:compound-congruence}
\|\wedge^s B\|\le C_{n,k,s}\sigma_s(h).
\end{equation}
The operator-norm bound applies to the full compound matrix and therefore controls the absolute value of every mixed $s$-minor.
\end{lemma}
\begin{proof}
It is enough to prove that, for every $s$-element index set $I$ and
$\lambda=(\lambda_1,\ldots,\lambda_n)\in\Gamma_k$,
\begin{equation}\label{eq:compound-eigenproduct}
\left|\prod_{i\in I}\lambda_i\right|
\le C_{n,k,s}\sigma_s(\lambda).
\end{equation}
We induct on $s$, with $s=0$ immediate. Fix $i\in I$ and put
$\widehat\lambda=\lambda|i$. The deleted tuple belongs to
$\Gamma_{k-1}$, so $\sigma_{s-1}(\widehat\lambda)>0$
and $\sigma_s(\widehat\lambda)>0$ when $s\le k-1$.
If $\lambda_i\ge0$, the identity
$\sigma_s(\lambda)=\sigma_s(\widehat\lambda)
+\lambda_i\sigma_{s-1}(\widehat\lambda)$ gives
\begin{equation}\label{eq:compound-one-factor}
|\lambda_i|\sigma_{s-1}(\widehat\lambda)
\le\sigma_s(\lambda).
\end{equation}
If $\lambda_i=-a<0$, then $\sigma_{s+1}(\lambda)>0$ and
\[
\sigma_{s+1}(\widehat\lambda)>a\sigma_s(\widehat\lambda)>0.
\]
The deleted tuple therefore belongs to $\Gamma_{s+1}$; this also
covers $s=k-1$, when the additional positivity follows from the
display rather than merely from $\widehat\lambda\in\Gamma_{k-1}$.
The normalized Newton inequality in dimension $m=n-1$ gives
\[
\frac{\sigma_{s+1}(\widehat\lambda)
             \sigma_{s-1}(\widehat\lambda)}
     {\sigma_s(\widehat\lambda)^2}
\le \theta_{m,s}
:=\frac{s(m-s)}{(s+1)(m-s+1)}<1.
\]
Consequently
\[
a\sigma_{s-1}(\widehat\lambda)
\le\theta_{m,s}\sigma_s(\widehat\lambda),\qquad
\sigma_s(\lambda)
\ge(1-\theta_{m,s})\sigma_s(\widehat\lambda),
\]
and hence \eqref{eq:compound-one-factor} holds with the fixed factor
$\theta_{m,s}/(1-\theta_{m,s})$ on its right-hand side. By the
induction hypothesis for $\widehat\lambda\in\Gamma_{k-1}$,
\[
\left|\prod_{j\in I\setminus\{i\}}\lambda_j\right|
\le C_{n-1,k-1,s-1}\sigma_{s-1}(\widehat\lambda).
\]
Multiplying the two inequalities proves
\eqref{eq:compound-eigenproduct}, including every sign pattern.
Since a symmetric matrix diagonalizes orthogonally, the eigenvalues
of $\wedge^s h$ are precisely these $s$-fold products, proving
\eqref{eq:compound-cone}. Finally,
\[
\wedge^sB=W^{-s}(\wedge^s g^{1/2})
               (\wedge^s h)(\wedge^s g^{1/2}).
\]
For $\lambda\ge0$ one has $0\prec g\preceq\Id$ and $W\ge1$,
so the operator norm of the two outside factors is at most one.
This proves \eqref{eq:compound-congruence}.
\end{proof}

\begin{lemma}[Concentration of the $k$-Hessian measure]\label[lemma]{lem:C4}
For the normalized solutions and radii just described, assume
\begin{equation}\label{eq:C4-source-domination}
U_j(y)\le C\cB(y)\qquad (|y|\le\rho_j),
\end{equation}
and that, on every fixed punctured compact set,
\[
V_j(x)\le C|x|^{-\gamma},
\qquad
|DV_j(x)|\le C|x|^{-\gamma-1}.
\]
Assume in addition that for every $K\Subset B_1\setminus\{0\}$,
\begin{equation}\label{eq:C4-lower-input}
\liminf_{j\to\infty}\inf_{x\in K}
\bigl(V_j(x)-\Phi_{m_*}(x)\bigr)\ge0.
\end{equation}

Then, after extraction,
\[
V_j\to V
\quad\text{locally uniformly on the punctured limit domain},
\]
and
\begin{equation}\label{eq:atom}
\mu_k[-V]=m_*\delta_0,
\qquad m_*\text{ is given by }\eqref{eq:C4-bubble-mass}.
\end{equation}
Moreover
\[
\Phi_{m_*}\le V\le C\Phi_{m_*}
\]
near the pole.
\end{lemma}

\begin{proof}
The assumed bounds for $V_j$ and $DV_j$ give local boundedness and equicontinuity on every compact subset of the punctured limit domain.  Hence, after passing to a subsequence,
\[
V_j\longrightarrow V
\]
locally uniformly there.  We keep this subsequence throughout the proof.

\smallskip
\noindent\emph{Step 1: lower-order minors and annular curvature mass.}
Let
\[
B_j=-D^2V_j.
\]
The symmetric curvature representative is
\[
h_j=W_jg_j^{-1/2}B_jg_j^{-1/2},
\qquad
g_j=\Id-\lambda_jDV_j\otimes DV_j.
\]
The uniform spacelikeness bound makes $g_j^{\pm1/2}$ and $W_j^{\pm1}$ uniformly bounded. Applying \cref{lem:compound-cone} to this exact congruence gives, for $0\le s\le k-1$,
\begin{equation}\label{eq:C4-compound-bound}
\|\wedge^s B_j\|\le C\sigma_s(h_j),
\qquad 0\le s\le k-1.
\end{equation}
The estimate controls the mixed minors and cofactors used in the double-divergence argument, even though $B_j$ need not be $k$-admissible.

We next prove the annular integral estimate
\begin{equation}\label{eq:C4-annular-sigma}
\int_{\{r<|x|<2r\}}\sigma_s(h_j)\,dx
\le Cr^{n-s(\gamma+2)},
\qquad 0\le s\le k-1,
\end{equation}
for $C/\rho_j\le r\le r_0$.  The case $s=0$ is the volume bound.  For $s\ge1$, take a cutoff $\chi$ supported in a slightly larger annulus and use
\[
\Div_{g_j}T_{s-1}(h_j)=0,
\qquad
\nabla_{g_j}^2V_j=-W_jh_j.
\]
After integration by parts,
\[
s\int W_j\sigma_s(h_j)\chi\,d\mu_j
=\int T_{s-1}(h_j)(\nabla V_j,\nabla\chi)\,d\mu_j.
\]
Since $T_{s-1}(h_j)>0$ and
\[
\tr T_{s-1}(h_j)=(n-s+1)\sigma_{s-1}(h_j),
\]
one has
\[
|T_{s-1}(h_j)(\nabla V_j,\nabla\chi)|
\le C\sigma_{s-1}(h_j)|\nabla V_j||\nabla\chi|.
\]
The constant must be uniform down to $r\asymp\rho_j^{-1}$, a scale not covered by the assumptions on fixed punctured compact sets.  Indeed, for $|y|\ge C$ and $|y|\le 2r_0\rho_j$, \eqref{eq:C4-source-domination} gives $U_j(y)\le C|y|^{-\gamma}$.  Apply Lemma~\ref{lem:C2} on a ball of radius comparable with $|y|$, contained in $B_{\rho_j}$, to obtain
\[
|DU_j(y)|\le C|y|^{-\gamma}
\bigl(|y|^{-1}+|y|^{-\gamma(p_*-k)/(2k)}\bigr)
\le C|y|^{-\gamma-1};
\]
the last inequality uses $\gamma(p_*-k)/(2k)=1+1/k$.  On bounded core balls, $U_j\le K_0$ and Lemma~\ref{lem:C2} instead give $|DU_j|\le C$.  Thus both $|\nabla V_j|\le Cr^{-\gamma-1}$ on the relevant annuli
and $\sup_{B_{2r_0\rho_j}}|DU_j|\le C$ hold, after fixing
$r_0<1/8$. To make the induction explicit, choose $\chi$ equal to one on
\(
A_r=\{r<|x|<2r\},
\)
supported in the enlarged annulus
\(
\widetilde A_r=\{r/2<|x|<4r\},
\)
and satisfying $|\nabla\chi|\le C/r$.  Since $W_j$ is uniformly
bounded above and below, $d\mu_j$ and $dx$ are uniformly equivalent.
The preceding integration-by-parts identity therefore gives, for
$1\le s\le k-1$,
\(
\int_{A_r}\sigma_s(h_j)\,dx
\le
C r^{-\gamma-2}
\int_{\widetilde A_r}\sigma_{s-1}(h_j)\,dx.
\)
Starting from the volume estimate at $s=0$ and iterating this
inequality yields
\eqref{eq:C4-annular-sigma}.  Since $\gamma+2=n/k$, the top exponent is
\[
n-(k-1)(\gamma+2)=\gamma+2.
\]
Combining \eqref{eq:C4-compound-bound} and \eqref{eq:C4-annular-sigma}, each dyadic shell $A_\rho=\{\rho<|x|<2\rho\}$ with $\rho\gtrsim\rho_j^{-1}$ satisfies
\[
\int_{A_\rho}|V_j|\,|T_{k-1}(B_j)|\,dx
\le C\rho^{-\gamma}\rho^{\gamma+2}=C\rho^2.
\]
The dyadic shells contribute at most $Cr^2$; we next estimate the innermost ball.  Returning there to the core variables $y=\rho_jx$ gives
\begin{align*}
&\int_{B_{C/\rho_j}}|V_j|\,|T_{k-1}(B_j)|\,dx\\
&\qquad=\rho_j^{\gamma+(\gamma+2)(k-1)-n}
\int_{B_C}U_j\,|T_{k-1}(-D^2U_j)|\,dy\\
&\qquad=\rho_j^{-2}
\int_{B_C}U_j\,|T_{k-1}(-D^2U_j)|\,dy=o_j(1).
\end{align*}
By \eqref{eq:compound-congruence},
\[
|T_{k-1}(-D^2U_j)|\le C\,\sigma_{k-1}(h_j)
\]
on the fixed core ball.  The integral of $\sigma_{k-1}(h_j)$ is uniformly bounded by the same recursive Newton-tensor integration-by-parts argument on nested fixed balls, starting from the volume bound and using $U_j\le K_0$ and Lemma~\ref{lem:C2}.  Hence the last integral is bounded uniformly in $j$.  In particular, the preceding annular estimates and the fixed-core estimates give, for every $1\le s\le k-1$ and fixed $0<r<r_0$,
\begin{equation}\label{eq:C4-lower-tight}
\int_{B_r}\|\wedge^s B_j\|\,dx
\le C r^{n-s(\gamma+2)}+C\rho_j^{s(\gamma+2)-n},
\end{equation}
and, for $0\le s\le k-1$ (with $\wedge^0B_j=1$),
\begin{equation}\label{eq:C4-weighted-tight}
\int_{B_r}V_j\|\wedge^s B_j\|\,dx
\le C r^{n-\gamma-s(\gamma+2)}
+C\rho_j^{\gamma+s(\gamma+2)-n}.
\end{equation}
Both exponents of $r$ are positive, and the exponents of $\rho_j$ are their negatives.  In the last display the smallest exponent of $r$, attained at $s=k-1$, is $2$.  Hence
\begin{equation}\label{eq:C4-VT-bound}
\int_{B_r}|V_j|\,|T_{k-1}(B_j)|\,dx
\le Cr^2+o_j(1).
\end{equation}

\smallskip
\noindent\emph{Step 2: the Lorentz correction disappears.}
We first record the exact Euclidean divergence formula.  Write $p=DV_j$ and $B_j=-D^2V_j$.  Since
\[
DW_j=-\lambda_jW_j^3B_jp,
\]
we have
\[
-D(W_jp)=W_j\bigl(B_j+\lambda_jW_j^2B_jp\otimes p\bigr).
\]
The rank-one expansion of $\sigma_k$ therefore gives
\[
\sigma_k(-D(W_jp))
=W_j^k\left[\sigma_k(B_j)
+\lambda_jW_j^2T_{k-1}(B_j)(B_jp,p)\right].
\]
On the other hand, $\Div_{\R^n}T_{k-1}(B_j)=0$ and
$\tr(T_{k-1}(B_j)B_j)=k\sigma_k(B_j)$, so direct differentiation of
$W_j^kT_{k-1}(B_j)p$ gives the same expression.  Hence
\begin{equation}\label{eq:Euclid-div}
\sigma_k(-D(W_jDV_j))
=-\frac1k\Div_{\R^n}\left(W_j^kT_{k-1}(B_j)DV_j\right).
\end{equation}
Combining \eqref{eq:Euclid-div} with the scaled equation yields
\begin{equation}\label{eq:C4-Euclidean-top}
\sigma_k(B_j)
=\rho_j^{-2}V_j^{p_*}+\frac1k\Div_{\R^n}\mathscr R_j,
\qquad
\mathscr R_j=(W_j^k-1)T_{k-1}(B_j)DV_j.
\end{equation}
On an annulus of radius $r$,
\[
|W_j^k-1|\le C\lambda_j|DV_j|^2
\le C\lambda_j r^{-2\gamma-2}.
\]
Using \eqref{eq:C4-annular-sigma}, \eqref{eq:C4-compound-bound}, and $|DV_j|\le Cr^{-\gamma-1}$ gives
\begin{equation}\label{eq:C4-R-shell}
\int_{\{r<|x|<2r\}}|\mathscr R_j|\,dx
\le C\lambda_jr^{-2\gamma-1}.
\end{equation}
Summing over dyadic shells from $r\asymp\rho_j^{-1}$ to a fixed $r_0$ gives the same order as the innermost shell.  On $B_{C/\rho_j}$, returning to the core variables gives
\[
\int_{B_{C/\rho_j}}|\mathscr R_j|\,dx
\le \rho_j^{-1}
\int_{B_C}|W_{\eps_j}^k-1|\,
|T_{k-1}(-D^2U_j)|\,|DU_j|\,dy.
\]
Since $|W_{\eps_j}^k-1|\le C\eps_j|DU_j|^2$, Lemma~\ref{lem:C2} and the fixed-scale Newton-tensor bound make the last integral $O(\eps_j)$.  Therefore
\begin{equation}\label{eq:C4-R-L1}
\|\mathscr R_j\|_{L^1(B_{r_0})}
\le C\frac{\varepsilon_j}{\rho_j}\longrightarrow0.
\end{equation}

\smallskip
\noindent\emph{Step 3: concentration of the source.}
At the critical exponent $\gamma p_*=n+2$. For every
$\eta\in C_c(B_{r_0})$, the change of variables $y=\rho_jx$ gives
\[
\int \eta(x)\rho_j^{-2}V_j(x)^{p_*}\,dx
=\int_{B_{\rho_j}}
\eta\!\left(\frac{y}{\rho_j}\right)U_j(y)^{p_*}\,dy.
\]
For each fixed $y$, the integrand converges to
$\eta(0)\cB(y)^{p_*}$, while \eqref{eq:C4-source-domination}
supplies the integrable majorant
$C\|\eta\|_{L^\infty}\cB^{p_*}$. Dominated convergence yields
\begin{equation}\label{eq:C4-source-atom}
\rho_j^{-2}V_j^{p_*}\,dx
\weakto m_*\delta_0,
\qquad
m_*=\int_{\mathbb R^n}\cB^{p_*}\,dx.
\end{equation}
By \eqref{eq:C4-Euclidean-top} and \eqref{eq:C4-R-L1},
\begin{equation}\label{eq:C4-top-distribution}
\sigma_k(B_j)\longrightarrow m_*\delta_0
\quad\text{in }\mathcal D'(B_{r_0}).
\end{equation}

\smallskip
\noindent\emph{Step 4: identification with the Hessian measure of the limit.}
Lorentz $k$-admissibility does not imply $B_j\in\Gamma_k$. We therefore establish convergence of the lower-order mixed minors and control their products with $V_j$ at the pole. After proving that the limit is $k$-convex, we apply Hessian-measure weak continuity to mollifications of that limit.  For a smooth function $w$ and every $s\ge1$,
\begin{equation}\label{eq:C4-double-div}
s\,\sigma_s(D^2w)
=\partial_i\partial_j\left(w\,T_{s-1}^{ij}(D^2w)\right).
\end{equation}
We shall also use the corresponding identity for mixed Hessian minors.  Let $I=(i_1,\ldots,i_s)$ and $J=(j_1,\ldots,j_s)$ be ordered index sets with distinct entries, and set
\[
M_{I,J}(D^2w)=\det\bigl(w_{i_a j_b}\bigr)_{a,b=1}^s.
\]
If $C^{I,J}_{ab}$ is the cofactor of the entry $w_{i_a j_b}$ in this minor, then
\begin{equation}\label{eq:C4-general-minor}
s\,M_{I,J}(D^2w)
=\sum_{a,b=1}^s\partial_{i_a}\partial_{j_b}
\bigl(w\,C^{I,J}_{ab}(D^2w)\bigr).
\end{equation}
The cofactor fields are divergence free in the selected row and column indices, by equality of mixed third derivatives.  We now justify convergence of the products at the pole.  Write $w_j=-V_j$.  The pointwise bound $V_j\le C\rho_j^\gamma$ on the fixed core, the source domination outside it, and $\gamma<n$ show that $w_j\to w:=-V$ in $L^1_{\rm loc}(B_{r_0})$, as well as uniformly on punctured compact sets.  By \eqref{eq:C4-lower-tight}, every minor of order $1\le s\le k-1$ has uniformly bounded total variation and no point mass in any of its measure limits at the pole.  When $k\ge2$, the order-one limit is the distributional Hessian $D^2w$.  Inductively, suppose the measure limits of all minors of order $s-1$ have been uniquely identified.  Their cofactors in \eqref{eq:C4-general-minor} converge weakly as measures; on a punctured compact set the uniform convergence $w_j\to w$ therefore identifies the products $w_jC^{I,J}_{ab}(B_j)$.  Estimate \eqref{eq:C4-weighted-tight}, with order $s-1$, removes the pole in this product limit: its mass in $B_r$ is at most
\[
C r^{n-\gamma-(s-1)(\gamma+2)}+o_j(1),
\qquad n-\gamma-(s-1)(\gamma+2)\ge2.
\]
The $r$-dependent bound tends to zero as $r\downarrow0$, so the product limit has no mass at the pole and is uniquely determined by its restriction off the pole.  Applying \eqref{eq:C4-general-minor} identifies all order-$s$ minors.  This proves the assertion for every lower-order mixed minor, including every cofactor of $T_{k-1}(B_j)$.  Applying the product argument with $s=k$ gives a unique weak measure limit for
\[
P_j^{ij}=w_jT_{k-1}^{ij}(B_j)\,dx;
\qquad \limsup_{j\to\infty}|P_j|(B_r)\le Cr^2.
\]
In particular, $P=\lim_jP_j$ carries no point mass at the pole.  Nevertheless, its double divergence
\[
\partial_i\partial_jP^{ij}
\]
may contain a Dirac mass there.  The $Cr^2$ bound rules out a defect measure in $P$ itself, but it does not exclude the singular $k$-Hessian mass recovered after taking the double divergence.

We also verify $k$-convexity across the pole. The core and annular
gradient bounds from Step~1 give, on $B_{2r_0}$,
$\theta_j:=\lambda_j|DV_j|^2\le C\eps_j\to0$.
By the exact congruence
$B_j=W_j^{-1}g_j^{1/2}h_jg_j^{1/2}$,
the expansions $W_j^{-1}=1+O(\theta_j)$ and
$g_j^{1/2}=\Id+O(\theta_j)$ yield
\[
|B_j-h_j|\le C\theta_j|h_j|\le C\eps_j|h_j|.
\]
Since $k\ge2$, $|h_j|\le\sigma_1(h_j)$; the order-one annular
estimate \eqref{eq:C4-annular-sigma} and the fixed-core estimate
bound $\int_{B_{r_0}}|h_j|$ uniformly. If $A$ is any constant matrix
in the dual cone $\Gamma_k^*$, then $A:h_j\ge0$, and hence
\[
A:D^2w_j=A:B_j\ge-C\eps_j|A|\,|h_j|.
\]
Passing to distributions gives $A:D^2w\ge0$.  The distributional characterization of $k$-convexity \cite[Lemma~2.2]{TrudingerWang} therefore supplies a proper upper-semicontinuous $k$-convex representative of $w$ on the whole ball. This representative is defined at the pole.

It remains to compare $P$ with the Hessian measure of this representative.  Let $w_\tau=w*\vartheta_\tau$ be a standard nonnegative mollification on a smaller ball.  These are smooth $k$-convex functions, $w_\tau\le0$, and $w_\tau\to w$ locally in $L^1$ \cite[Lemma~2.3]{TrudingerWang}.  The above growth estimates imply, whenever $\tau\le r$,
\[
\int_{B_{4r}}|w_\tau|\le Cr^{n-\gamma}.
\]
Apply the local weighted Hessian estimate \cite[Theorem~4.3, $l=k-1$, $q=1$]{TrudingerWang} to $r^\gamma w_\tau(r\,\cdot)$ on nested fixed balls.  Its rescaled $L^1$ norm is uniformly bounded, and $n>2k$ makes $q=1$ admissible.  Rescaling back and using the compound-minor bound on $\Gamma_k$ gives
\begin{equation}\label{eq:C4-mollifier-weighted}
\int_{B_r}|w_\tau|\,|T_{k-1}(D^2w_\tau)|\,dx\le Cr^2,
\qquad \tau\le r,
\end{equation}
with the constant independent of $\tau$ and $r$.  On a punctured compact set, the lower-order minors of $D^2w_\tau$ have bounded total variation by the local Hessian-mass estimate \cite[Theorem~3.1]{TrudingerWang} and the compound-minor bound.  Repeating the preceding induction, now with $w_\tau\to w$ uniformly on these compact sets, shows that their mixed cofactor limits agree with those obtained from $w_j$.  Consequently $w_\tau T_{k-1}(D^2w_\tau)\,dx\rightharpoonup P$ off the pole.  Estimates \eqref{eq:C4-VT-bound} and \eqref{eq:C4-mollifier-weighted} remove the pole for both product sequences, so this convergence holds throughout the smaller ball.  The double-divergence identity and Trudinger--Wang weak continuity \cite[Theorem~1.1]{TrudingerWang} now give the precise identity
\[
\mu_k[w]
=\lim_{\tau\downarrow0}\sigma_k(D^2w_\tau)\,dx
=\frac1k\partial_i\partial_jP^{ij}
=\lim_{j\to\infty}\sigma_k(B_j)
\quad\text{in }\mathcal D'(B_{r_0}).
\]
Combining this with \eqref{eq:C4-top-distribution} gives
\[
\mu_k[-V]=m_*\delta_0.
\]
To compare $V$ with the pole profile, we use the assumed lower bound \eqref{eq:C4-lower-input}.  By local
uniform convergence on punctured compact sets, \eqref{eq:C4-lower-input}
gives $V\ge\Phi_{m_*}$ on $B_{r_0}\setminus\{0\}$.  The assumed upper
bound on $V_j$ passes to the limit and, after adjusting its constant,
gives $V\le C\Phi_{m_*}$ there. Hence
\[
\Phi_{m_*}\le V\le C\Phi_{m_*}
\]
near the pole.  This proves the lemma.
\end{proof}

\subsection{Single-pole expansion and the critical Pohozaev identity}

We retain the notation $\Phi_m$ introduced in the preceding subsection.

\begin{lemma}[Single-pole expansion]\label[lemma]{lem:C5}
Let $V$ be positive and continuous in a punctured ball, and let
$w=-V$ extend to a proper $k$-convex function on the whole ball.
Assume
\[
\mu_k[-V]=m\delta_0,
\qquad
\Phi_m\le V\le C\Phi_m.
\]
Then there exist $b_0\ge0$, $\sigma\in(0,1)$ and $r_*>0$ such that
\begin{equation}\label{eq:single-pole-expansion}
|V-\Phi_m-b_0|
+|x|\,|D(V-\Phi_m)|
+|x|^2|D^2(V-\Phi_m)|
\le C|x|^\sigma
\end{equation}
for $0<|x|<r_*$.
\end{lemma}

\begin{proof}
We identify the leading coefficient using the Hessian mass, establish uniform ellipticity on rescaled annuli, and then prove that the bounded remainder extends across the pole.

\smallskip
\noindent\emph{Step 1: bounded Riesz remainder and identification of its coefficient.}
Regard $\mathcal F_k:=\overline{\Gamma_k}$ as an orthogonally
invariant subequation. Its Riesz characteristic is
$p_{\mathcal F_k}=n/k>2$ \cite[Example~3.6(5)]{HarveyLawson}; thus
$\gamma=p_{\mathcal F_k}-2$ and its normalized Riesz kernel is
$K_{p_{\mathcal F_k}}=-|x|^{-\gamma}/\gamma$. Consistently with
$\Phi_m=A_m|x|^{-\gamma}$, define the unnormalized density by
\[
\widetilde\Theta(w,0):=
\lim_{r\downarrow0}\frac{\sup_{B_r}w}{-r^{-\gamma}},
\qquad
\Theta_{\mathcal F_k}(w,0)=\gamma\widetilde\Theta(w,0).
\]

On every punctured compact subset the Hessian measure of $w$ vanishes.
The equivalence between zero Hessian measure and the homogeneous
$k$-Hessian equation in the admissible viscosity sense
\cite[Section~2, remark after Lemma~2.4]{TrudingerWang}
therefore makes $w$ $\mathcal F_k$-harmonic there.
The Riesz-density theory for orthogonally invariant subequations of
finite Riesz characteristic
\cite[Sections~3--4]{HarveyLawson}
gives the existence of the density
$\Theta_{\mathcal F_k}(w,0)$.
The two-sided comparison
$\Phi_m\le V\le C\Phi_m$
then implies
$
\widetilde\Theta(w,0)
\in[A_m,CA_m].
$
In particular, the density is positive and finite, and $w(0)=-\infty$.

Choose a small ball $B_{r_0}$ on whose boundary $w$ is continuous.
The ball is strictly $\mathcal F_k$-convex; the orthogonal invariance
and finiteness of $p_{\mathcal F_k}$ give the structural property
(F3) in \cite[Lemma~3.4]{HarveyLawson}.
The prescribed-singularity theorem
\cite[Theorem~3.5]{HarveyLawson}, together with the prescribed-density
formulation in \cite[Corollary~4.2]{HarveyLawson}, gives a solution with
this boundary value and density $\Theta_{\mathcal F_k}(w,0)$.
Uniqueness for the prescribed-density problem follows from
\cite[Corollary~4.5]{HarveyLawson}.  Since $w$ has the same boundary
value and the same density, it coincides with that solution.
The corresponding polar asymptotics therefore give
\begin{equation}\label{eq:C5-O1}
	V(x)=A|x|^{-\gamma}+O(1)
	\qquad (x\to0)
\end{equation}
for some $A>0$.  The Hessian mass determines $A$.  Define
$
w_r(x)=r^\gamma w(rx).
$
Then $w_r\to-A|x|^{-\gamma}$ locally uniformly away from the origin and
$
\mu_k[w_r]=m\delta_0.
$
This convergence also holds in $L^1_{\rm loc}$ across the pole.
Indeed, for each fixed $R>0$ and all sufficiently small $r$,
the two-sided bound for $V$ gives
\[
|w_r(x)|=r^\gamma V(rx)\le C|x|^{-\gamma}
\qquad (0<|x|<R).
\]
Since $0<\gamma=(n-2k)/k<n$, the right-hand side is
integrable on $B_R$. The convergence away from the origin and
dominated convergence therefore imply
\[
w_r\longrightarrow -A|x|^{-\gamma}
\quad\text{in }L^1_{\rm loc}(\mathbb R^n),
\]
and in particular in measure.
Each $w_r$ is Euclidean $k$-convex.  The weak continuity theorem of Trudinger--Wang \cite[Theorem~1.1]{TrudingerWang} therefore gives
$
\mu_k[-A|x|^{-\gamma}]=m\delta_0.
$
A direct computation of the radial \(k\)-Hessian measure yields
\[
m=\frac{|S^{n-1}|}{k}\binom{n-1}{k-1}(A\gamma)^k,
\]
so $A=A_m$.  Hence
\begin{equation}\label{eq:C5-bounded-e}
e:=V-\Phi_m=O(1).
\end{equation}

\smallskip
\noindent\emph{Step 2: local ellipticity on rescaled annuli.}
Although $\sigma_k(-D^2\Phi_m)=0$, the corresponding point of the cone boundary is nondegenerate:
\begin{equation}\label{eq:C5-Tpositive}
T_{k-1}(-D^2\Phi_m)\succ0
\qquad(x\ne0).
\end{equation}
Indeed, writing $\theta=x/|x|$ and $b=A_m\gamma |x|^{-\gamma-2}$, the eigenvalues of $-D^2\Phi_m$ are
\[
-(\gamma+1)b,\quad b,\ldots,b,
\]
so the radial and tangential eigenvalues of $T_{k-1}$ are positive multiples of $b^{k-1}$.

To apply Savin's small-perturbation theorem, define locally
\[
\mathfrak F(N)=\sup\{t\in\mathbb R:N-t\Id\in\overline\Gamma_k\}.
\]
Near the compact set of matrices $-D^2\Phi_m$ on a fixed annulus, the implicit function theorem and \eqref{eq:C5-Tpositive} show that $\mathfrak F$ is smooth and locally uniformly elliptic, with
\[
D_N\mathfrak F(N)
=\frac{T_{k-1}(N)}{\tr T_{k-1}(N)}
\]
on the cone boundary.  Let
\[
w_r(x)=-r^\gamma V(rx),
\qquad
\phi(x)=-\Phi_m(x),
\qquad z_r=w_r-\phi.
\]
By \eqref{eq:C5-bounded-e}, on a fixed annulus
\[
\|w_r-\phi\|_{L^\infty}\le Cr^\gamma\to0.
\]
To verify Savin's small-perturbation hypothesis, we normalize by the square
root of the remainder size.  Put
\[
A^+=\{1/4<|x|<4\},\qquad
A^-=\{1/2<|x|<2\},\qquad
\delta_r=\|z_r\|_{L^\infty(A^+)}\le Cr^\gamma .
\]
If $\delta_r=0$, the desired estimate is immediate.  Otherwise set
\[
t_r=\sqrt{\delta_r},\qquad
v_r=\frac{z_r}{t_r},
\]
and define
\[
G_r(M,x)
=
\frac{
	\mathfrak F(D^2\phi(x)+t_rM)-\mathfrak F(D^2\phi(x))
}{t_r}.
\]
Then
\[
G_r(D^2v_r,x)=0,\qquad
G_r(0,x)=0,
\]
and
\[
\|v_r\|_{L^\infty(A^+)}
=
\sqrt{\delta_r}\longrightarrow0.
\]

We next verify the uniform structural hypotheses for the rescaled
operators.  The set
\[
K=\{D^2\phi(x):x\in\overline{A^+}\}
\]
is a compact subset of the smooth nondegenerate portion of
$\{\mathfrak F=0\}$.  Hence there exist
$\kappa,\lambda_0,\Lambda_0>0$ such that $\mathfrak F$ is smooth and
\[
\lambda_0\Id
\le D_N\mathfrak F
\le\Lambda_0\Id
\]
throughout the $\kappa$-neighborhood of $K$.

Fix a radius $R>0$ in matrix space and take $r$ sufficiently small that
$t_rR<\kappa/2$.  For $\|M\|\le R$,
\begin{equation}\label{eq:C5-scaled-operator}
	G_r(M,x)
	=
	\int_0^1
	D_N\mathfrak F(D^2\phi(x)+s t_rM):M\,ds,
\end{equation}
and
\[
D_MG_r(M,x)
=
D_N\mathfrak F(D^2\phi(x)+t_rM).
\]
Thus the ellipticity constants $\lambda_0,\Lambda_0$ and the admissible
matrix radius $R$ are independent of $r$.  Moreover,
\[
D^2_{MM}G_r
=
t_rD^2_{NN}\mathfrak F
\]
is uniformly bounded.

Differentiating the integral representation
\eqref{eq:C5-scaled-operator} with respect to $x$ and $M$ shows that the
first and second $x$-derivatives, as well as the mixed $(M,x)$-derivatives,
of $G_r$ are uniformly bounded; in particular, no factor $t_r^{-1}$ is
lost in these estimates.  Hence the coefficient oscillation and the
modulus of continuity of $D_MG_r$ on balls of fixed radius are uniform
in $r$, since the smooth reference solution $\phi$ is fixed.

Finally, $\mathfrak F$ is globally degenerate elliptic because
\[
\overline{\Gamma_k}+\{P\ge0\}
\subset\overline{\Gamma_k}.
\]
Consequently, the operators $G_r$ have the global monotonicity and the
local smooth uniform ellipticity required by Savin's small-perturbation
theorem \cite{Savin}.

Cover $A^-$ by finitely many balls compactly contained in $A^+$.  On each,
rescale the spatial variable to a unit ball (and divide $v_r$ by the square
of that ball's fixed radius so that the Hessian variable is unchanged).
The $L^\infty$ norm of the resulting solution is $O(\sqrt{\delta_r})$,
below Savin's uniform threshold for all sufficiently small $r$.
Savin's estimate gives $\|v_r\|_{C^{2,\eta}(A^-)}\le C$, with $C$ independent of
$r$.  Consequently
\begin{equation}\label{eq:C5-relative-C2}
\|z_r\|_{C^{2,\eta}(\{1/2<|x|<2\})}
\le C\sqrt{\delta_r}\le Cr^{\gamma/2}.
\end{equation}
In original variables, on $\{r/2<|x|<2r\}$ this yields
\[
|D^2e(x)|\le C r^{-\gamma-2+\gamma/2},
\qquad
\frac{|D^2e(x)|}{|D^2\Phi_m(x)|}\le Cr^{\gamma/2}\longrightarrow0.
\]
The annular $C^{0,\eta}$ seminorm of the rescaled relative Hessian error
has the same $O(r^{\gamma/2})$ bound.  This decay is used in the Schauder argument below.

The exact difference of the two homogeneous $k$-Hessian equations gives
\begin{equation}\label{eq:C5-linear-e}
a^{ij}(x)e_{ij}=0,
\qquad
a(x)=\int_0^1T_{k-1}\bigl(-D^2\Phi_m-tD^2e\bigr)\,dt.
\end{equation}
Define the normalized coefficient field
\begin{equation}\label{eq:C5-ahat}
\widehat a(x)=c_m^{-1}|x|^{(k-1)(\gamma+2)}a(x),
\qquad
c_m=\binom{n-1}{k-1}(A_m\gamma)^{k-1}.
\end{equation}
Multiplication by this positive scalar does not change the equation.  The coefficients $\widehat a$ are uniformly elliptic and converge, as $x\to0$, to the radial matrix
\[
\theta\otimes\theta+
\frac{\gamma+1}{n-1}(\Id-\theta\otimes\theta).
\]
More precisely, on every rescaled annulus
$\{r/2<|x|<2r\}$ the difference from this radial field has
$C^{0,\eta}$ norm at most $Cr^{\gamma/2}$, where the H\"older norm is
computed after putting $x=ry$.

\smallskip
\noindent\emph{Step 3: removal of the puncture and the constant regular part.}
Fix $0<\beta<\gamma$.  The limiting radial operator satisfies
\[
\mathcal L_0(|x|^{-\beta})
=\beta(\beta-\gamma)|x|^{-\beta-2}<0.
\]
By the convergence of the normalized coefficients, after decreasing $r_*$ if necessary,
\begin{equation}\label{eq:C5-singular-barrier}
\widehat a^{ij}\partial_{ij}|x|^{-\beta}
\le-c|x|^{-\beta-2}
\qquad(0<|x|<r_*).
\end{equation}
A bounded solution of \eqref{eq:C5-linear-e}, equivalently of $\widehat a^{ij}e_{ij}=0$, therefore has a removable isolated singularity.  To see this, replace $\widehat a$ inside $B_{2\delta}$ by a uniformly elliptic coefficient field $\widehat a_\delta$ that is constant on $B_\delta$. Solve the Dirichlet problem on $B_{r_*}$ with boundary data $e$. Uniform Krylov--Safonov estimates then give a continuous subsequential limit $h$ as $\delta\downarrow0$.  Then $e-h$ is bounded, solves the homogeneous equation on the punctured ball, and vanishes on the outer boundary.  Applying the maximum principle to
\[
\pm(e-h)-\varepsilon\bigl(|x|^{-\beta}-r_*^{-\beta}\bigr)
\]
and then sending $\varepsilon\downarrow0$ shows $e\equiv h$.  Hence
\[
b_0:=e(0)
\]
is well defined and, because $e\ge0$, $b_0\ge0$.  Krylov--Safonov gives
\[
|e(x)-b_0|\le C|x|^\sigma
\]
for some $\sigma\in(0,1)$.  Rescaling \eqref{eq:C5-linear-e} on annuli and applying Schauder estimates, using the $C^{0,\eta}$ control supplied by \eqref{eq:C5-relative-C2}, gives
\[
|e-b_0|+|x||De|+|x|^2|D^2e|
\le C|x|^\sigma,
\]
which is \eqref{eq:single-pole-expansion}.

\end{proof}

\begin{lemma}[Stability on fixed annuli]\label[lemma]{lem:C5-annulus}
Let $V_j$ be a first-crossing approximating sequence as in \cref{lem:C4}, with locally uniform convergence to $V$ on punctured compact sets.
Assume that $V$ satisfies \cref{lem:C5}, and let $r_*$ be supplied there.
For each fixed $0<r<\min\{r_*,1\}$, one has $V_j\to V$ in $C^2$ in a
neighborhood of $\partial B_r$.
\end{lemma}

\begin{proof}
Fix $0<r<\min\{r_*,1\}$.  On an annulus $A_r$ around $\partial B_r$, the reference function $V$ is smooth and the matrices $-D^2V$ stay in a compact nondegenerate neighborhood of the smooth part of $\partial\Gamma_k$.  Write $w_j=-V_j$, $w=-V$.  To specify a globally monotone representative of the scaled Lorentz equation near the reference triple $(D^2w,Dw,w)$, put
\[
Q_j(M,p)=W_j(p)g_j(p)^{-1/2}Mg_j(p)^{-1/2},
\quad g_j(p)=\Id-\lambda_jp\otimes p,
\quad W_j(p)=(1-\lambda_j|p|^2)^{-1/2},
\]
and define $\mathcal H_j(M,p,z)$ as the unique scalar $t$ satisfying
\[
Q_j(M,p)-t\Id\in\overline\Gamma_k,
\qquad
\sigma_k(Q_j(M,p)-t\Id)=\rho_j^{-2}(-z)^{p_*}.
\]
Only $p$ with $\lambda_j|p|^2<1$ and $z<0$ are needed here; outside a fixed neighborhood of the reference set $\{(D^2w(x),Dw(x),w(x)):x\in\overline{A_r}\}$, extend the $p,z$ dependence by a smooth truncation.  For fixed $p,z$ the operator is globally monotone in $M$, since positive semidefinite increments preserve the cone condition.  The implicit function theorem and $T_{k-1}(-D^2V)\succ0$ give a fixed smoothness neighborhood and uniform ellipticity constants for all sufficiently large $j$.  On the admissible solutions the original equation is precisely
\[
\mathcal H_j(D^2w_j,Dw_j,w_j)=0.
\]
After subtracting the reference function, $z_j=w_j-w$ solves
\[
\widetilde{\mathcal H}_j(D^2z_j,Dz_j,z_j,x)=f_j(x),
\qquad
\widetilde{\mathcal H}_j(0,0,0,x)=0,
\]
where explicitly
\[
f_j(x)=-\mathcal H_j(D^2w(x),Dw(x),w(x)).
\]
On the fixed annulus the reference data are smooth, while the first-crossing parameters satisfy $\lambda_j\to0$ and $\rho_j^{-2}\to0$.  Hence
\[
\|f_j\|_{C^{0,\eta}(A_r)}\le C(\lambda_j+\rho_j^{-2})\longrightarrow0.
\]
For large $j$, the cone-level representatives are globally monotone,
uniformly elliptic on a fixed matrix neighborhood, uniformly Lipschitz
in the lower-order variables, and have a common modulus for their
Hessian derivatives. These verify the ellipticity, lower-order, and Hessian-modulus hypotheses of Fan's
nonhomogeneous theorem \cite[Theorem~1.7 in the arXiv version]{Fan}.
To make the coefficient oscillation small, cover a smaller annulus by
finitely many $B_\ell(x_a)\Subset A_r$ and set
$\widetilde z_{j,a}(y)=\ell^{-2}z_j(x_a+\ell y)$.
The Hessian argument is unchanged, while the gradient and value
arguments gain factors $\ell$ and $\ell^2$. The smooth reference
triple varies by at most $C\ell^\eta$, so the rescaled coefficient
oscillation is $O(\ell^\eta)$, uniformly in $j$ and $a$.
The right-hand side $\widetilde f_{j,a}(y)=f_j(x_a+\ell y)$ satisfies
\[
\|\widetilde f_{j,a}\|_{C^{0,\eta}(B_1)}
\le (1+\ell^\eta)\|f_j\|_{C^{0,\eta}(A_r)}\to0.
\]
Choose $\ell$ first so that the coefficient oscillation is below Fan's threshold. With this $\ell$ fixed, take $j$ large enough that both $\ell^{-2}\|z_j\|_{L^\infty(A_r)}$ and the rescaled right-hand-side norm are below the corresponding thresholds. Fan's theorem gives uniform $C^{2,\eta}$
bounds for these rescaled differences; interpolation with $z_j\to0$
uniformly gives $D^2z_j\to0$ on the smaller balls.  Hence
\[
V_j\to V\quad\text{in }C^2
\]
near $\partial B_r$.
\end{proof}

We shall also use the following critical Pohozaev identity for smooth
positive solutions $Y$ satisfying $\lambda|DY|^2<1$,
$A_\lambda[Y]\in\overline\Gamma_k$, and
$\sigma_k(A_\lambda[Y])=\kappa Y^{p_*}$ with
$\lambda\ge0$ and $\kappa\ge0$.  For
\[
B=-D^2Y,\quad \xi=DY,\quad T=T_{k-1}(B),\quad
Z=T(\xi,\xi),\quad \mathfrak R=D_BZ,
\]
define
\[
\psi_{k,\lambda}(s)
=\frac1k\int_0^1\frac{t^k}{(1-\lambda t^2s)^{k/2}}\,dt.
\]
At $p=p_*$ set
\begin{align}
\Pi_{\lambda,\kappa}[Y]
={}&\frac{W_\lambda^k}{k}
(x\cdot DY+\alpha_* Y)T\xi\label{eq:pohozaev-current}\\
&+\psi_{k,\lambda}\mathfrak RBx
-x\psi_{k,\lambda}Z
+\frac{\kappa}{p_*+1}xY^{p_*+1},
\nonumber
\end{align}
where $\alpha_*=(n-2k)/(k+1)$ and every occurrence of
$\psi_{k,\lambda}$ is evaluated at $s=|DY|^2$.
For the divergence calculation, we use the following identities. The
Newton cofactor formula and the symmetry of third derivatives of $Y$
give
\[
\partial_i\mathfrak R^{ij}=(k-1)(T\xi)^j,
\qquad \mathfrak R:B=(k-1)Z,
\qquad \mathfrak R B\xi=Z\xi-sT\xi.
\]
For example, when $B$ is diagonal with eigenvalues $\beta_i$, its
off-diagonal entries satisfy
$\mathfrak R_{i\ell}=-\sigma_{k-2}(\beta|i,\ell)\xi_i\xi_\ell$,
while
$\mathfrak R_{ii}=\sum_{\ell\ne i}\sigma_{k-2}(\beta|i,\ell)\xi_\ell^2$.
The last identity follows from
$\sigma_{k-1}(\beta|\ell)-\sigma_{k-1}(\beta|i)
=(\beta_i-\beta_\ell)\sigma_{k-2}(\beta|i,\ell)$;
the first follows by differentiating the cofactor formula, whose terms
containing $DB$ cancel in pairs because $B=-D^2Y$.
Using $Ds=-2B\xi$ and $DB$ symmetric in its three indices, we obtain
\[
\Div_{\R^n}\bigl(\psi_{k,\lambda}\mathfrak R Bx
-x\psi_{k,\lambda}Z\bigr)
=\bigl((k+1)\psi_{k,\lambda}+2s\psi'_{k,\lambda}\bigr)
\langle T\xi,Bx\rangle-(n-k+1)\psi_{k,\lambda}Z.
\]
Differentiating $t^{k+1}(1-\lambda t^2s)^{-k/2}$ and integrating over
$0\le t\le1$ gives
\[
\frac{W_\lambda^k}{k}=(k+1)\psi_{k,\lambda}(s)
+2s\psi'_{k,\lambda}(s).
\]
Finally, \eqref{eq:Euclid-div} and the equation for $Y$ imply
$\Div_{\R^n}(W_\lambda^kT\xi)=-k\kappa Y^{p_*}$.
The reaction terms cancel because
$\alpha_*=n/(p_*+1)$, and the remaining terms cancel because
$n-k+1=(k+1)(\alpha_*+1)$.  Thus
\begin{equation}\label{eq:pohozaev-div}
\Div_{\R^n}\Pi_{\lambda,\kappa}[Y]
=2(\alpha_*+1)|DY|^2\psi'_{k,\lambda}(|DY|^2)Z.
\end{equation}
Here $\psi'_{k,\lambda}\ge0$.  It remains to justify $Z\ge0$ without assuming that $B=-D^2Y$ itself lies in $\Gamma_k$.  At a point where $DY\ne0$, choose the Euclidean frame $e_1=DY/|DY|$.  Since
\[
h=W_\lambda g_\lambda^{-1/2}Bg_\lambda^{-1/2},
\qquad
g_\lambda=\Id-\lambda DY\otimes DY,
\]
the transverse block satisfies $h_\perp=W_\lambda B_\perp$.  Therefore, with $F_h=T_{k-1}(h)$,
\[
F_h(e_1,e_1)=\sigma_{k-1}(h_\perp)
=W_\lambda^{k-1}\sigma_{k-1}(B_\perp)
=W_\lambda^{k-1}T(e_1,e_1).
\]
Lorentz admissibility gives $F_h(e_1,e_1)\ge0$, and hence
\[
Z=T(DY,DY)=|DY|^2W_\lambda^{-(k-1)}F_h(e_1,e_1)\ge0.
\]
At $DY=0$ this is immediate.  Thus the divergence in \eqref{eq:pohozaev-div} is nonnegative.
For the single-pole limit $\lambda=\kappa=0$, so $W_0=1$,
$\psi_{k,0}=1/[k(k+1)]$, and $\psi'_{k,0}=0$.
For $Y=\Phi_m$, let $e_r=x/r$ and
$F_r=T_{k-1}(-D^2\Phi_m)(e_r,e_r)$. Then
$D\Phi_m=-\gamma\Phi_m e_r/r$,
$Z=F_r\gamma^2\Phi_m^2/r^2$, and the radial cofactor formula
gives $\mathfrak R Bx=0$. Since $\alpha_*=k\gamma/(k+1)$,
the remaining terms of \eqref{eq:pohozaev-current} cancel:
\[
\frac1k
(x\cdot D\Phi_m+\alpha_*\Phi_m)
T_{k-1}(-D^2\Phi_m)D\Phi_m
-
x\psi_{k,0}Z
=0.
\]
Thus $\Pi_{0,0}[\Phi_m]\equiv0$ off the origin.

Adding the constant $b_0$ changes neither the gradient nor the Hessian.
Therefore
\[
\Pi_{0,0}[\Phi_m+b_0]
-
\Pi_{0,0}[\Phi_m]
=
\frac{\alpha_*b_0}{k}
T_{k-1}(-D^2\Phi_m)D\Phi_m.
\]
The mass identity $\mu_k[-\Phi_m]=m\delta_0$ gives
$\int_{\partial B_r}T_{k-1}(-D^2\Phi_m)D\Phi_m\cdot\nu\,dS=-km$.
Consequently,
\[
\int_{\partial B_r}
\Pi_{0,0}[\Phi_m+b_0]\cdot\nu\,dS
=
-\alpha_*mb_0.
\]

By \eqref{eq:single-pole-expansion}, $V=\Phi_m+b_0+\widetilde e$ with
\[
|\widetilde e|
+r|D\widetilde e|
+r^2|D^2\widetilde e|
\le Cr^\sigma.
\]
The homogeneity of the terms in
\eqref{eq:pohozaev-current} and the preceding estimate imply
\[
\int_{\partial B_r}
\Bigl(
\Pi_{0,0}[V]
-
\Pi_{0,0}[\Phi_m+b_0]
\Bigr)\cdot\nu\,dS
=o(1)
\qquad(r\downarrow0).
\]
By \eqref{eq:pohozaev-div}, $\Div\Pi_{0,0}[V]=0$ on the punctured
ball, so the boundary integral is independent of $r$.
Taking $r\downarrow0$ and using the preceding $o(1)$ estimate gives
\begin{equation}\label{eq:pohoz-limit}
	\int_{\partial B_r}
	\Pi_{0,0}[V]\cdot\nu\,dS
	=
	-\alpha_*mb_0
\end{equation}
for every sufficiently small $r>0$.
In \cref{prop:C6}, we compare the single-pole boundary integral in
\eqref{eq:pohoz-limit} with the nonnegative boundary integrals of the
smooth first-crossing approximants.

\section{The critical endpoint: neck exclusion, core counting, and completion}\label{sec:critical-global}

Assume $n\ge4k+3$. We use the bubble and single-pole results of \cref{sec:critical-local} to derive a scale drop, count the cores, and establish a recursive inequality for the endpoint defect.

\subsection{Exclusion of a two-core neck and the scale-drop argument}

We use the second selected core to produce a first spherical crossing. A local bound at the crossing then yields the scale drop, including when the rescaled second core escapes every fixed compact set.

Recall $\alpha_*=(n-2k)/(k+1)$ and $\gamma=(n-2k)/k$
from \cref{sec:critical-local}, and set
\[\ell(x)=u(x)^{-1/\alpha_*}.\]
Thus
\begin{equation}\label{eq:critical-scale-relations}
\frac{p_*-k}{2k}=\frac1{\alpha_*},
\qquad
\frac\gamma{\alpha_*}=1+\frac1k,
\qquad
\gamma p_*=n+2.
\end{equation}
Write the bubble tail as
\[
\cB(r)=\kappa_\infty r^{-\gamma}+O(r^{-\gamma-2}),
\qquad
\Phi(x)=\kappa_\infty|x|^{-\gamma}.
\]
By \eqref{eq:C4-bubble-mass}--\eqref{eq:C4-Phi-tail},
$\Phi=\Phi_{m_*}$ and the Hessian mass of $-\Phi$ is $m_*$.

Fix a large separation parameter $\Lambda$.  Use the contact envelope
\[
\underline\ell(x)=\inf_y\left\{\ell(y)+\frac{|x-y|}{100\Lambda}\right\}
\]
and call a point $y\in\R^n$ a contact point if
$
\underline\ell(y)=\ell(y).
$
We extract a maximal selected family of contact points $x_i$ such that
\begin{equation}\label{eq:core-separation}
|x_i-x_j|\ge\Lambda(\ell_i+\ell_j),
\qquad \ell_i=\ell(x_i).
\end{equation}
The covering argument below fixes the constants used later in the core count. The envelope is
$1/(100\Lambda)$-Lipschitz and equals $\ell$ at every contact point.
If a contact point $y$ is not in the selected family, maximality
gives a selected $x_s$ with
$d:=|x_s-y|<\Lambda(\ell_s+\ell(y))$.
Because both points are contacts,
$|\ell_s-\ell(y)|\le d/(100\Lambda)$; thus
$99\ell_s<101\ell(y)<200\ell(y)$ and
$d<3\Lambda\ell(y)$. If $y$ is itself selected, take $s=y$.
This proves the covering property: every contact point $y$ is
associated with a selected core $s$ satisfying
\begin{equation}\label{eq:core-covering}
\ell_s\le2\ell(y),
\qquad
|x_s-y|\le3\Lambda\ell(y).
\end{equation}

\begin{proposition}[Nonlinear two-core neck exclusion]\label[proposition]{prop:C6}
At $p=p_*$, there exist constants $L_0>1$, $\Lambda_0>1$, and $\delta_{\mathrm{core}}>0$, depending only on $n,k$, with the following property. Choose the selected family with separation parameter $\Lambda\ge\Lambda_0$. If $i$ is a selected core satisfying
\[
\ell_i^{\mu_{\mathrm{sc}}}\int_{B_{L_0\ell_i}(x_i)}\cD\,d\mu<\delta_{\mathrm{core}},
\qquad \mu_{\mathrm{sc}}=\frac2{k+1},
\]
and another selected core $j\ne i$ satisfies $\ell_j\ge\ell_i$, then there exists a selected core $s$ such that
\begin{equation}\label{eq:C6scaledrop}
\ell_s\le\frac14\ell_i,
\qquad
|x_s-x_i|\le\frac1{20}|x_j-x_i|.
\end{equation}
\end{proposition}

For any fixed $L_0>1$ and $\delta_{\mathrm{core}}>0$, call a
selected core $i$ good if
\[
\ell_i^{\mu_{\mathrm{sc}}}\int_{B_{L_0\ell_i}(x_i)}\cD\,d\mu
<\delta_{\mathrm{core}},
\]
and bad otherwise. In the counting argument, $L_0$ and $\delta_{\mathrm{core}}$ are the constants supplied by \cref{prop:C6}.

The proof of \cref{prop:C6} is developed in four steps below, followed
by a concluding comparison of Pohozaev boundary integrals. We analyze
a hypothetical failure of the scale drop along good cores with vanishing
normalized defect and increasing separation.

\smallskip
\noindent\emph{Step 1: normalization under failure of the scale drop.}

Fix $L_0>4(R_{\rm turn}+1)$ and put $\mu_{\mathrm{sc}}=2/(k+1)$.
For the preparatory argument fix a positive threshold
$\delta_{\mathrm{core}}$; under the negation of the proposition
these thresholds tend to zero as $\Lambda$ tends to infinity.
Let $i$ be a good core, let $j\ne i$ satisfy $\ell_j\ge\ell_i$, and put
\[
d_{ij}=|x_i-x_j|.
\]
We argue by contradiction under the hypothesis that there is no selected core $s$ such that
\begin{equation}\label{eq:no-smaller-core}
\ell_s\le\frac14\ell_i,
\qquad
|x_s-x_i|\le\frac1{20}d_{ij}.
\end{equation}
The natural normalization
\[
U(x)=\ell_i^{\alpha_*} u(x_i+\ell_i x)
\]
is defined on all of $\R^n$, since $u$ is an entire solution. Its Lorentz
parameter is $\eps_i=\ell_i^{-2(\alpha_*+1)}$, and
$W_{\eps_i}[U](x)=W[u](x_i+\ell_i x)\le W_*$. The exact endpoint-defect scaling is
\begin{equation}\label{eq:C6-defect-scaling}
\int_{B_{L_0}}\cD[U]\,d\mu_{\eps_i}
=\ell_i^{\mu_{\mathrm{sc}}}
\int_{B_{L_0\ell_i}(x_i)}\cD[u]\,d\mu.
\end{equation}
The contact condition at $x_i$ also gives the centering inequality
required in Lemma~\ref{lem:C3}.  Indeed, $\underline\ell(x_i)=\ell_i$ and
$\underline\ell$ is $1/(100\Lambda)$-Lipschitz.  Hence
\[
U(x)^{-1/\alpha_*}
=
\frac{\ell(x_i+\ell_i x)}{\ell_i}
\ge
\frac{\underline\ell(x_i+\ell_i x)}{\ell_i}
\ge
1-\frac{|x|}{100\Lambda}.
\]
Moreover, the universal height bound implies
$
\ell_i\ge U_*^{-1/\alpha_*},
$
and therefore
$
0\le \eps_i
=
\ell_i^{-2(\alpha_*+1)}
\le
U_*^{\,2(\alpha_*+1)/\alpha_*}.
$
Thus the Lorentz parameters are uniformly bounded as required in
Lemma~\ref{lem:C3}.
The scaling identity shows that the good-core condition becomes the fixed-ball defect hypothesis of Lemma~\ref{lem:C3}.  We now normalize at $i$ and write
\[
x_i=0,\qquad \ell_i=1.
\]
Write
\[
L=\frac{d_{ij}}{\ell_i}=d_{ij},
\qquad
S=\frac{\ell_j}{\ell_i}\ge1.
\]
Then \eqref{eq:core-separation} gives
\begin{equation}\label{eq:LSseparation}
L\ge\Lambda(1+S),
\qquad
\frac SL\le\Lambda^{-1}.
\end{equation}

Take
\[
K_0=4000^{\alpha_*}.
\]
If $U(z)>K_0$ at a point $|z|<L/1000$, then the natural length at $z$ is less than $1/4000$.  Applying the contact envelope and \eqref{eq:core-covering} produces a selected core $s$ with $\ell_s<1/2000$ and
\[
|x_s|\le \frac L{1000}+\frac{103\Lambda}{4000}<\frac L{20},
\]
where the last inequality follows from \eqref{eq:LSseparation}.  This contradicts \eqref{eq:no-smaller-core}.  Hence
\begin{equation}\label{eq:protected-height}
\quad U\le K_0\quad\text{in }B_{L/1000}.
\end{equation}

To prove the existence of fixed constants for this $L_0$, suppose a contradiction sequence has $\Lambda_j\to\infty$ and good cores whose normalized defects in \eqref{eq:C6-defect-scaling} tend to zero.  The protected bound \eqref{eq:protected-height} holds on balls whose radii tend to infinity, and Lemma~\ref{lem:C3} then yields
\begin{equation}\label{eq:Uglobalbubble}
U_j\to\cB\quad\text{locally in }\mathbb R^n.
\end{equation}

\smallskip
\noindent\emph{Step 2: the second-core barrier and the first spherical crossing.}

Lemma~\ref{lem:inner-pole-lower} supplies the lower bound from the first core. A second radial barrier supplies an additional lower bound centered at the other core.

The solution has a uniform lower height bound on the natural ball around the second selected core.  Lemma \ref{lem:C2} applied in the $j$-core normalization gives
\[
U_j(x)\ge cS_j^{-\alpha_*}
\quad\text{on }B_{S_j/2}(L_j\omega_j),
\]
where $|\omega_j|=1$.  Center a radial zero-curvature barrier at $L_j\omega_j$ with slope variable
\[
y(r)=a_jr^{-\gamma-1},
\qquad a_j\asymp S_j^{\gamma-\alpha_*},
\]
chosen sufficiently small so that its value on $r=S_j/2$ is below the preceding Harnack lower bound. The denominator $\sqrt{1+\eps_j y(r)^2}$ in this radial barrier is uniformly controlled.
Indeed, the parameter in the second-core normalization is
$\widehat\eps_j:=\eps_jS_j^{-2(\alpha_*+1)}\le C$ by the universal
height bound. Thus for $r\ge S_j/2$,
$y(r)\le CS_j^{-\alpha_*-1}$ and
$\eps_jy(r)^2\le C\widehat\eps_j\le C$; in particular
$(1+\eps_jy(r)^2)^{-1/2}\ge c>0$ with constants independent of $j$.
It follows that the corresponding radial zero-curvature barrier satisfies
\[
\int_r^\infty
\frac{y(s)}
{\sqrt{1+\eps_jy(s)^2}}\,ds
\ge
ca_jr^{-\gamma}
\asymp
cS_j^{\gamma-\alpha_*}r^{-\gamma}.
\]
The normalized solutions are entire, so the same truncated-annulus comparison, followed by an outer-radius limit, yields
\begin{equation}\label{eq:second-core-barrier}
U_j(x)\ge c_bS_j^{\gamma-\alpha_*}
|x-L_j\omega_j|^{-\gamma}
\end{equation}
whenever $|x-L_j\omega_j|\ge S_j$.

Scale by the inter-core distance:
\[
Z_j(y)=L_j^\gamma U_j(L_jy).
\]
Define the lower half-relaxed limit of $(Z_j)$ by
\[
Z_*(y):=\lim_{m\to\infty}
\inf\bigl\{Z_j(z):j\ge m,\ |z-y|<1/m\bigr\}
.
\]
By \eqref{eq:inner-pole-liminf}, \eqref{eq:second-core-barrier}, and $S_j\ge1$, the lower half-relaxed limit $Z_*$ satisfies
\begin{equation}\label{eq:two-core-liminf}
Z_*\ge\Phi,
\qquad
Z_*(y)\ge c_b|y-\omega|^{-\gamma}
\end{equation}
away from the two distinguished points, after taking $\omega_j\to\omega$.

Let
\[
T_\Phi=T_{k-1}(-D^2\Phi).
\]
The matrix $T_\Phi$ is strictly positive definite away from the origin.
Let $\zeta\in C^2$ touch $Z_*-\Phi$ from below at a point
$y\ne0$ where $Z_*(y)<\infty$, and put $\varphi=\Phi+\zeta$.
After making the contact strict and passing to a subsequence, the
half-relaxed-limit argument gives points $y_j\to y$ and constants
$c_j$ such that $\varphi+c_j$ touches $Z_j$ from below at $y_j$.
The Lorentz parameter of $Z_j$ is
$\delta_j=\eps_jL_j^{-2(\gamma+1)}\to0$.
At each contact point the gradients agree and
$D^2\varphi(y_j)\preceq D^2Z_j(y_j)$.
Thus the symmetric shape representative of $\varphi+c_j$
dominates that of $Z_j$ in their common graph metric and belongs
to $\overline{\Gamma_k}$. Passing to the limit gives
$-D^2\varphi(y)\in\overline{\Gamma_k}$.

Since $-D^2\Phi(y)\in\partial\Gamma_k$, its supporting
functional is $T_\Phi(y)=T_{k-1}(-D^2\Phi(y))$, and
$T_\Phi:(-D^2\Phi)=0$. Consequently
\[
0\le T_\Phi(y):(-D^2\varphi(y))
=-T_\Phi(y):D^2\zeta(y).
\]
As $T_\Phi$ is strictly positive definite away from the origin,
this proves
\begin{equation}\label{eq:Lphi-super}
	\mathcal L_\Phi(Z_*-\Phi)
	:=T_\Phi:D^2(Z_*-\Phi)\le0
\end{equation}
in the viscosity sense wherever $Z_*$ is finite.
The second inequality in \eqref{eq:two-core-liminf} allows us to choose
$\eta>0$ so small that
$Z_*-\Phi\ge1$
on $\partial B_\eta(\omega).$
Fix
$q_0=10^{-4},y_0=-q_0\omega.$
Choose a smooth bounded domain $\Omega_0$ such that
\[
y_0\in
\Omega':=
\Omega_0\setminus\overline{B_\eta(\omega)},
\qquad
0\notin\overline{\Omega_0},
\qquad
\overline{B_\eta(\omega)}\subset\Omega_0.
\]
Let $h$ solve
\[
\begin{cases}
	\mathcal L_\Phi h=0 & \text{in }\Omega',\\
	h=1 & \text{on }\partial B_\eta(\omega),\\
	h=0 & \text{on }\partial\Omega_0.
\end{cases}
\]
Then $h>0$ in $\Omega'$ by the strong maximum principle.

To avoid any issue at points where the lower half-relaxed limit may be
infinite, set
$
\widetilde Z
=
\min\{Z_*-\Phi,2\}.
$
The minimum of a viscosity supersolution and a constant supersolution is
again a viscosity supersolution for the linear operator
$\mathcal L_\Phi$.  Hence
$
\mathcal L_\Phi\widetilde Z\le0
\text{ in }\Omega',
$
and the boundary inequalities give
$\widetilde Z\ge h$
 in $\Omega'.$
Consequently
$Z_*-\Phi\ge h$
 at $y_0.$ 
Hence, for the fixed choice
\begin{equation}\label{eq:strict-two-core-excess}
Z_*(y_0)\ge\Phi(y_0)+c_0,
\qquad c_0>0.
\end{equation}
Here $c_0=h(y_0)>0$ by the strong maximum principle.

Choose once and for all
\begin{equation}\label{eq:Theta-choice}
1<\Theta<1+\frac{c_0}{2\Phi(y_0)}.
\end{equation}
Since $L_j^\gamma\cB(L_jy_0)\to\Phi(y_0)$, \eqref{eq:strict-two-core-excess} implies
$
U_j(L_jy_0)>\Theta\cB(L_j|y_0|)
$
for large $j$.  On the other hand \eqref{eq:Uglobalbubble} gives $U_j/\cB\to1$ on every fixed ball.  Hence there is a first spherical crossing radius $\rho_j$ such that
\begin{equation}\label{eq:first-spherical-crossing}
\begin{gathered}
U_j<\Theta\cB\quad\text{in }B_{\rho_j},\\
U_j(\rho_j\xi_j)=\Theta\cB(\rho_j),
\qquad |\xi_j|=1,\\
\rho_j\to\infty,
\qquad
\rho_j\le q_0L_j.
\end{gathered}
\end{equation}

\smallskip
\noindent\emph{Step 3: compactness at the first crossing.}

Define the first-crossing blow-down
\[
V_j(x)=\rho_j^\gamma U_j(\rho_jx),
\qquad
A_j=\frac{L_j}{\rho_j}\ge q_0^{-1}.
\]
The second core is centered at
\[
P_j=A_j\omega_j,
\qquad
r_j^{\rm core}=\frac{S_j}{\rho_j},
\]
and \eqref{eq:LSseparation} gives
\begin{equation}\label{eq:relative-core-radius}
\frac{r_j^{\rm core}}{|P_j|}=\frac{S_j}{L_j}\longrightarrow0.
\end{equation}
After extraction, either (i) $A_j\to A_\infty<\infty$, when the second core has a possible limit $P_\infty=A_\infty\omega$ outside $B_2$ and $r_j^{\rm core}\to0$, or (ii) $A_j\to\infty$, when its core ball leaves every fixed compact set. In either case, \eqref{eq:protected-height} and $A_j\ge q_0^{-1}$ give, for the fixed sufficiently small $q_0$,
\begin{equation}\label{eq:coarse-contact-height}
0<V_j\le K_0\rho_j^\gamma\quad\text{in }B_4.
\end{equation}
The next lemma preserves the strict crossing inequality in both cases.

\begin{lemma}[Local boundedness at the crossing point]\label[lemma]{lem:crossing-point}
Let $V_j>0$ be smooth strictly $k$-admissible solutions satisfying
\[
\lambda_j|DV_j|^2<1,\qquad
A_{\lambda_j}[V_j]\in\Gamma_k,\qquad
\sigma_k(A_{\lambda_j}[V_j])=\rho_j^{-2}V_j^{p_*}
\quad\text{in }B_4,
\]
where
\begin{equation}\label{eq:crossing-lambda-scaling}
\lambda_j=\eps_j\rho_j^{-2(\gamma+1)},
\qquad 0\le\eps_j\le\eps_0,
\end{equation}
and assume $W_j\le W_*$, $\rho_j\to\infty$, and
\[
0<V_j\le K_0\rho_j^\gamma.
\]
Assume that $\xi_j\in\partial B_1$ satisfies
\[
0<c_*\le V_j(\xi_j)\le C_*.
\]
Then there are fixed $\delta,C>0$, depending only on
$n,k,p_*,\eps_0,W_*,K_0,c_*,C_*$, such that
\begin{equation}\label{eq:crossing-point-conclusion}
\sup_{B_\delta(\xi_j)}V_j\le C,
\qquad
\sup_{B_\delta(\xi_j)}|D\log V_j|\le C.
\end{equation}
\end{lemma}

\begin{proof}
Set
\[
\mathcal M_j(x)=\rho_j^{-1/k}V_j(x)^{1/\alpha_*}.
\]
Suppose that $\mathcal M_j$ were unbounded in a fixed small neighborhood of $\xi_j$.  A standard doubling selection yields points $z_j$ with
\[
H_j:=V_j(z_j),
\qquad
s_j=\mathcal M_j(z_j)^{-1}\to0,
\]
such that
\begin{equation}\label{eq:doubling-peak}
V_j\le2^{\alpha_*} H_j
\quad\text{in }B_{2s_j}(z_j).
\end{equation}
Because
\[
s_j=\rho_j^{1/k}H_j^{-1/\alpha_*},
\]
Lemma \ref{lem:C2}, applied on the scale $s_j$, gives
\[
|D\log V_j|\le Cs_j^{-1}
\quad\text{in }B_{s_j}(z_j).
\]
Hence
\begin{equation}\label{eq:solid-peak}
V_j\ge c_1H_j
\quad\text{in }B_{s_j/2}(z_j).
\end{equation}

Construct a radial zero-curvature barrier around $z_j$ by setting
\[
y_j(r)=c_2H_js_j^\gamma r^{-\gamma-1},
\qquad
G_j(r)=\int_r^{r_0}
\frac{y_j(t)}{\sqrt{1+\lambda_jy_j(t)^2}}\,dt,
\]
so that $G_j(r_0)=0$.  The two scale identities
\begin{equation}\label{eq:crossing-point-scale-identities}
H_js_j^\gamma
=\rho_j^{\gamma/k}H_j^{-1/k}
\ge K_0^{-1/k},
\end{equation}
\[
\lambda_j\frac{H_j^2}{s_j^2}
=\varepsilon_j
\left(\frac{H_j}{\rho_j^\gamma}\right)^{2+2/\alpha_*}
\le C(K_0)
\]
show, respectively, that the fundamental-solution coefficient of the barrier does not vanish and that the Lorentz denominator remains uniformly controlled.  Choosing $c_2$ sufficiently small, \eqref{eq:solid-peak} dominates the barrier on the inner boundary.  Comparison then gives
\begin{equation}\label{eq:crossing-point-barrier}
V_j(x)\ge c_3|x-z_j|^{-\gamma}
\end{equation}
whenever $s_j/2\le|x-z_j|\le r_0/2$.

Choose the contact neighborhood so small that $c_3(3\delta)^{-\gamma}>2C_*$.  If $\xi_j$ lies in the core ball $B_{s_j/2}(z_j)$, \eqref{eq:solid-peak} contradicts the bounded contact value.  Otherwise \eqref{eq:crossing-point-barrier} gives the same contradiction.  Thus $\mathcal M_j$ is uniformly bounded near $\xi_j$.  Lemma \ref{lem:C2} now yields the second estimate in \eqref{eq:crossing-point-conclusion}; integrating the logarithmic gradient from the bounded nonzero value at $\xi_j$ gives the first.
\end{proof}

At the crossing point,
\begin{equation}\label{eq:contact-value-limit}
V_j(\xi_j)=\rho_j^\gamma\Theta\cB(\rho_j)\longrightarrow\Theta\kappa_\infty.
\end{equation}
Lemma \ref{lem:crossing-point} therefore applies.  After taking $\xi_j\to\xi\in\partial B_1$, the first-crossing limit extends to a fixed neighborhood of $\xi$ and satisfies
\begin{equation}\label{eq:strict-contact-limit}
V(\xi)=\Theta\kappa_\infty>\Phi(\xi).
\end{equation}
The strict inequality holds in both cases. In case (i), any external singularity lies outside $B_2$; in case (ii), the second core leaves every fixed compact set.

\smallskip
\noindent\emph{Step 4: pole concentration and positivity of the regular part.}

The first-crossing inequality gives, on every fixed punctured compact subset of $B_1$,
\[
V_j(x)\le C|x|^{-\gamma},
\]
and Lemma \ref{lem:C2} gives the corresponding gradient bound.  Moreover
\[
U_j^{p_*}\le\Theta^{p_*}\cB^{p_*}
\quad\text{in }B_{\rho_j},
\]
so, using \eqref{eq:critical-scale-relations}, for each fixed $0<r<1$,
\[
\int_{B_{r\rho_j}}U_j^{p_*}\,dy\longrightarrow
\int_{\mathbb R^n}\cB^{p_*}\,dy=m_*.
\]
The first-crossing inequality supplies the source domination
\eqref{eq:C4-source-domination}. Before invoking Lemma~\ref{lem:C4},
we verify its additional lower-bound hypothesis: the normalized solutions are
entire, and the independently proved exact inner-pole barrier
\eqref{eq:inner-pole-liminf}, applied with $t_j=\rho_j$ and
\eqref{eq:C4-Phi-tail}, gives, for every
$K\Subset B_1\setminus\{0\}$,
\[
\liminf_{j\to\infty}\inf_K(V_j-\Phi_{m_*})\ge0.
\]
This verifies \eqref{eq:C4-lower-input} before the concentration lemma is applied. Lemma~\ref{lem:C4} now applies and gives
\begin{equation}\label{eq:C6singleatom}
\mu_k[-V]=m_*\delta_0
\quad\text{in }B_1.
\end{equation}
The same lower barrier, together with local convergence near the
crossing point, also gives
\begin{equation}\label{eq:C6Phi-lower}
V\ge\Phi
\end{equation}
throughout the connected open set
\[
\Omega_*=(B_1\cup B_{\delta/2}(\xi))\setminus\{0\}.
\]
In particular Lemma \ref{lem:C5} applies at the pole and yields
\begin{equation}\label{eq:C6regularconstant}
V=\Phi+b_0+O_{C^2_{\rm weighted}}(|x|^\sigma),
\qquad b_0\ge0.
\end{equation}

Set $e=V-\Phi$.  The supporting hyperplane to $\overline{\Gamma_k}$ at $-D^2\Phi$ gives
\begin{equation}\label{eq:C6support}
\mathcal L_\Phi e:=T_{k-1}(-D^2\Phi):D^2e\le0
\quad\text{in }\Omega_*.
\end{equation}
The coefficients are strictly elliptic on compact subsets of $\Omega_*$.  By \eqref{eq:C6Phi-lower}, $e\ge0$, while \eqref{eq:strict-contact-limit} shows that $e$ is not identically zero.  The strong minimum principle, applied on the connected set $\Omega_*$, gives $e>0$ throughout $\Omega_*$.  Hence
\begin{equation}\label{eq:b1positive}
b_1:=\min_{\partial B_{1/2}}e>0.
\end{equation}

Since $\mathcal L_\Phi(|x|^{-\gamma})=0$, for $\epsilon>0$ the function
\[
e+\epsilon|x|^{-\gamma}-b_1
\]
is a supersolution of \eqref{eq:C6support}.  It is nonnegative on $\partial B_{1/2}$ and, by \eqref{eq:C6regularconstant}, on a sufficiently small inner sphere.  The annular minimum principle, followed first by shrinking the inner sphere and then by $\epsilon\downarrow0$, gives $e\ge b_1$ in $B_{1/2}\setminus\{0\}$.  Passing to the pole in \eqref{eq:C6regularconstant} yields
\begin{equation}\label{eq:C6b0positive}
{\quad b_0\ge b_1>0.\quad}
\end{equation}

\begin{proof}[Proof of \cref{prop:C6}]
If no constants $\Lambda_0,\delta_{\mathrm{core}}$ worked for this fixed $L_0$, we could choose counterexamples with $\Lambda_j\to\infty$ and normalized good-core defects tending to zero. By \eqref{eq:C6-defect-scaling}, these normalized solutions satisfy the small-defect compactness hypotheses verified in Step~1. The second-core barrier of Step~2 and the crossing-point compactness of Step~3 produce a first-crossing limit; Step~4 gives $b_0>0$ by \eqref{eq:C6b0positive}. On the other hand, every first-crossing approximant is smooth at the origin and its critical Pohozaev divergence is nonnegative. The single-pole expansion \cref{lem:C5} and the fixed-annulus
convergence \cref{lem:C5-annulus}, applied on a fixed sufficiently
small sphere, identify the limit of the Pohozaev boundary
integrals. Nonnegativity of these boundary integrals for the smooth approximants gives
\begin{equation}\label{eq:b0nonpos}
b_0\le0.
\end{equation}
More precisely,
\[
0\le\lim_{j\to\infty}
\int_{\partial B_r}\Pi_j\cdot\nu\,dS
=-\alpha_* m_*b_0<0,
\]
a contradiction.  Hence fixed $L_0,\Lambda_0,\delta_{\mathrm{core}}$ exist and \eqref{eq:C6scaledrop} follows.
\end{proof}

\subsection{Core counting and a recursive defect estimate}
\label{sec:C8-high-dim}

Retain the good/bad classification and fix the constants given in
\cref{prop:C6}. Since $n\ge4k+3$,
\[
-1<M_*<0,\qquad
\delta_*:=-M_*-1=\frac{4k-n}{n-2k}\in(-1,0).
\]
With $\mu_{\mathrm{sc}}=2/(k+1)$ as above, set $\nu=2k/(k+1)$, so $\mu_{\mathrm{sc}}+\nu=2$.
For the endpoint height exponent $a=M_*+p_*+1$, one has
\begin{equation}\label{eq:alphaa}
\alpha_* a=n-\nu.
\end{equation}
We fix $\Lambda\ge\Lambda_0$ with $\Lambda>2L_0$; the enlarged defect balls
are then pairwise disjoint. Set
\[C_0=8(1+103\Lambda+L_0)>1.\]
The counting argument uses the scale drop
\eqref{eq:C6scaledrop} directly. Constants in this subsection may
depend on $n,k,\Lambda,L_0,\delta_{\mathrm{core}}$ and the
universal a priori bounds, but are independent of $R$ and of the
particular entire solution.

For $R\ge2$ define
\begin{align}
\mathcal C_R
&=\left\{i:\ |x_i|\le(1+103\Lambda)R,
\ \ell_i\le2R\right\},\label{eq:CRdef}\\
\mathcal C_R^{\rm bad}
&=\left\{i:\ i\text{ bad},\ |x_i|\le2(1+103\Lambda)R,
\ \ell_i\le2R\right\},\label{eq:CRbaddef}
\end{align}
and let
\[
N_R=\#\mathcal C_R,
\qquad
N_R^{\rm bad}=\#\mathcal C_R^{\rm bad}.
\]

\begin{lemma}[Core counting and a recursive defect estimate]\label[lemma]{lem:C8}
For all sufficiently large $R$,
\begin{equation}\label{eq:count1}
N_R\le C\left[1+\log(e+R)N_R^{\rm bad}\right],
\end{equation}
\begin{equation}\label{eq:count2}
N_R^{\rm bad}
\le C\left[R^{\mu_{\mathrm{sc}}} D(C_0R)\right]^{n/(n+\mu_{\mathrm{sc}})}.
\end{equation}
Moreover,
\begin{equation}\label{eq:Ia-core}
I_a(R):=\int_{B_R}u^a\,d\mu
\le CR^\nu(1+N_R).
\end{equation}
The finite Newton-tensor descent gives
\begin{equation}\label{eq:N-high-dim}
\mathcal N(R)
\le C I_a(2R)+CR^\nu,
\end{equation}
and consequently
\begin{equation}\label{eq:feedback}
D(R)
\le CR^{-\mu_{\mathrm{sc}}}
+C\log(e+R)R^{-\mu_{\mathrm{sc}}^2/(n+\mu_{\mathrm{sc}})}
D(C_0R)^{n/(n+\mu_{\mathrm{sc}})}.
\end{equation}
The defect has at most polynomial growth, and \eqref{eq:feedback} implies $D\equiv0$.
\end{lemma}

\begin{proof}
\emph{Step 1: parent--descendant structure.}
For a core $i$, let
\[
\mathcal P(i)=\{j:\ \ell_j\le\ell_i/4\}.
\]
We refer to an element of $\mathcal P(i)$ as a finer-scale parent of $i$;
thus following the parent map moves toward smaller natural lengths.
By the universal height bound there is a fixed constant
$\ell_{\min}>0$ such that
\[
\ell_i\ge \ell_{\min}
\qquad\text{for every selected core }i.
\]
Together with the separation condition,
\[
|x_i-x_j|
\ge \Lambda(\ell_i+\ell_j)
\ge 2\Lambda\ell_{\min},
\]
this implies that the selected family is locally finite.
Hence, if $\mathcal P(i)\ne\varnothing$, there exists a closest
finer-scale parent; choose one and denote it by $p(i)$, and set
\[
e_i=|x_i-x_{p(i)}|.
\]
If $\mathcal P(i)=\varnothing$, call $i$ a root.

Suppose $i$ is good and has a child $j$, so $p(j)=i$.  Then $\ell_i\le\ell_j/4<\ell_j$, and Proposition~\ref{prop:C6} applied to $(i,j)$ produces a core $k$ with
\[
\ell_k\le\ell_i/4,
\qquad
|x_i-x_k|\le\frac1{20}|x_i-x_j|=\frac1{20}e_j.
\]
By the minimality of $p(i)$,
\begin{equation}\label{eq:edge-shrink}
e_i\le\frac1{20}e_j.
\end{equation}
In particular a good root cannot have a child.  There is at most one good root: if $i,j$ were two distinct good roots and $\ell_i\le\ell_j$, Proposition~\ref{prop:C6} would create an admissible parent for $i$.  Every parent chain is finite because the universal height bound gives a uniform positive lower bound for all $\ell_i$.

We now count the cores in $\mathcal C_R$.  If there is a unique core of maximal natural length in $\mathcal C_R$, place it in an exceptional set; also place the unique good root there if it belongs to $\mathcal C_R$.  Thus the exceptional set has cardinality at most two.  Let $i$ be any remaining good core in $\mathcal C_R$.  There is another core $j\in\mathcal C_R$ with $j\ne i$ and $\ell_i\le\ell_j$.  Since both centers lie in $B_{(1+103\Lambda)R}$, Proposition~\ref{prop:C6} and the closest-parent definition give
\begin{equation}\label{eq:first-edge-bound}
e_i\le\frac1{20}|x_i-x_j|
\le\frac1{10}(1+103\Lambda)R.
\end{equation}
Follow the parent chain until the first bad core $b=b(i)$ is reached.  Such a bad ancestor must exist: otherwise the finite chain would terminate at a good root having a child.  If
\[
i=i_0,\ i_1=p(i_0),\ldots,i_m=b,
\]
then $i_0,\ldots,i_{m-1}$ are good, and repeated use of \eqref{eq:edge-shrink} yields
\[
e_{i_s}\le20^{-s}e_i.
\]
Hence
\begin{equation}\label{eq:ancestor-distance}
|x_i-x_b|
\le e_i\sum_{s=0}^\infty20^{-s}
\le\frac{20}{19}e_i.
\end{equation}
By \eqref{eq:first-edge-bound}, $|x_b|\le2(1+103\Lambda)R$, while $\ell_b\le\ell_i\le2R$.  Thus
\[
b(i)\in\mathcal C_R^{\rm bad}.
\]

Fix $b\in\mathcal C_R^{\rm bad}$.  For $r>0$ let
\[
\mathcal I_b(r)
=\{i\in\mathcal C_R:\ i\text{ good},\ b(i)=b,\ r\le e_i<2r\}.
\]
Each dyadic edge class assigned to $b$ contains at most one good core. Indeed, if distinct $i,j\in\mathcal I_b(r)$ satisfy $\ell_i\le\ell_j$, Proposition~\ref{prop:C6} and the minimality of $p(i)$ give
\[
e_i\le\frac1{20}|x_i-x_j|,
\]
so $|x_i-x_j|\ge20r$.  On the other hand \eqref{eq:ancestor-distance} gives
\[
|x_i-x_b|<\frac{40}{19}r,
\qquad
|x_j-x_b|<\frac{40}{19}r,
\]
and therefore $|x_i-x_j|<80r/19<20r$, a contradiction.  Thus each dyadic edge class contains at most one good core assigned to $b$.

The separation bound above gives $e_i\ge2\Lambda\ell_{\min}$,
while \eqref{eq:first-edge-bound} gives $e_i\le CR$.
Hence only $C\log(e+R)$ dyadic edge classes can occur. Every bad core in $\mathcal C_R$ is counted in $N_R^{\rm bad}$. Each
$b\in\mathcal C_R^{\rm bad}$ supports at most $C\log(e+R)$ good cores
in $\mathcal C_R$, apart from the uniformly bounded exceptional set. Therefore
\[
N_R
\le
N_R^{\rm bad}
+C\log(e+R)N_R^{\rm bad}
+C,
\]
which yields \eqref{eq:count1}.

\smallskip
\noindent\emph{Step 2: bad-core count.}
The choice of $C_0$ above ensures that every ball
$B_{L_0\ell_i}(x_i)$ with $i\in\mathcal C_R^{\rm bad}$
is contained in $B_{C_0R}$.  Since $i$ is bad,
\[
\ell_i^{-\mu_{\mathrm{sc}}}
\le\delta_{\mathrm{core}}^{-1}
\int_{B_{L_0\ell_i}(x_i)}\cD\,d\mu.
\]
The enlarged balls are pairwise disjoint, hence
\begin{equation}\label{eq:bad-costs}
\sum_{i\in\mathcal C_R^{\rm bad}}\ell_i^{-\mu_{\mathrm{sc}}}
\le CD(C_0R),
\qquad
\sum_{i\in\mathcal C_R^{\rm bad}}\ell_i^n
\le CR^n.
\end{equation}
Indeed, the first inequality follows by summing the bad-core costs, while the second is the Euclidean volume packing estimate for the disjoint balls $B_{L_0\ell_i}(x_i)\subset B_{C_0R}$.
Writing
\[
1=(\ell_i^{-\mu_{\mathrm{sc}}})^{n/(n+\mu_{\mathrm{sc}})}(\ell_i^n)^{\mu_{\mathrm{sc}}/(n+\mu_{\mathrm{sc}})}
\]
and applying H\"older proves \eqref{eq:count2}.

\smallskip
\noindent\emph{Step 3: height mass from the contact cover.}
Because $\ell=u^{-1/\alpha_*}$ and \eqref{eq:alphaa} holds,
\[
u^a=\ell^{-(n-\nu)}.
\]
On $B_R\cap\{\ell\ge R\}$,
\[
\int u^a\,d\mu\le CR^{-(n-\nu)}|B_R|\le CR^\nu.
\]
Let $x\in B_R$ with $\ell(x)<R$. Continuity and the coercive
distance term ensure that $\underline\ell(x)$ has a minimizer $y$, with
$\ell(y)+|x-y|/(100\Lambda)=\underline\ell(x)\le\ell(x)<R$.
The minimizer $y$ is a contact point: if $\underline\ell(y)<\ell(y)$, some $z$ has
$\ell(z)+|z-y|/(100\Lambda)<\ell(y)$; the triangle inequality then
gives $\ell(z)+|x-z|/(100\Lambda)<\underline\ell(x)$, a contradiction.
The covering property gives a selected core $i$ such that
\[
\ell_i\le2\ell(y),
\qquad
|x_i-y|\le3\Lambda\ell(y).
\]
Here $\ell_i<2R$ and
$|x_i|\le|x|+|x-y|+3\Lambda\ell(y)<(1+103\Lambda)R$,
so $i\in\mathcal C_R$. Moreover,
\begin{equation}\label{eq:length-kernel-comparison}
\ell_i+\frac{|x-x_i|}{100\Lambda}
\le\frac{203}{100}\,\ell(x).
\end{equation}
Consequently
\[
u(x)^a\mathbf1_{\{\ell(x)<R\}}
\le C\sum_{i\in\mathcal C_R}
\left(\ell_i+\frac{|x-x_i|}{100\Lambda}\right)^{-(n-\nu)}.
\]
Since $d\mu=W^{-1}dx\le dx$ and $|x-x_i|\le CR$ for $x\in B_R$, Euclidean polar integration around $x_i$ (using $n-\nu<n$) gives
\[
\int_{B_R}
\left(\ell_i+\frac{|x-x_i|}{100\Lambda}\right)^{-(n-\nu)}dx
\le CR^\nu.
\]
Summing over $i\in\mathcal C_R$ proves \eqref{eq:Ia-core}.

\smallskip
\noindent\emph{Step 4: the finite descent in the range $n\ge4k+3$.}
Here
\[
\delta_*=-M_*-1\in(-1,0).
\]
Fix once and for all a cutoff power $\theta>2k$, and let
$\chi_R$ be the standard cutoff equal to one on $B_R$ and
supported in $B_{2R}$. Since $q\le z$, the first
curvature-weighted quantity in Lemma~\ref{lem:descent}
controls $\mathcal N(R)$.  The negative-$\delta$ form of that lemma gives
\begin{equation}\label{eq:C8-descent}
\mathcal N(R)
\le C I_a(2R)
+CR^{-2k}\int_{B_{2R}}u^b\,d\mu,
\qquad
b=M_*+k+1.
\end{equation}
Because $n\ge4k+3$, one has $b>0$, and
\[
a-b=p_*-k=\frac{2k(k+1)}{n-2k}>0.
\]
Young's inequality on $B_{2R}$ gives
\[
R^{-2k}\int_{B_{2R}}u^b\,d\mu
\le\frac12 I_a(2R)
+CR^{n-\frac{2ka}{a-b}}.
\]
The two $I_a(2R)$ terms on the right of \eqref{eq:C8-descent} can be combined. At the critical exponent,
\[
n-\frac{2ka}{a-b}
=\frac{2k}{k+1}
=\nu.
\]
Thus
\[
\mathcal N(R)\le C I_a(2R)+CR^\nu,
\]
which is \eqref{eq:N-high-dim}.

Apply \eqref{eq:N-high-dim} at $2R$ in \eqref{eq:defect-cutoff};
this produces $I_a(4R)$. We then apply \eqref{eq:Ia-core} and
\eqref{eq:count1} at the radius $4R$. For the bad cores in
$\mathcal C_{4R}^{\rm bad}$, we repeat the estimate of Step~2
directly. Indeed, if $i\in\mathcal C_{4R}^{\rm bad}$, then
\[
|x_i|\le 8(1+103\Lambda)R,
\qquad
\ell_i\le 8R,
\]
and hence
\[
B_{L_0\ell_i}(x_i)
\subset
B_{\,8(1+103\Lambda+L_0)R}
=B_{C_0R}.
\]
Therefore
\[
N_{4R}^{\rm bad}
\le
C\left[
R^{\mu_{\mathrm{sc}}}D(C_0R)
\right]^{n/(n+\mu_{\mathrm{sc}})}.
\]
Consequently,
\begin{align*}
D(R)
&\le CR^{-2}[I_a(4R)+R^\nu]\\
&\le CR^{-\mu_{\mathrm{sc}}}
+C\log(e+R)R^{-\mu_{\mathrm{sc}}^2/(n+\mu_{\mathrm{sc}})}D(C_0R)^{n/(n+\mu_{\mathrm{sc}})},
\end{align*}
which is \eqref{eq:feedback}.

\smallskip
\noindent\emph{Step 5: Souplet radius selection.}
We use a radius-selection argument of the type in \cite{Souplet}.
The universal height bound gives $I_a(R)\le CR^n$, and hence \eqref{eq:N-high-dim} together with \eqref{eq:defect-cutoff} gives the explicit polynomial bound
\[
D(R)\le CR^{n-2}
\qquad(R\ge2).
\]
Suppose $D\not\equiv0$. Since $D$ is nonnegative and
nondecreasing, there exists $R_0\ge2$ such that
$D(R_0)>0$. Choose $A>C_0^{n-2}$.  If
\[
D(C_0R)>AD(R)
\]
held for every sufficiently large $R$, iteration would give $D(C_0^mR)>A^mD(R)$, contradicting the polynomial bound as $m\to\infty$.  Hence there are $R_j\to\infty$ with
\[
D(C_0R_j)\le AD(R_j).
\]
Substituting this into \eqref{eq:feedback}, set
\[
\theta=\frac{n}{n+\mu_{\mathrm{sc}}},
\qquad
1-\theta=\frac{\mu_{\mathrm{sc}}}{n+\mu_{\mathrm{sc}}}.
\]
Young's inequality with exponents $1/\theta=(n+\mu_{\mathrm{sc}})/n$ and $1/(1-\theta)=(n+\mu_{\mathrm{sc}})/\mu_{\mathrm{sc}}$ gives
\[
C\log(e+R_j)R_j^{-\mu_{\mathrm{sc}}^2/(n+\mu_{\mathrm{sc}})}D(R_j)^\theta
\le\frac12D(R_j)
+C R_j^{-\mu_{\mathrm{sc}}}\log^{(n+\mu_{\mathrm{sc}})/\mu_{\mathrm{sc}}}(e+R_j).
\]
After absorbing the first term,
\[
D(R_j)
\le
CR_j^{-\mu_{\mathrm{sc}}}
\left[1+\log^{(n+\mu_{\mathrm{sc}})/\mu_{\mathrm{sc}}}(e+R_j)\right]
\longrightarrow0.
\]
For any fixed $r>0$, choose $j$ so large that $r<R_j$.
Since $D$ is nonnegative and nondecreasing,
$
0\le D(r)\le D(R_j)\longrightarrow0.
$
Thus $D(r)=0$ for every $r>0$, and therefore
$
D\equiv0.
$
\end{proof}

\subsection{Completion of the Liouville theorem}\label{subsec:completion}
\begin{proof}[Completion of the proof of \cref{thm:liouville}]
\Cref{sec:liouville} proves the assertion for $k\le p<\pstar$. For $p=\pstar$, we distinguish the two dimension ranges.  If $2k<n\le4k+2$, \cref{prop:direct-low-dim} already gives $u\equiv0$.  When $n\ge4k+3$, \cref{lem:C8} gives
\[
 D(R)=0\qquad\text{for every }R>0.
\]
Since $\cD\ge0$, the vanishing of every defect integral gives $\cD\equiv0$.  Returning to the exact endpoint decomposition \eqref{eq:flux-square}, the quartic term has coefficient $k\tau_*/4>0$ and $u>0$ under the nontriviality assumption.  Hence $q\equiv0$.  Therefore $W\equiv1$, $Du\equiv0$, and $u$ is constant.  The equation then gives $0=\sigma_k(A[u])=u^{\pstar}$, so $u\equiv0$.  Together with the subcritical proof, the cases $p=k$, $k<p<p_*$,
$p=p_*$ with $2k<n\le4k+2$, and $p=p_*$ with $n\ge4k+3$
cover the full range in \eqref{eq:range}.
\end{proof}

\subsection{Geometric consequence}
To prove the geometric corollary, we show that a complete spacelike
immersion admits an entire graph representation. Completeness is used
to obtain this representation, after which \cref{thm:liouville} applies.

\begin{proof}[Proof of \cref{cor:geometric}]
Apply an ambient translation and a time-orientation-preserving Lorentz
isometry sending $\mathsf P$ to $\{t=0\}$ and $N_0$ to $(0,1)$.
The signed height in the statement then becomes the time coordinate,
and the future normal and the curvature convention are preserved.
Let $\pi:\Sigma\to\R^n$ be the spatial projection of the immersion.
For every tangent vector $Y$,
\[
 |d\pi(Y)|_{\R^n}^2=|Y|_\Sigma^2+|du(Y)|^2\geq|Y|_\Sigma^2.
\]
Thus $\pi$ is a local diffeomorphism, and a locally lifted spatial path
has no greater length than the original path.
A lift of a finite-length smooth path is Cauchy as the path approaches
its endpoint. Completeness gives a limit in $\Sigma$, and a local
inverse of $\pi$ extends the lift to that endpoint.
The lifting argument extends to piecewise smooth paths by applying it to each smooth segment.

We now show that these path lifts make $\pi$ a covering map. This uses completeness but does not require the immersion to be proper. Fix $y\in\R^n$ and a Euclidean ball $B_r(y)$. For each $a\in\pi^{-1}(y)$, lift the
radial segments starting at $y$ with initial point $a$. The path-lifting continuation gives an inverse branch of $\pi$ over
$B_r(y)$. Its regularity follows by
composing finitely many local inverses along each compact radial segment;
the same local inverses lift all sufficiently nearby segments.
Uniqueness of local lifts makes different sheets disjoint. Conversely, every point of $\pi^{-1}(B_r(y))$ belongs to one of these
sheets: lift the radial segment from its projection back to $y$.
The projection is surjective by lifting a segment from one fixed image
point to an arbitrary point of $\R^n$. Hence it is a covering.
Since $\R^n$ is simply connected and $\Sigma$ is connected, the covering
has one sheet.

The immersion is therefore an entire spacelike graph with nonnegative
height. \Cref{thm:liouville} implies that the height vanishes identically.
The image of the immersion is therefore the reference hyperplane.
\end{proof}

\section*{Acknowledgments}
The authors would like to thank Professor Xi-Nan Ma for his encouragement and suggestions on this subject. This work was supported by the Natural Science Foundation of Heilongjiang Province (Grant No. PL2025A006) and Program for Young Talents of Basic Research in University of Heilongjiang Province (Grant No. YQJH2025116). The authors acknowledge the use of AI tools. All mathematical statements and proofs were independently verified by the authors, who take full responsibility for the content of the manuscript.

\end{document}